\documentclass[10pt, oneside,reqno]{amsart}
\usepackage{amssymb}
\usepackage{amsmath}
\usepackage[normalem]{ulem}
\usepackage{multirow}
\usepackage{array}
\usepackage{graphicx}
\graphicspath{ {./images/} }
\usepackage{graphicx} 
\usepackage{mathrsfs}
\usepackage[textwidth=14mm]{todonotes}
\usepackage{cancel}
\usepackage{amsthm}
\usepackage{full page}
\usepackage{xcolor}
\usepackage{hyperref}
\allowdisplaybreaks
\usepackage{geometry,mathtools}

\newtheorem{thm}{Theorem}[section]
\newtheorem{prop}[thm]{Proposition}
\newtheorem{cor}[thm]{Corollary}
\newtheorem{lem}[thm]{Lemma}
\theoremstyle{definition}

\newtheorem{rem}[thm]{Remark}
\newtheorem{defn}[thm]{Definition}

\newcommand{\C}{\mathbb{C}}

\newcommand{\Z}{\mathbb{Z}}
\newcommand{\wt}{\mbox{\rm wt}\ }
\newcommand{\res}{\mbox{\rm Res}}
\newcommand\vir{\mathcal{L}}
\newcommand{\vac}{\mathbf{1}}

\reversemarginpar

\begin{document}

\title[Graded pseudo-traces for strongly interlocked modules]{
Graded pseudo-traces for strongly interlocked modules for a vertex operator algebra and applications 
}

\author{Katrina Barron}
\address{Department of Mathematics, University of Notre Dame, Notre Dame, IN 46556}
\email{kbarron@nd.edu}
\thanks{The first named author was supported by the Association for Women in Mathematics and the National Science Foundation  Travel Grant program and a Simons Foundation Travel Support Grant.  The second named author was supported by FONDECYT Project 3190144. The third named author was supported by the National Science Foundation under Grant No. DMS-2102786.  The last named author was supported by the College of Arts and Sciences at Illinois State University Research Grant.}
\author{Karina Batistelli}
\address{Department of Mathematics, Universidad de Chile, Santiago, Chile}
\email{kbatistelli@unc.edu.ar}
\author{Florencia Orosz Hunziker}
\address{Department of Mathematics, University of Colorado Boulder, Boulder, CO 80309.}
\email{florencia.orosz@colorado.edu}
\author{Gaywalee Yamskulna}\address{Department of Mathematics, Illinois State University, Normal, IL 61790} \email{gyamsku@ilstu.edu }


\date{\today}


\begin{abstract} We define the notion of {\it strongly interlocked} for indecomposable generalized modules for a vertex operator algebra, and show that the notion of graded pseudo-trace is well defined for modules which satisfy this property  in certain settings.  We prove that in these settings the graded pseudo-trace is a symmetric linear operator that satisfies the logarithmic derivative property.  As an application, 
we prove that all the indecomposable reducible generalized modules for the rank one Heisenberg (one free boson) vertex operator algebras are strongly interlocked,  independent of the choice of conformal vector and have well-defined graded pseudo-traces. We also completely characterize which indecomposable reducible generalized modules for the universal Virasoro vertex operator algebras induced from the level zero Zhu algebra are strongly interlocked. In particular, we prove that the universal Virasoro vertex operator algebra with central charge  $c$ has modules induced from the level zero Zhu algebra with conformal weight $h$ that are strongly interlocked if and only if either $(c,h)$ is outside the extended Kac table, or the central charge is either $c = 1$ or $25$, the conformal weight satisfies a certain property, and the level zero Zhu algebra module being induced is determined by a Jordan block of size less than a certain specified parameter. We prove that all these modules for the universal Virasoro vertex operator algebra that are strongly interlocked have well-defined graded pseudo-traces. We give several examples of graded pseudo-traces for these Heisenberg and Virasoro strongly interlocked modules. 

\end{abstract}

\maketitle

\section{Introduction}

Vertex operator algebras are the basic building blocks of conformal field theory, play a major role in the construction of modular tensor categories,  and have fundamental connections to number theory and the representation theory of Lie algebras, finite simple groups, and quantum groups \cite{FLM, Bo, FZ, FGST, CG, CMY, GN, kr}. In particular, in the seminal work of Zhu \cite{Z1, Z}, the space of graded traces (also called graded characters) for $\mathbb{Z}_{\geq 0}$-gradable modules of vertex operator algebra $V$, where $V$ satisfies certain nice properties--- including $C_2$-cofiniteness and rationality--- were shown to be modular invariant. In Zhu's setting, all modules under consideration are graded by eigenspaces of a certain operator $L_0$, these eigenspaces are finite-dimensional, and the rationality condition implies that all indecomposable modules are necessarily irreducible. 

Subsequently, Miyamoto \cite{Miyamoto2004}, studied how if one relaxes the rationality but retains the $C_2$-cofiniteness of the vertex operator algebra, then modular invariance can still hold if one includes not just graded traces, but graded pseudo-traces for  indecomposable reducible modules that are now graded by generalized eigenspaces of $L_0$ which are finite-dimensional (and thus called {\it generalized} modules) and that have a certain property called ``interlocked"  with respect to a certain map; see also \cite{Huang-recent}. Despite the fact that vertex operator algebras pertaining to the Miyamoto setting---irrational and $C_2$-cofinite---have been widely studied \cite{A, am1, AM2, AM3, AM4, CF, FFHST, FGST1, FGST2, FGST, NT}, very few concrete examples have been constructed. Thus to date, very few graded pseudo-traces have been computed or studied \cite{Miyamoto-RIMS, AN}. In addition, Miyamoto's notion of graded pseudo-trace relies on the structure of a symmetric  linear map with certain properties in relation to the higher level Zhu algebras of $V$, and these Zhu algebras are necessarily finite-dimensional in the $C_2$-cofinite setting. 

In this paper, we define the notion of ``strongly interlocked" for  generalized modules for any vertex operator algebra $V$, and show that the notion of graded pseudo-trace for such modules for certain settings is well defined.  In addition, we prove that graded pseudo-traces defined in this way are symmetric linear operators that satisfy the logarithmic derivative property. These are the key properties necessary to prove many other facts about the graded pseudo-traces as stated in \cite{Miyamoto2004}, but in our setting we have these properties and the resulting facts without the need for a symmetric linear map defined with respect to the higher level Zhu algebras for $V$.  

We apply the notions of strongly interlocked generalized modules and graded pseudo-traces for these modules to the setting of the two most prevalent vertex operator algebras, namely the Heisenberg and universal Virasoro vertex operator algebras. In particular, we prove that all  indecomposable reducible  generalized $V$-modules for $V$ the Heisenberg algebra are strongly interlocked  for any central charge, and have a well-defined notion of graded pseudo-trace. 

If $V = V_{Vir}(c,0)$ is the universal Virasoro vertex operator algebra with central charge $c$, we give a complete characterization of which indecomposable reducible  generalized $V$-modules  are strongly interlocked when the module is induced from the level zero Zhu algebra. 
 We show this classification of strongly interlocked modules is highly dependent both on the central charge of $V$, and the size of the level zero Zhu algebra module being induced.  
 This classification of strongly interlocked modules involves the development of new techniques for studying and classifying indecomposable $V_{Vir}(c,0)$-modules. We show that all the strongly interlocked modules for the universal Virasoro vertex operator algebra induced at level zero have well-defined graded pseudo-traces.  We  give several key examples of graded pseudo-traces for the strongly interlocked modules classified in this work. 

The Heisenberg and Virasoro vertex operator algebras are irrational and $C_1$-cofinite but not $C_2$-cofinite, and this paper offers the first systematic study of the extension of Zhu and Miyamoto's work to graded pseudo-traces of generalized $V$-modules beyond  the $C_2$-cofinite setting.

\subsection{Background}

In \cite{Z1, Z}, Zhu showed that if $V$ is rational (i.e., has semi-simple representation theory) and is $C_2$-cofinite (i.e., $\dim V/C_2 < \infty$ for $C_2 = \mathrm{span}_\mathbb{C} \{u_{-2} v \; | \; u, v \in V \}$) then, in particular, the graded dimensions of the simple $V$-modules, as functions in $\tau$ (for $q = e^{2 \pi i \tau}$) converge to holomorphic functions on the complex upper half plane, and the linear space spanned by these holomorphic functions is invariant under the action of $SL_2(\mathbb{Z})$. Here the action of $SL_2(\mathbb{Z})$ refers to the standard action on the complex upper half plane via linear fractional transformations, i.e., M\"obius transformations.  

In fact, Zhu showed much more, including results about more generalized graded traces than graded dimensions, and results about $n$-point correlation functions with $n >1$, i.e., graded traces of products of vertex operators for multiple elements in $V$. However, for the purposes of this introductory paper on the subject of graded pseudo-traces, we will focus on the generalizations by Miyamoto of Zhu's work on graded traces of single modes. More recently, in this $C_2$-cofinite irrational setting, graded pseudo-traces were studied by Huang \cite{Huang-recent} for intertwining operators. 

In 2004 \cite{Miyamoto2004, Miyamoto-Ukraine}, Miyamoto studied how the results of Zhu can hold for certain irrational $C_2$-cofinite vertex operator algebras if one expands the set of graded traces to include ``graded pseudo-traces". That is, if one no longer has rationality of the vertex operator algebra, then the linear space of graded traces is not generally closed under the action of $SL_2(\mathbb{Z})$, however the linear space of graded traces supplemented with the graded pseudo-traces for indecomposable reducible generalized modules is closed under this action.

Irrational vertex operator algebras, i.e., those which admit $L_0$-gradable modules that are indecomposable yet reducible, are the basic objects underlying logarithmic conformal field theory. Recently, logarithmic conformal field theory has come to the fore due to applications to disordered systems in physics as well as deep connections to number theory and the representation theory of quantum groups (cf. \cite{FGST, AM, GRR, CG, CMY}). Of particular interest in the irrational setting of logarithmic conformal field theory, is when $V$ is still $C_2$-cofinite, and this is the setting of Miyamoto's work on graded pseudo-traces. However, to date, there is only one family of examples of such  $C_2$-cofinite irrational vertex operator algebras that is starting to be well understood in terms of its tensor category structure and connections to the representation theory of quantum groups---the $\mathcal{W}(p)$ triplet vertex operator algebras \cite{NT, CMY}---and for which graded pseudo-traces  have been studied \cite{Miyamoto-RIMS,AN}. Unfortunately it is very difficult to construct $C_2$-cofinite irrational vertex operator algebras, and it is also very difficult to carry out the construction of Miyamoto's graded pseudo-traces. Motivated by both the need to understand representation categories for more general vertex algebras, and the need for simplifying the machinery previously required to realize well-defined graded pseudo-traces, in this work we begin to build the theoretical framework to compute graded pseudo-traces for  both a larger class of vertex operator algebras as well as in a setting that requires less machinery.

More precisely, in  \cite{Miyamoto2004} (see also the summary \cite{Miyamoto-Ukraine}, as well as \cite{Huang-recent}) the graded pseudo-traces, in general, rely on the notion of the higher level Zhu algebras for the vertex operator algebra.  Given a vertex operator algebra, $V$, the level $n$ Zhu algebras, for $n \in \mathbb{Z}_{\geq 0}$, form a family of associative algebras that depend only on the internal structure of $V$ but carry fundamental information about the representation theory of $V$. What is now called the level zero Zhu algebra was defined by Frenkel and Zhu \cite{FZ}, and used by Zhu to prove his results. The higher level Zhu algebras, for $n>0$, were defined by 
 Dong, Li, and Mason in \cite{DLM}. 
If $V$ is $C_2$-cofinite and rational, then Frenkel and Zhu proved there is a bijective correspondence between the irreducible modules for $V$ and the irreducible modules for the level zero Zhu algebra. However if $V$ is irrational,  the situation is much more complicated and in general, one needs the higher level Zhu algebras to detect certain indecomposable reducible modules.

Only recently have examples of higher level Zhu algebras been calculated (first by the first named author of the current paper, along with  Vander Werf, and Yang, \cite{BVY}--\cite{BVY-Virasoro}, later by \v{C}eperi\'c in talks given in 2019 presenting the level one Zhu algebra for symplectic fermions, and more recently by the first named author of the current paper, along with  Addabbo \cite{AB-Heisenberg}). Moreover, only recently has the correspondence between indecomposable modules for the vertex operator algebra versus indecomposable 
modules for higher level Zhu algebras started to be understood as studied in \cite{BVY}. However as yet, to our knowledge (and according to Miyamoto) no concrete examples of the graded pseudo-traces arising directly from the higher level 
Zhu algebras have been determined.

The Heisenberg and universal Virasoro vertex operator algebras are examples of irrational vertex operator algebras which are not $C_2$-cofinite, but satisfy another nice finiteness condition called $C_1$-cofiniteness.  These are also the two most prominent families of vertex operator algebras as every vertex operator algebra contains either a universal Virasoro vertex operator algebra as a subalgebra, or a quotient of such a subalgebra, and almost all known vertex operator algebras contain or arise from Heisenberg vertex operator subalgebras. 

To define the notion of graded pseudo-traces for an indecomposable reducible module, Miyamoto makes use of a certain symmetric linear map $\phi$ that is defined on an algebra arising from each level $n$ Zhu algebra of $V$ in the $C_2$-setting where the Zhu algebras have the structure of a Frobenius algebra.  Miyamoto then defines the notion of ``interlocked module with respect to this map $\phi$".  In the more general $C_1$-cofinite setting, one loses some of this machinery because the Zhu algebras are infinite dimensional and thus not Frobenius. However, as we show in this paper, this is not an intractable barrier to defining graded pseudo-traces if one introduces the notion of strongly interlocked modules, a notion that is independent of the structure of the Zhu algebras.

Although Miyamoto's and Huang's results extend the modularity results of Zhu from $C_2$-cofinite rational vertex operator algebras to the $C_2$-cofinite irrational setting, whether these modularity properties can be extended further, remains an open question. This paper lays the fundamental ground work for the beginning  of a systematic study of this question, for, for instance, $C_1$-cofinite vertex operator algebras, as well as even more general settings,  and in \cite{BOHY} we give further insights into modular-type properties of some of these graded pseudo-traces as studied here.

\subsection{Current results}

In this paper, we give a general definition of ``interlocked" for a generalized $V$-module that is independent of the Zhu algebras for $V$.  We then define a refinement of this notion, called ``strongly interlocked", which is a very natural condition.  We give two settings in which strongly interlocked generalized $V$-modules have a well-defined notion of graded pseudo-trace.  We show that in settings such as these where such graded pseudo-traces are well defined, that these graded pseudo-traces are symmetric linear operators that satisfy the logarithmic derivative property. 

Given the importance of the logarithmic derivative property in the proof of the modular invariance of characters obtained by Zhu in the rational and $C_2$-cofinite case \cite{Z} as well as in the proof of modular invariance when graded pseudo-traces are included as studied by Miyamoto in the $C_2$-cofinite and irrational case \cite{Miyamoto2004}, this indicates that the graded pseudo-traces defined here for strongly interlocked modules in certain settings give the correct notion of character in broader settings, such as the irrational $C_1$-cofinite setting in which there are infinitely many non isomorphic irreducible modules. This suggests that the category of strongly interlocked modules associated to  a vertex operator algebra, if closed under the tensor product, can give rise to tensor structures analogous to the modular tensor categories constructed in the rational $C_2$-cofinite setting by  Huang and Lepowsky \cite{HL1}-\cite{HL3}, \cite{H4}. The logarithmic tensor category theory necessary in the irrational setting was developed by Lepowsky, Huang and Zhang \cite{HLZ1}-\cite{HLZ8} and recently applied to construct braided tensor structures associated to universal Virasoro vertex operator algebras at all central charges by the third named author of this work, along with Creutzig, Jiang, Ridout and Yang in \cite{CJOHRY}.

After developing the notions of strongly interlocked modules and identifying two particular settings in which these will give rise to well-defined graded pseudo-traces, we  apply our results to the settings of the most prominent $C_1$-cofinite vertex operator algebras, namely the Heisenberg and universal Virasoro vertex operator algebras, both of which have level zero Zhu algebra  $\mathbb{C}[x]$, which for example fails to be a Frobenius algebra as it is not finite dimensional. For these vertex operator algebras, we classify the indecomposable modules induced from the level zero Zhu algebra  that are  interlocked.  We show that these interlocked Heisenberg and Virasoro modules, are strongly interlocked and that the notion of a graded pseudo-trace is well defined since these examples fall into our two settings where we have shown these strongly interlocked modules have well-defined graded pseudo-traces. We then calculate some of their key graded pseudo-traces. 

We prove that all indecomposable reducible modules for the Heisenberg vertex operator algebra for any choice of conformal vector are strongly interlocked, using the fact that all indecomposable modules are induced from the level zero Zhu algebra in this case, and thus falls into one of our general settings for having strongly interlocked modules with well-defined graded pseudo-traces.  We then compute certain graded pseudo-traces, including the ones associated to the vacuum and conformal vectors. 

For the universal Virasoro vertex operator algebras, $V_{Vir}(c,0)$ for $c \in\mathbb{C}$, as proved in \cite{BVY, BVY-Virasoro}, there are indecomposable modules that are not induced from the level zero Zhu algebra. In this work, we characterize the indecomposable $V_{Vir}(c,0)$-modules induced from the level zero Zhu algebra which are  interlocked. It is an interesting problem, although very difficult, to characterize  interlocked indecomposable $V_{Vir}(c,0)$-modules induced from higher level Zhu algebras. The techniques we develop in this paper for analyzing the $V_{Vir}(c,0)$-modules induced from the level zero Zhu  algebra can be used to analyze the modules induced at higher levels including when one does not know explicitly the higher level Zhu algebra.  

After classifying the generalized $V_{Vir}(c,0)$-modules induced from the level zero Zhu algebra that are interlocked, we show that they are strongly interlocked and fall into the two settings where we have shown such strongly interlocked modules have well-defined graded pseudo-traces. For $V_{Vir}(c,0)$ the results of when an indecomposable module is interlocked are much more intricate than for the Heisenberg vertex operator algebra. In particular, there is a dependency on the central charge $c$ of the vertex operator algebra, and the conformal weight $h$ of the module, as well as the Jordan block size of the module for the level zero Zhu algebra being induced. In our systematic classification of interlocked modules induced from the level zero Zhu algebra for $V_{Vir}(c,0)$, we develop and employ new techniques in our analysis of indecomposable modules. These new techniques allow us to classify which combinations of central charge $c$, conformal weight $h$, and Jordan block size $k$ for an indecomposable module for $V_{Vir}(c,0)$ induced from the level zero Zhu algebra, will result in an interlocked module.  In particular, we uncover subtle behavior for central charges $c = 1$ and $c = 25$. These particular central charges are known for exhibiting particularly interesting behavior in other settings, cf. \cite{Mil2, OH, MY}, and in this work we show that our new techniques uncover new interesting properties.

Via the results of this paper, including the techniques developed here, along with \cite{BVY, BVY-Virasoro} where the level one Zhu algebra is calculated for $V_{Vir}(c,0)$ it is now possible to induce indecomposable reducible $V_{Vir}(c,0)$-modules from this level one Zhu algebra and analyze when such modules are strongly interlocked and thus have well-defined graded pseudo-traces.

 It is also important to note that the results of this paper uncover the following two phenomena: (1) There are settings in which  a vertex operator algebra $V$ has strongly interlocked indecomposable reducible modules with well-defined graded pseudo-traces that vanish for a particular $v \in V$---namely when $V$ is the Heisenberg vertex operator algebra with a certain conformal vector  $\omega^a$, the module is of a certain conformal weight $\lambda = a$, and $v$ is the vacuum, as shown in  Corollary \ref{cor:vanish}; (2) There are settings in which some indecomposable reducible modules for a  fixed vertex operator algebra $V$ are strongly interlocked and have well-defined graded pseudo-traces  whereas other $V$-modules are not strongly interlocked or even interlocked, and have no known well-defined notion of graded pseudo trace. In particular, this latter setting  occurs for $V = V_{Vir}(c,0)$ with central charge $c = 1$ or $25$ as shown in Theorem \ref{interlocked-thm2}.

\subsection{Organization and Main Results} 
This paper is organized as follows:

In Section 2, we start by giving the necessary preliminary definitions, including various notions of module for a vertex operator algebra $V$, the notion of the Zhu algebras for $V$, and the induction functor from modules for the Zhu algebra at any level to a $\mathbb{Z}_{\geq 0}$-graded module for $V$. In Section 2.2, we recall the notion of graded trace, and then in Section 2.3, we briefly recall the notion of graded pseudo-traces in the $C_2$-cofinite setting as defined by Miyamoto in terms of interlocked $V$-modules when one has a certain symmetric linear map $\phi$ on the higher level Zhu algebras. Here we define the notions of {\it weakly interlocked} and {\it interlocked.}  In Section 2.3, we also recall the logarithmic derivative property for graded pseudo-traces as defined by Miyamoto.  

In Section 3, we introduce the  notion of {\it strongly interlocked}  for any  generalized $V$-module,  and define the notion of graded pseudo-trace  in the setting in which there exists a {\it strongly interlocked family of bases} for which the pseudo-trace is invariant with respect to a change of such a basis. We then prove three of the main results of this paper,  Theorems \ref{extra-conditions-theorem}, \ref{Virasoro-exceptional-theorem}, and \ref{log-thm}  which we summarize here:

\medskip

{\bf Setting 1 giving well-defined graded pseudo-traces:} {\it If $V$ has a single generator, level zero Zhu algebra isomorphic to $\mathbb{C}[x]$, and a certain nondegenerate bilinear form, then a $V$-module induced from an indecomposable $\mathbb{C}[x]$-module is strongly interlocked and has well-defined graded pseudo-traces. }

\medskip

{\bf Setting 2 giving well-defined graded pseudo-traces:} {\it If $V$ has a unique irreducible module of weight $\lambda \in \mathbb{C}$ with a one-dimmensional lowest weight space, and $W$ is a strongly interlocked $V$-module with conformal weight $\lambda \in \mathbb{C}$, then $W$ has well-defined graded pseudo-traces.}

\medskip

{\bf Properties of graded pseudo-traces for strongly interlocked modules:} {\it  If $W$ is a strongly interlocked generalized $V$-module, for a vertex operator algebra $V$, with well-defined  graded pseudo-traces, then these graded pseudo-traces are   symmetric, linear, and satisfy the logarithmic derivative property.}  

\medskip

In Section 4, we apply our results to $V = M_a(1)$ the rank one Heisenberg vertex operator algebra for any choice of conformal vector $\omega^a$. First, in Section 4.1, we recall the 
 definition of $M_a(1)$, and in Section 4.2, we recall results on  generalized $M_a(1)$-modules.  In Section 4.3, we recall more details about these generalized $M_a(1)$-modules and their graded dimensions. In Section 4.4, we prove that $M_a(1)$ falls under Setting 1 of Theorem \ref{extra-conditions-theorem} and give Corollary \ref{extra-conditions-cor} to Theorem \ref{extra-conditions-theorem}:

 \medskip

{\bf Classification of strongly interlocked modules for the Heisenberg vertex operator algebra and existence of well-defined graded pseudo-traces}: {\it All indecomposable generalized modules for the Heisenberg vertex operator algebra are strongly interlocked and have well-defined graded pseudo-traces}.

\medskip

In Section 4.5, we calculate the vacuum graded pseudo-trace for all indecomposable  generalized $M_a(1)$-modules. We then also calculate the graded pseudo-traces for the Heisenberg generator $\alpha_{-1} \vac$ and note that the graded pseudo-trace for the conformal element can be derived using the logarithmic derivative property proved in Theorem \ref{log-thm}.

In Section 5, we recall the relevant facts about the universal Virasoro vertex operator algebras $V_{Vir}(c,0)$ for $c \in \mathbb{C}$, and the Verma modules $M(c,h)$ for $c,h \in \mathbb{C}$. 
 In Section 5.1, we recall the definitions of $V_{Vir}(c,0)$ and $M(c,h)$. We recall the results of Feigin and Fuchs that $M(c,h)$ is simple if and only if $(c,h) \notin \Phi_{r,s}(c,h)$, where the $\Phi_{r,s}(c,h)$ are certain curves in the $\mathbb{C}^2$ plane defined by the parameters $r,s \in \mathbb{Z}_{>0}$, and that the maximal proper submodule $T(c,h)$ of $M(c,h)$ is generated by zero, one, or two singular vectors. Such singular vectors are parameterized by $r,s$ and $t \in \mathbb{C}^\times$ and denoted $S_{r,s}(t)$. We also recall some details about these singular vectors. In Section 5.1, we  also introduce notation for the purposes of the organization of this paper, namely the three cases:
 
Case (0):  $T(c,h) = 0$.

Case (1):  $T(c,h)$ is generated by one singular vector $S_{r,s}(t)$.

Case (2): $T(c,h)$ is generated by two singular vectors $S_{r,s}(t)$ and $S_{r',s'}(t)$.

In the literature on the modules $M(c,h)$ and $L(c,h) = M(c,h)/T(c,h)$ there are various cases denoted, and these almost always include a refined and/or different list of more nuanced cases and subcases than these Cases (0)--(2) we give above, and in Section 5.1. However, as our results show, for the purposes of this paper, we will only need to give these three main cases, and two subcases for Case (1), to give our classification of interlocked modules for $V_{Vir}(c,0)$.

In Section 5.2, we recall the Shapovalov form and facts about the determinant of the Gram matrix $\mathcal{A}_{\ell}(c,h)$ of this form in terms of the $\Phi_{r,s}(c,h)$. Our results will rely heavily on the matrix $\mathcal{A}_{\ell}(c,h)$, its determinant and its partial derivatives with respect to the complex variable $h$.

In Section 6, we give the most technical results of this paper as we give the analysis of the indecomposable $V_{Vir}(c,0)$-modules induced from the level zero Zhu algebra that is necessary and sufficient to characterize which of these modules are strongly interlocked. In Section 6.1, we show that the indecomposable $V_{Vir}(c,0)$-modules induced from a module of Jordan block size $k$ and conformal weight $h$ for the level zero Zhu algebra are given by $W(c,h,k) = M(c,h,k)/J(c,h,k)$ where $M(c,h,k)$ is a certain universal module, and $J(c,h,k) = \coprod_{\ell \in \mathbb{Z}_{\geq 0}} J(c,h,k)(\ell)$ is characterized as the coproduct of the kernels of a certain family of matrices $\{ \mathfrak{A}_\ell^{(k)}\}_{\ell \in \mathbb{Z}_{\geq 0}}$.  

In Section 6.2, we show that $\det (\mathfrak{A}_\ell^{(k)}(c,h)) = (\det \mathcal{A}_\ell(c,h))^k$ for $\mathcal{A}_\ell(c,h)$ the Gram matrix of the Shapavolov form at degree $\ell \in \mathbb{Z}_{\geq 0}$.  In Section 6.3, we give some examples of the matrices $\mathfrak{A}_\ell^{(k)}(c,h)$. 

In Section 6.4, we further analyze the dependency of $\mathfrak{A}_\ell^{(k)}$ on $\mathcal{A}_\ell(c,h)$ and its partial derivatives with respect to $h$ viewed as a formal variable. In Section 6.5, we give a family of linear equations depending on the partial derivatives of $\mathcal{A}_\ell(c,h)$ with respect to $h$ that must be satisfied in order for $v \in M(c,h,k)(\ell)$ to be in $J(c,h,k)(\ell) = Ker \mathfrak{A}_\ell^{(k)}$.  We then use the expression of $\mathcal{A}_\ell(c,h)$ in terms of the $\Phi_{r,s}(c,h)$, and Jacobi's Formula to determine $J(c,h,k)(\ell) = Ker \, \mathfrak{A}_\ell^{(k)}$ for degrees $\ell \leq d = rs$ where $d$ is the lowest degree of a singular vector $S_{r,s}(t)$ in $T(c,h)$, if $T(c,h) \neq 0$. In particular in Theorem 6.16, we prove that $J(c,h,k)(d)$ for $d = rs$ is one dimensional for $k>1$ if and only if one of the following holds: Either $c\notin \{1, 25\}$, or $c \in \{1,25\}$ but $r \neq s$ for the singular vector $S_{r,s}(t)$ generating $T(c,h)$. 

If $c = 1$ or $25$, and  $T(c,h)\neq 0$, then $h = h_{r,s}(\pm1)$ and $T(c,h) = \langle S_{r,s}(\pm1)\rangle$, i.e., $T(c,h)$  is generated by one singular vector and thus this case falls under Case (1).  If in addition $r \neq s$, we call this Case (1)(ii). The remaining cases for when $T(c,h)$ is generated by one singular vector but does not fall under Case (1)(ii), we call Case(1)(i). 

In Theorem 6.16, for the case when $\dim J(c,h,k)(d) >1$, i.e., in this Case (1)(ii) which is when $c = c(\pm 1)$ with $c(1) = 25$ and $c(-1) = 1$, $k>1$, and $h = h_{r,s}(\pm1)$ for $r \neq s$, we define the parameter $\kappa_{r,s}^\pm$ which gives the dimension of $J(c,h,k)(d)$ in relation to $k$, and is determined by how many of the linear equations depending on the partial derivatives of $A_{\ell}(c,h)$ with respect $h$ given in Section 6.5 have a solution when $h$ is set to $h = h_{r,s}(\pm1)$. 

In the rest of Section 6.5, we give some more analysis of the subspace $J(c,h,k)\subset M(c,h,k)$, in particular for Cases (1) and (2). In Theorem 6.22, for Case (1) (ii),  i.e., the case when $c = 1$ or $25$ and $r\neq s$, we determine $J(c,h,k)$ completely in terms of $\kappa_{r,s}^\pm$. 

In Section 6.6, we give two examples in Case (1), i.e., when $T(c,h)$ is generated by one singular vector. These cases illustrate some of the more interesting and nuanced behavior proved in Section 6.5, for instance at $c = -2$ versus $c \neq -2, \pm 1$, and in Case (1)(ii) an example of $\kappa_{r,s}^\pm$ for the case when $(r,s) = (2,1)$.

In Section 7, we give another one of our main results of this paper for $V = V_{Vir}(c,0)$:  The classification of all strongly interlocked indecomposable  $V_{Vir}(c,0)$-modules induced from the level zero Zhu algebra.  In particular, in Theorem \ref{interlocked-thm2}, we show the following:

\medskip

{\bf Classification Theorem for interlocked  indecomposable reducible $V_{Vir}(c,0)$-modules induced from the level zero Zhu algebra:} {\it  For $(c,h) \in \mathbb{C}^2$ and $k \in \mathbb{Z}_{>0}$, with $k\geq 2$, let $W(c,h,k) = \mathfrak{L}_0(U(c,h,k))$ where $U(c,h,k) = \mathbb{C}[x]/((x - h)^k)$ as an indecomposable  reducible module for the level zero Zhu algebra for $V_{Vir}(c,0)$, and $\mathfrak{L}_0$ is the induction functor from $\mathbb{C}[x]$-modules to $V_{Vir}(c,0)$-modules.  The $V_{Vir}(c,0)$-module $W(c,h,k)$ is interlocked if and only if one of the following holds:

\medskip

Case (0) holds, i.e., $T(c,h) = 0$;

\medskip

Case (1)(ii) holds, (i.e., $t \ \pm1$ and $r \neq s$ for $(c,h) = (c(\pm1), h_{r,s}(\pm1))$), and $2 \leq k \leq  \kappa^\pm_{r,s}$, where $\kappa^\pm_{r,s}$ is defined as in Theorem \ref{degree-d-theorem}. Note that $c(1) = 25$ and $c(-1) = 1$.

\medskip

Moreover, in these cases when $W(c,h,k)$ is interlocked, then it is strongly interlocked.
}  

\medskip

We also note in Remark \ref{M-bar-remark}
that in Cases (1) and (2), the universal $V_{Vir}(c,0)$-module induced from the $A_0(V_{Vir}(c,0))$-module $U(c,h,k)$, denoted $\overline{M}_0(U(c,h,k))$, is not interlocked, as opposed to in Case (0) when $W(c,h,k) = \overline{M}_0(U(c,h,k))$ and, as noted in the Theorem given above, is interlocked which follows directly from Theorem \ref{extra-conditions-theorem} as Case (0) falls under Setting 1; see also Corollary \ref{Case-0-cor}.    

In Section 8, we begin by recalling the vacuum graded traces, i.e., graded dimensions, for the Verma modules $M(c,h)$ for Cases (0) and (1) which include all the cases for when the indecomposable reducible $V_{Vir}(c,0)$-modules are interlocked, as there are no such modules in Case (2). 

In Section 8.2, we calculate the graded pseudo-traces for the vacuum and $\omega$ for Case (0).

In Section 8.3, we note that in Case (1)(ii), the strongly interlocked modules fall under Setting 2 which is the setting of Theorem \ref{Virasoro-exceptional-theorem} and thus these modules have well-defined graded pseudo-traces.  We calculate the graded pseudo-traces for the vacuum and conformal element for these modules. 

In Section 9, we summarize our results and discuss future directions of this work.

\bigskip

\subsection{Acknowledgements} The authors thank Ana Ros Camacho and Nezhla Agahee for founding the Women in Mathematical Physics (WoMaP) research and mentoring program from which the research presented in this paper arose. 
 We also thank Banff International Research Station for their hospitality for the virtual WoMaP I, 2020 Workshop, and for their generous support for the in-person WoMaP II, 2023 Workshop.  We thank Darlayne Addabbo, Kiyokazu Nagatomo, Veronika Pedi\'c Tomi\'c, David Ridout, and Kyle Luh for helpful discussions.  

 The first named author was supported by the Association for Women in Mathematics and the National Science Foundation  Travel Grant program  and a Simons Foundation Travel Support Grant. The second named author was supported by FONDECYT Project 3190144. The third named author was supported by the National Science Foundation under Grant No. DMS-2102786 and by the hospitality of the University of Denver. The last named author was supported by the College of Arts and Sciences at Illinois State University Research Grant.

We thank Matthew Headrick for making his Mathematica package Virasoro.nb available, which we found useful for testing some of our work and ideas. Matthew Headrick's code can be downloaded from his personal webpage  \url{https://people.brandeis.edu/~headrick/Mathematica/}.  

Finally, we thank the anonymous referee for very thorough reports catching several errors and typos and making thoughtful suggestions that have helped to correct and clarify many aspects of the manuscript.

\section{Preliminaries} 

In this Section we recall the various module structures associated to a vertex operator algebra $V$, important general results about the Zhu algebra of $V$ and its relationship to the representations of $V$, as well as the definition of graded traces, and graded pseudo-traces as defined by Miyamoto for $V$-modules ``interlocked with $\phi$".

We refer the reader to \cite{FLM, LL, FHL} for the notions of vertex operator algebra and weak $V$-module for a vertex operator algebra $V$.

\begin{defn}\label{N-gradable-definition}
${}$

(i) A {\it $\mathbb{Z}_{\geq 0}$-gradable weak $V$-module} (also often called an {\it admissible $V$-module} as in \cite{DLM}) $W$ for a vertex operator algebra $V$ is a weak $V$-module that is $\mathbb{Z}_{\geq 0}$-gradable, $W = \coprod_{\ell \in \mathbb{Z}_{\geq 0}} W(\ell)$, with $v_m W(\ell) \subset W(\ell + \mathrm{wt}\, v - m -1)$ for homogeneous $v \in V$, $m \in \mathbb{Z}$ and $\ell \in \mathbb{Z}_{\geq 0}$, and without loss of generality, we can and do assume $W(0) \neq 0$, unless otherwise specified.  We say elements of $W(\ell)$ have {\it degree} $\ell \in \mathbb{Z}_{\geq 0}$. Here $\wt v = n$ if $v \in V_n$ where $V  = \coprod_{n \in \mathbb{Z}} V_n$ gives the decomposition of the vertex operator algebra $V$ into $L_0$ eigenspaces, for $L_0 = \omega_1 = o(\omega)$ for $\omega \in V_2$ the conformal vector. 

(ii) A {\it $\mathbb{Z}_{\geq 0}$-gradable generalized weak $V$-module} $W$ is a $\mathbb{Z}_{\geq 0}$-gradable weak $V$-module that admits a decomposition into generalized eigenspaces via the spectrum of $L_0 = \omega_1$ as follows: $W=\coprod_{\lambda \in{\C}}W_\lambda$ where $W_{\lambda}=\{w\in W \, | \, (L_0 - \lambda \, Id_W)^j w= 0 \ \mbox{for some $j \in \mathbb{Z}_{>0}$}\}$, and in addition, $W_{n +\lambda}=0$ for fixed $\lambda$ and for all sufficiently small integers $n$. We say elements of $W_\lambda$ have {\it weight} $\lambda \in Spec L_0 \subset \mathbb{C}$, denoted $\wt w = \lambda$ if $w \in$ {$W_{\lambda}$}.

(iii) A {\it generalized $V$-module} $W$ is a $\mathbb{Z}_{\geq 0}$-gradable generalized weak $V$-module where $\dim W_{\lambda}$ is finite for each $\lambda \in \mathbb{C}$.   

(iv) An {\it (ordinary) $V$-module} is a  generalized $V$-module such that  the generalized eigenspaces $W_{\lambda}$ are in fact eigenspaces, i.e., $W_{\lambda}=\{w\in W \, | \, L_0 w=\lambda w\}$.
\end{defn}

Note that we will often omit the term ``weak" when referring to $\mathbb{Z}_{\geq 0}$-gradable weak and $\mathbb{Z}_{\geq 0}$-gradable generalized weak $V$-modules.

\begin{rem} 
     The term {\it logarithmic} is also often used in the literature to refer to $\mathbb{Z}_{\geq 0}$-gradable  generalized weak modules  or generalized modules. In addition, we note that a $\mathbb{Z}_{\geq 0}$-gradable $V$-module with $W(\ell)$ of finite dimension for each $\ell \in \mathbb{Z}_{\geq 0}$ is not necessarily a generalized $V$-module since the generalized eigenspaces might not be finite dimensional. 
     We also note that our notion of a generalized $V$-module is sometimes referred to as a {\it lower-bounded generalized $V$-module} in the literature.
\end{rem}

\subsection{Zhu algebras and induced modules}

In this section, we recall the definition and some properties of the algebras $A_n(V)$ for $n \in \mathbb{Z}_{\geq 0}$, first introduced in \cite{Z} for $n = 0$, and then generalized to $n >0$ in \cite{DLM}.   We then recall the functors $\Omega_n$ and $\mathfrak{L}_n$ defined in \cite{DLM}, and we recall some results from \cite{BVY}.

For $n \in \mathbb{Z}_{\geq 0} $, let $O_n(V)$ be the subspace of $V$ spanned by elements of the form
\begin{equation}\label{elements-in-O}
 u \circ_n v =
\res_x \frac{(1 + x)^{\mathrm{wt}\, u + n}Y(u, x)v}{x^{2n+2}}
\end{equation}
for all homogeneous $u \in V$ and for all $v \in V$, and by elements of the form $(L_{-1} + L_0)v$ for all $v \in V$, where $L_{-1} = \omega_0$. The vector space $A_n(V)$ is defined to be the quotient space $V/O_n(V)$.

\begin{rem} 
As noted in \cite{AB-generaln}, for $n=0$, since $v \circ_0 \mathbf{1} = v_{-2} \mathbf{1} + (\mathrm{wt} \, v) v = L_{-1} v + L_0 v$, it follows that $O_0(V)$ is spanned by elements of the form (\ref{elements-in-O}).  But this is not necessarily true of $O_n(V)$ for $n>0$.   
\end{rem}

We define the following multiplication on $V$
\[
u *_n v = \sum_{m=0}^n(-1)^m\binom{m+n}{n}\res_x \frac{(1 + x)^{\mathrm{wt}\, u + n}Y(u, x)v}{x^{n+m+1}},
\]
for $v \in V$ and homogeneous $u \in V$, and for general $u \in V$, $*_n$ is
defined by linearity.   It is shown in \cite{DLM} that with this multiplication, the subspace $O_n(V)$ of $V$ is a two-sided ideal of $V$, and $A_n(V)$ is an associative algebra, called the {\it level $n$ Zhu algebra}.

For every homogeneous element $u \in V$
and $m \geq k \geq 0$, elements of the form
\begin{equation*}
\res_x \frac{(1 + x)^{\mathrm{wt}\, u + n+ k}Y(u, x)v}{x^{m+2n+2}}
\end{equation*}
lie in $O_n(V)$.  This fact follows from the $L_{-1}$-derivative property for $V$.  This implies that $O_n(V) \subset O_{n-1}(V)$.  In fact, from Proposition 2.4 in \cite{DLM}, we have that the map 
\begin{eqnarray*}
A_n(V) & \longrightarrow & A_{n-1}(V) \\
v + O_n(V) & \mapsto & v + O_{n-1}(V) \nonumber
\end{eqnarray*}
is a surjective algebra homomorphism.

From Lemma 2.1 in \cite{DLM}, we have that 
\begin{equation*}
u *_n v - v *_n u - \mathrm{Res}_x (1 + x)^{ \mathrm{wt} \, u - 1} Y(u, x)v \in O_n(V),
\end{equation*}
and from Theorem 2.3 in \cite{DLM}, we have that $\omega + O_n(V)$ is a central element of $A_n(V)$.  

Next, we recall the functors $\Omega_n$ and $\mathfrak{L}_n$, for $n \in \mathbb{Z}_{\geq 0}$, defined and studied in \cite{DLM}.  Let $W$ be a $\mathbb{Z}_{\geq 0}$-gradable $V$-module, and let
\begin{equation}\label{Omega-def}
\Omega_n(W) = \{w \in W \; | \; v_iw = 0\;\mbox{if}\; \wt v_i < -n \; 
\mbox{for $v\in V$ of homogeneous weight}\}.
\end{equation}
It was shown in \cite{DLM} that $\Omega_n(W)$ is an $A_n(V)$-module
via the action $o(v+O_n(V)) = v_{\mathrm{wt} \, v -1}$ for $v \in V$.   In particular, this action satisfies $o(u *_n v) = o(u)o(v)$ for $u,v \in A_n(V)$.

Furthermore, it was shown in \cite{DLM} and \cite{BVY} that there is a bijection between the isomorphism classes of irreducible $A_n(V)$-modules which cannot factor through $A_{n-1}(V)$ and the isomorphism classes of irreducible $\mathbb{Z}_{\geq 0}$-gradable $V$-modules with nonzero degree $n$ component.  

In order to define the functor $\mathfrak{L}_n$ from the category of $A_n(V)$-modules to the category of $\mathbb{Z}_{\geq 0}$-gradable $V$-modules, we need several notions, including the notion of the universal enveloping algebra of $V$, which we now define.  

Let
\begin{equation*}
\hat{V} = \C[t, t^{-1}]\otimes V/D\C[t, t^{-1}]\otimes V,
\end{equation*}
where $D = \frac{d}{dt}\otimes 1 + 1 \otimes L_{-1}$. For $v \in V$, let $v(m) = v \otimes t^m +  D\C[t, t^{-1}]\otimes V \in \hat{V}$.  Then $\hat{V}$ can be given the structure of a $\mathbb{Z}$-graded Lie algebra as follows:  Define the degree of $v(m)$ to be $ deg(v(m))=\wt v - m - 1$ for homogeneous $v \in V$, and define the Lie bracket on $\hat{V}$ by
\begin{equation*}
[u(j), v(k)] = \sum_{i \in\mathbb{Z}_{\geq 0}}\binom{j}{i}(u_iv)(j+k-i),
\end{equation*}
for $u, v \in V$, $j,k \in \mathbb{Z}$.
Denote the homogeneous subspace of degree $m$ by $\hat{V}(m)$. In particular, the degree $0$ space of $\hat{V}$, denoted by $\hat{V}(0)$, is a Lie subalgebra.

Denote by $\mathcal{U}(\hat{V})$ the universal enveloping algebra of the Lie algebra $\hat{V}$.  Then $\mathcal{U}(\hat{V})$ has a natural $\mathbb{Z}$-grading induced from $\hat{V}$, and we denote by $\mathcal{U}(\hat{V})_\ell$ the degree $\ell$ space with respect to this grading, for $\ell \in \mathbb{Z}$.

We can regard $A_n(V)$ as a Lie algebra via the bracket $[u,v] = u *_n v - v *_n u$, and then the map $v( \mathrm{wt} \, v -1) \mapsto v + O_n(V)$ is a well-defined Lie algebra epimorphism from $\hat{V}(0)$ onto $A_n(V)$.

Let $U$ be an $A_n(V)$-module.  Since $A_n(V)$ is naturally a Lie algebra homomorphic image of $\hat{V}(0)$, we can lift $U$ to a module for the Lie algebra $\hat{V}(0)$, and then to a module for $P_n = \bigoplus_{p > n}\hat{V}(-p) \oplus \hat{V}(0) = \bigoplus_{p < -n} \hat{V}(p) \oplus \hat{V}(0)$ by letting $\hat{V}(-p)$ act trivially for $p\neq 0$.  Define
\[
M_n(U) = \mbox{Ind}_{P_n}^{\hat{V}}(U) = \mathcal{U}(\hat{V})\otimes_{\mathcal{U}(P_n)}U.
\]

We impose a grading on $M_n(U)$ by letting $U$ be degree $n$, and letting $M_n(U)(i)$ be the $\mathbb{Z}$-graded subspace of $M_n(U)$ induced from $\hat{V}$, i.e., $M_n(U)(i) = \mathcal{U}(\hat{V})_{i-n}U$. 

For $v \in V$, define $Y_{M_n(U)}(v,x) \in (\mathrm{End} (M_n(U)))((x))$ by
\begin{equation*}
Y_{M_n(U)}(v,x) = \sum_{m\in\mathbb{Z}} v(m) x^{-m-1}.
\end{equation*} 

Let $W_{A}$ be the subspace of $M_n(U)$ spanned linearly by the coefficients of 
\begin{multline}\label{relations-for-M}
(x_0 + x_2)^{\mathrm{wt} \, v + n} Y_{M_n(U)}(v, x_0 + x_2) Y_{M_n(U)}(w, x_2) u \\ 
- (x_2 + x_0)^{\mathrm{wt} \, v + n} Y_{M_n(U)}(Y(v, x_0)w, x_2) u
\end{multline}
for $v,w \in V$, with $v$ homogeneous, and $u \in U$.  Set
\begin{equation}\label{define-M-bar} 
\overline{M}_n(U) = M_n(U)/\mathcal{U} (\hat{V})W_A . 
\end{equation}

It is shown in \cite{DLM} that if $U$ is an $A_n(V)$-module that does not factor through $A_{n-1}(V)$, then $\overline{M}_n(U) = \coprod_{\ell \in \mathbb{Z}_{\geq 0}} \overline{M}_n(U) (\ell)$ is a $\mathbb{Z}_{\geq 0}$-gradable $V$-module satisfying $\overline{M}_n(U) (0)\neq 0$, and  as an $A_n(V)$-module,  $\overline{M}_n(U) (n) \cong U$. Note that the condition that $U$ itself does not factor though $A_{n-1}(V)$ is indeed a necessary and sufficient condition for $\overline{M}_n(U) (0)\neq 0$ to hold.  

It is also observed in \cite{DLM} that $\overline{M}_n(U)$ satisfies the following universal property:  For any weak $V$-module $M$ and any $A_n(V)$-module homomorphism $\phi: U \longrightarrow \Omega_n(M)$, there exists a unique weak $V$-module homomorphism $\Phi: \overline{M}_n(U) \longrightarrow M$, such that $\Phi \circ \iota = \phi$ where $\iota$ is the natural injection of $U$ into $\overline{M}_n(U)$. This follows from the fact that $\overline{M}_n(U)$ is generated by $U$ as a weak $V$-module, again with the possible need of a grading shift.

Let $U^* = \mbox{Hom}(U, \C)$.  As in the construction in \cite{DLM}, we can extend the action of $U^*$ to $M_n(U)$ by first an induction to $M_n(U)(n)$ and then by letting $U^*$ annihilate $\coprod_{\ell \neq n} M_n(U)(\ell)$.  In particular, we have that elements of $M_n(U)(n) = \mathcal{U}(\hat{V})_0U$ are spanned by elements of the form 
\[o_{p_1}(a_1) \cdots o_{p_s}(a_s)U\]
where $s \in \mathbb{N}$, $p_1 \geq \cdots \geq p_s$, $p_1 + \cdots + p_s =0$, $p_i \neq 0$, $p_s \geq -n$, $a_i \in V$ and $o_{p_i}(a_i) = (a_i)(\mathrm{wt} \, a_i - 1 - p_i)$. Then inducting on $s$ by using Remark 3.3 in \cite{DLM} to reduce from length $s$ vectors to length $s-1$ vectors, we have a well-defined action of $U^*$ on $M_n(U)(n)$.  

Define $J_n(U)$ (denoted just $J$ if the context is clear) to be
\begin{equation}\label{J-def}
J_n(U) = \{v \in M_n(U) \, | \, \langle u', xv\rangle = 0 \;\mbox{for all}\; u' \in U^{*}, x \in \mathcal{U}(\hat{V})\}
\end{equation}
and set
\begin{align} \label{defLnfunctor}
\mathfrak{L}_n(U) = M_n(U)/ J_n(U).
\end{align}

\begin{rem}\label{L-a-V-module-remark} 
It is shown in \cite{DLM}, Propositions 4.3, 4.6 and 4.7,  that if $U$ does not factor through $A_{n-1}(V)$, for $n \in \mathbb{Z}_{>0}$, then $\mathfrak{L}_n(U)$ is a well-defined $\mathbb{Z}_{\geq 0}$-gradable $V$-module with $\mathfrak{L}_n(U)(0) \neq 0$. In particular, it is shown that $\mathcal{U}(\hat{V})W_A \subset J_n (U)$, for $W_A$ the subspace of $M_n(U)$ spanned by the coefficients of (\ref{relations-for-M}), i.e., moding by $J_n(U) $ in Eq 
 (\ref{defLnfunctor}) gives the associativity relations for the weak vertex operators on $M_n(U)$.
\end{rem}

 We have the following theorem from \cite{BVY}.

\begin{thm}\label{mainthm}\cite{BVY}
For $n \in \mathbb{Z}_{\geq 0}$, let $U$ be a nonzero $A_n(V)$-module such that if $n>0$, then $U$ does not factor through $A_{n-1}(V)$. Then $\mathfrak{L}_n(U)$ is a $\mathbb{Z}_{\geq 0}$-gradable $V$-module with $\mathfrak{L}_n(U)(0) \neq 0$.  If we assume further that there is no nonzero submodule of $U$ that factors through $A_{n-1}(V)$, then $\Omega_n/\Omega_{n-1}(\mathfrak{L}_n(U)) \cong U$.
\end{thm}

One of the main reasons we are interested in Theorem \ref{mainthm} is what it implies for the question of when modules for the higher level Zhu algebras give rise to indecomposable non simple modules for $V$ not seen by the lower level Zhu algebras.  
For instance, all indecomposable modules for the Heisenberg vertex operator algebra are induced from the level zero Zhu algebra, whereas this is not the case for the universal Virasoro vertex operator algebra.  Although all irreducible modules for the Virasoro vertex operator algebra arise from inducing an irreducible module for the level zero Zhu algebra, there are indecomposable reducible modules that are not induced from level zero; see \cite{BVY}-\cite{BVY-Virasoro}.

\begin{defn}\cite{Z} Let $V$ be a vertex operator algebra and let $C_2(V)$ denote the complex vector space spanned by the set $\{u_{-2}v\ | \ u,v \in V \}$. If the quotient space $V/C_2(V)$ is finite dimensional, we say that $V$ is {\it $C_2$-cofinite}.
\end{defn}

As proved by Zhu \cite{Z}, a $C_2$-cofinite vertex operator algebra $V$ admits only finitely many irreducible $\mathbb{Z}_{\geq 0}$-graded modules. There are, however several important examples of vertex operator algebras, including the Heisenberg and the universal Virasoro vertex operator algebras, that are not $C_2$-cofinite. We will show that they satisfy the following less restrictive cofiniteness condition introduced by Li in \cite{L}.

\begin{defn} \label{def:C1}
Let $V$ be a vertex operator algebra. Define $V^+:= \coprod_{n>0} V_n$ and let $C_1(V)$ be the complex vector space spanned by the set $$\{u_{-1}v\ | \ u,v\in V^+\}\cup\{L_{-1}v \ |\ v\in V\}.$$
The vertex operator algebra $V$ is said to be {\it $C_1$-cofinite} if the dimension of the quotient space $V/C_1(V)$ is finite.
\end{defn}
\begin{rem}
As explained in \cite{L}, for a vertex operator algebra $V$, the space $C_1(V)$ must be defined in a different way than $C_2(V)$ since for any element $v\in V$ one has $v=v_{-1}\vac$. (See also Remark 2.4 in \cite{H}.)
\end{rem} 

\subsection{Graded traces}

First we recall some facts about graded traces before we generalize them to graded pseudo-traces.

Let $V$ be a vertex operator algebra with $V = \coprod_{m \in \mathbb{Z}} V_m$ the grading of $V$ with respect to eigenspaces of the $L_0$ operator for $V$. Then the {\it zero mode} of a homogeneous element $v \in V_m$ is given by $o(v) = v_{m-1} = \res_z z^{m-1} Y(v,z) = \res_z z^{m-1} \sum_{n \in \mathbb{Z} } v_n z^{-n-1}$.

The {\it graded traces}  of $V$ (also called {\it one-point correlation functions} or {\it characters}) are defined to be 
\[ Z_V(v,q) = q^{-c/24} \sum_{n \in \mathbb{Z}} \left(\mathrm{tr}|_{V_n} o(v) \right) q^n,\]
where $q$ is a formal parameter that can be evaluated to be $q = e^{2\pi i \tau}$ for $\tau$ in the complex upper half plane, which we denote by $\mathbb{H}$,  with the goal of studying modular invariance properties for these series. To highlight this number-theoretic interpretation of the graded traces, it is common to write them as a function of $\tau\in \mathbb{H}$, as we will do in this work.

More generally, if $W = \coprod_{\lambda \in \mathbb{C}} W_\lambda$ is a generalized $V$-module, then we have the graded traces corresponding to $W$ given by 
\[ Z_W(v,\tau) = q^{-c/24} \sum_{\lambda \in \mathbb{C}}\left(  \mathrm{tr}|_{W_\lambda} o^W(v) \right)q^\lambda\]
where $o^W(v) =  \res_z z^{m-1} Y^W(v,z)$ is the zero mode of $v$ giving the weight-preserving action of $v$ on $W$. When the action is clear, we will omit the superscript $W$ and just write $o^W(v) = o(v)$.

Note that if $v = {\bf 1}$, then $o^W({\bf 1}) = Id_W$ and 
\[Z_W({\bf 1}, \tau) = q^{-c/24} \,  \sum_{\lambda \in \C}  \, ( \mathrm{dim} \ W_\lambda) \, q^\lambda\] 
is called the {\it generalized graded dimension} of $W$. In the case $W = V$, then $Z_V({\bf 1}, q)$ is often called the {\it graded dimension}, or the {\it 0-point correlation function} or the {\it partition function} of $V$.

\subsection{Graded pseudo-traces in the $C_2$-cofinite setting following Miyamoto} 

To define graded pseudo-traces in \cite{Miyamoto2004}, Miyamoto assumes $V$ is $C_2$-cofinite, in which case he notes that each of the higher level Zhu algebras $A_n(V)$ for $V$ admits a symmetric linear map he calls $\phi$, and taking the quotient of $A_n(V)$ by a certain type of null space with respect to $\phi$, gives the resulting quotient the structure of a symmetric algebra.  Then a property of modules called ``interlocked with $\phi$" with respect to this structure of a symmetric algebra and linear operator $\phi$ is introduced and used to define the pseudo-trace for any weight-preserving linear operator on a $V$-module induced from such $V$-module ``interlocked with $\phi$".

Later in, for instance, \cite{Miyamoto-Ukraine}, Miyamoto introduces a simplified notion of ``interlocked" $V$-module, which we give here.  

\begin{defn} \label{interlocked-def}
${}$

(1) Let $U$ be an indecomposable $A$-module for an associative algebra $A$.  We say $U$ is {\it interlocked} if  for every submodule $U^{(1)}$ of $U$, there exists a submodule $U^{(2)}$ of $U$ such that $U^{(1)} \cong U/U^{(2)}$ and $U^{(2)} \cong U/U^{(1)}$ as $A$-modules. And in general, if $U^{(1)}$ and $U^{(2)}$ are submodules of $U$ satisfying $U^{(1)} \cong U/U^{(2)}$ and $U^{(2)} \cong U/U^{(1)}$ as $A$-modules, we say that $U^{(1)}$ {\it is interlocked with} $U^{(2)}$.  If the socle of $U$ is interlocked with the radical of $U$, but not all submodules of $U$ are necessarily interlocked with another submodule, we say that $U$ is {\it weakly interlocked}.

(2) Let $W$ be an indecomposable generalized  $V$-module for a vertex operator algebra $V$.  We say $W$ is {\it interlocked} as a $V$-module if for every submodule $W^{(1)}$ of $W$, there exists a submodule $W^{(2)}$ of $W$ such that $W^{(1)} \cong W/W^{(2)}$ and $W^{(2)} \cong W/W^{(1)}$ as $V$-modules. And in general, if $W^{(1)}$ and $W^{(2)}$ are submodules of $W$ satisfying $W^{(1)} \cong W/W^{(2)}$ and $W^{(2)} \cong W/W^{(1)}$ as $V$-modules, we say that $W^{(1)}$ {\it is interlocked with} $W^{(2)}$.  If the socle of $W$ is interlocked with the radical of $W$, but not all submodules $W$ are necessarily interlocked with another submodule, we say that $W$ is {\it weakly interlocked}. 

Here, the socle of a module is the direct sum of all simple submodules, and the radical is the intersection of its maximal proper submodules.  
\end{defn}

In the $C_2$-cofinite setting, Miyamoto indicates in \cite{Miyamoto2004, Miyamoto-Ukraine} that if $W$ is ``interlocked with $\phi$" as a $V$-module, then $\phi$ gives rise to a natural isomorphism $W/Rad(W)\cong Soc(W)$  where $Soc(W)$ denotes the socle of $W$, and $Rad(W)$ denotes the radical of $W$, and thus $W$ is weakly interlocked. Furthermore, Miyamoto claims that then there exists a basis for $W_\lambda$, for each $\lambda \in Spec L_0$, such that any weight preserving map  $\sigma$ on $W$ satisfying $\sigma(Soc(W))\subseteq Soc(W)$ has the form
\begin{equation}\label{Miyamoto-decomp}
 \sigma |_{W_\lambda} = \left[ \begin{array}{ccc}
A & C & B \\
0 & E & ^*C \\
0 & 0 & A
\end{array} \right] ,  
\end{equation}
 with respect to this basis, where 
\[A = { \sigma}|_{Soc(W)_\lambda}\] 
and $C$ and the notation $^*C$ reflect the isomorphism between $W/Rad(W)$ and $Soc(W)$ arising from $\phi$.

That is, for each weight space $W_\lambda$, in the $C_2$-cofinite setting, a basis can be chosen so that the first terms of the basis are a basis for $Soc(W)_\lambda$, the next terms complete to a basis for $Rad(W)_\lambda$, and the last terms reflect the isomorphism $Soc(W)_\lambda \cong (W/Rad(W))_\lambda$ arising from $\phi$.  This allows Miyamoto to define the pseudo-trace of $\sigma|_{W_\lambda}$ to be 
\begin{equation}\label{Miyamoto-pseudo-trace} 
\mathrm{tr}^\phi(\sigma|_{W_\lambda}) = \mathrm{tr}(B). 
\end{equation}
That is, for a generalized $V$-module $W$ that is ``interlocked with $\phi$", and a weight-preserving operator---such as the zero mode $o(v) = v_{\mathrm{wt} \, v -1}$ for $v \in V$---the restriction of this weight-preserving operator to $W_\lambda$ has a well-defined pseudo-trace in the $C_2$-cofinite setting.  This step of defining these pseudo-trace components is the crucial step for defining graded pseudo-traces. 

To define graded pseudo-traces, 
one also considers the operator $\sigma = o(\omega) = L_0$ which can be decomposed as $L_0|_{W_\lambda} = o(\omega)|_{W_\lambda} = (L^S_0)_\lambda  + (L^N_0)_\lambda$ into semi-simple and nilpotent parts.  Then one defines 
the formal series in $\log q$
\[q^{(L^N_0)_\lambda} = (e^{2 \pi i \tau})^{(L^N_0)_\lambda} = \sum_{j \in \mathbb{Z}_{\geq 0}} \frac{1}{j!}(L^N_0)_\lambda^j (\log q)^j  \]
where $\log q = 2 \pi i \tau$, for $\tau \in\mathbb{H}$.  Then for any weight-preserving operator $\sigma$ on $W$, define
\begin{equation}
tr_W^\phi(\sigma q^{L_0}) = \sum_{\lambda \in \mathbb{C}} \mathrm{tr}^\phi 
(\sigma|_{W_\lambda} q^{(L^N_0)_\lambda } ) q^{(L_0^S)_\lambda }
= \sum_{\lambda \in \mathbb{C}} \sum_{j \in \mathbb{Z}_{\geq 0}} \mathrm{tr}^\phi(\sigma|_{W_\lambda} \frac{1}{j!} (L^N_0)_\lambda^j (\log q)^j) q^\lambda. 
\end{equation}

The notion of graded pseudo-trace with respect to Miyamoto's definition of pseudo-trace, $\mathrm{tr}^\phi$, for an element $v$ in $V$ and a $V$-module $W$ that is  `interlocked with respect to $\phi$" for $V$ a $C_2$-cofinite vertex operator algebra, is then defined to be 
\begin{eqnarray}\label{Miyamoto-graded-pseudo-trace}
\mathrm{tr}^\phi_W (v,\tau ) &=&  \mathrm{tr}^\phi_W (o(v) q^{L_0 - c/24}) \ = \ \sum_{\lambda \in \mathbb{C}} \mathrm{tr}^\phi (o(v)|_{W_\lambda} q^{(L^N_0)_\lambda}  ) q^{(L_0^S)_\lambda} q^{ - c/24}.
\end{eqnarray} 

However, note that the key point in the notion of graded pseudo-trace being well defined is the decomposition of  any weight-preserving operator $\sigma$ at weight $\lambda$ with respect to some  ``canonical" basis of $W_\lambda$ so that $\sigma|_{W_\lambda}$ has a matrix realization given by Eqn.\ (\ref{Miyamoto-decomp}) where $C$ and $^*C$ are ``dual'' in some sense. These are the key aspects we will develop in the next Section for any $V$ and any $V$-module $W$ that satisfies the  property called ``strongly interlocked" along with certain invariance properties of a family of ``strongly interlocked bases".

 An important property used in both Zhu's proof of modular invariance in the $C_2$-cofinite rational setting for graded traces \cite{Z} and by Miyamoto's extension of the modular invariance to the $C_2$-cofinite irrational setting for graded pseudo-traces \cite{Miyamoto2004} is the fact that in these settings the graded traces or graded pseudo-traces, respectively, are symmetric  linear operators that also satisfy the logarithmic derivative property as given below.\\

{\bf Logarithmic Derivative Property:} 
\begin{eqnarray}\label{log-derivative}
tr^\phi_W(o(\omega)o(v) q^{L_0})  &=& \sum_{\lambda \in \mathbb{C} } (\mathrm{tr}^\phi|_{W_\lambda} o(\omega) o(v)q^{(L_0^N)_\lambda})) q^{(L_0^S)_\lambda} \nonumber\\
 &=&  \frac{1}{2\pi i}\frac{d}{d\tau} \sum_{\lambda \in \mathbb{C}} (\mathrm{tr}^\phi|_{W_\lambda}  o(v)q^{(L_0^N)_\lambda})) q^\lambda 
 \ = \ q\frac{d}{dq}\sum_{\lambda \in \mathbb{C}} (\mathrm{tr}^\phi|_{W_\lambda}  o(v)q^{(L_0^N)_\lambda}) q^\lambda  \nonumber \\
&=&  q \frac{d}{dq} q^{c/24} \mathrm{tr}^\phi_W(v, \tau)
\end{eqnarray}
 where $c$ denotes the central charge and $\omega$ denotes the conformal vector of the vertex operator algebra $V$.
For instance, letting $v = {\bf 1}$, the logarithmic derivative property in this specific case is
\begin{equation}
q^{c/24} \mathrm{tr}^\phi_W(\omega, \tau) = q \frac{d}{dq} q^{c/24} \mathrm{tr}^\phi_W({\bf 1} ,\tau) = \frac{1}{2 \pi i} \frac{d}{d \tau} q^{c/24} \mathrm{tr}^\phi_W ({\bf 1}, \tau) .
\end{equation}

\section{The notion of graded pseudo-trace for strongly interlocked $V$-modules in certain settings}.

In this Section we define the notion of strongly interlocked module which allows us to give a well-defined notion of graded pseudo-trace in certain settings without the need for a symmetric linear map $\phi$ on the higher level Zhu algebras.  Our notion of strongly interlocked then applies to a broader range of vertex operator algebras, in particular to those that are not $C_2$-cofinite,  yielding graded pseudo-traces that still satisfy symmetry, linearity, and the logarithmic derivative property as proved at the end of this Section. 

\begin{defn}
Let $W$ be an interlocked indecomposable generalized $V$-module. If $W$ is irreducible or is reducible but has only a finite number of nontrivial proper submodules $W^{(1)}, \dots, W^{(k-1)}$ satisfying:
\begin{enumerate}
\item each $W^{(j)}$ is indecomposable, for $j = 1, \dots, k$;
\item the submodules fit in a chain \begin{equation}\label{chain} 0  = W^{(0)} \subset W^{(1)} \subset W^{(2)} \subset \cdots  \subset W^{(k-1)} \subset W^{(k)} = W ; \textrm{ and} 
 \end{equation}
 \item   each subquotient in the chain satisfies $W^{(j+1)}/W^{(j)} \cong W^{(1)}$ for $j = 0,\dots, k-1$,
\end{enumerate}
then we call $W$ {\it strongly interlocked}.  In particular, we have $W^{(1)} = Soc(W)$ and for $k\geq 2$, we have $W^{(k-1)} = Rad(W)$, and  $W^{(j)}$ is interlocked with $W^{(k-j)}$ for each $j = 0, \dots, k$. 
\end{defn}

We have the following linear algebra result for strongly interlocked indecomposable generalized $V$-modules.

\begin{prop}\label{matrix-form-prop}
If $W$ is a strongly interlocked indecomposable generalized $V$-module with all proper submodules given by the chain (\ref{chain}), then there exists a basis for each  weight space $W_\lambda$ of $W = \coprod_{\lambda \in \mathbb{C}} W_\lambda = \coprod_{\lambda \in Spec L_0} W_\lambda$, such that any linear  weight-preserving  $V$-module operator $\sigma$ (such as $o(v)$) restricted to $W_\lambda$ with respect to this basis is given by
\begin{equation}\label{Our-decomp}
\sigma|_{W_\lambda}  = \left[ \begin{array}{cccccc}
A & C_1 & C_2 & \cdots & C_{k-1} & B \\
0 & A & \star & \cdots & \star & C_{k-1} \\
0& 0 & A & \cdots & \star & C_{k-2} \\
\vdots & \vdots & \vdots & \ddots & \vdots & \vdots \\
0 & 0 & 0 & \cdots & A & C_1\\
0 & 0 & 0 & \cdots & 0 & A 
\end{array} \right] ,  
\end{equation}
where $A = \sigma|_{Soc(W)_\lambda}$ and each submatrix $A, C_1, \dots, C_{k-1}, B$, and $\star$ is $d_\lambda \times d_\lambda$ where $d_\lambda = \dim Soc(W)_\lambda$. 
\end{prop}
\begin{proof}  If $\dim W^{(1)}_\lambda = \dim Soc(W)_\lambda = d_\lambda$, then $\dim Rad (W)_\lambda = \dim W_\lambda - d_\lambda$.  By the fact that $W$ is strongly interlocked so that $W/W^{(k-j)} \cong W^{(j)}$ and $W^{(j)}/W^{(j-1)}\cong W^{(1)}$ for each $1 \leq j \leq k$, it follows that $\dim W_\lambda = k d_\lambda$ and $\dim W^{(j)}_\lambda = j d_\lambda$ for each $1 \leq j \leq k$.

The fact that $\sigma$ is a weight-preserving module map implies that we can choose a basis 
\begin{equation}\label{chain-basis} \mathcal{B}_{W_\lambda} = \{ B_{1,1}, \dots, B_{1,d_\lambda}, B_{2,1}, \dots, B_{2,d_\lambda}, \dots, B_{k,1}, \dots, B_{k,d_\lambda} \}
\end{equation}
where $\{ B_{1,1}, \dots, B_{1,d_\lambda}, B_{2,1}, \dots, B_{2,d_\lambda} , \dots, B_{j,1}, \dots, B_{j,d_\lambda} \} $ is a basis for $W^{(j)}_\lambda$ for each $1 \leq j \leq k$, and $\{B_{j+1,1} + W^{(j)}_\lambda, \dots, B_{j+1,d_\lambda} + W^{(j)}_\lambda \}$ is a basis for $(W^{(j+1)}/W^{(j)})_\lambda \cong W^{(1)}_\lambda$ such that $\pi^{(j+1)}_{(j)} \sigma|_{W^{({j+1})}_\lambda} = A$ where $\pi^{(j+1)}_{(j)}$ is the projection onto the span of $\{ B_{j+1, 1}, \dots, B_{j+1, d_\lambda} \}$.  Then with respect to this basis 
$\sigma|_{W_\lambda}$ is block upper triangular with blocks that are each $d_\lambda \times d_\lambda$   
with $A = \sigma|_{Soc(W)_\lambda}$ blocks on the diagonal.  
\end{proof} 

Note in particular, that Proposition \ref{matrix-form-prop}
implies that if $W$ is strongly interlocked, then  $\sigma|_{W_\lambda}$ for any weight-preserving module map $\sigma$ and $\lambda  \in \mathbb{C}$ has a matrix realization  of the form of Eqn.\ (\ref{Miyamoto-decomp}) with 
$C = [ C_1 \ C_2 \ \cdots \ C_{k-1} ]$, 
\[ ^*C = \begin{bmatrix}
  C_{k-1} \\
  C_{k-2} \\
  \vdots\\
  C_2 \\
  C_1
\end{bmatrix} \] 
and $E$ a block upper triangular matrix with matrices $A$ on the block diagonal.  
We will call a collection of bases $\{\mathcal{B}_{W_\lambda}\}_{\lambda \in Spec L_0}$ for $\{W_\lambda\}_{\lambda \in Spec L_0}$ 
a {\it strongly interlocked family of bases} if for any weight-preserving module map  $\sigma$ and $\lambda \in Spec L_0$, $\sigma|_{W_\lambda}$ has a matrix representation of the form (\ref{Our-decomp}) with respect to these bases.

If $\mathrm{tr}(B)$ is invariant with respect to a choice of strongly interlocked family of bases then, 
analogously to the setting of Miyamoto but without the need for the higher level Zhu algebras or $\phi$, we can define the pseudo-trace at each weight  $\lambda \in SpecL_0$ to be 
\begin{equation}\label{our-pstr} 
\mathrm{pstr}(\sigma|_{W_\lambda}) = \mathrm{tr}(B) 
\end{equation}
for $B$ the upper right corner block in Eqn. (\ref{Our-decomp}) realizing $\sigma|_{W_\lambda}$ with respect to a strongly interlocked basis.  Then,  in analogy to Eqn.\ (\ref{Miyamoto-graded-pseudo-trace}) in Miyamoto's setting, but using Eqn.\ (\ref{our-pstr}) instead of Eqn.\ (\ref{Miyamoto-pseudo-trace}), 
we define the {\it graded pseudo-trace of the strongly interlocked $V$-module $W$ associated to $v \in V$}  to be the formal series in $q$ and $\log q$ (or equivalently in $\tau$ for $q = e^{2 \pi i \tau}$ and $\log q = 2 \pi i \tau$), given by 
\begin{eqnarray}\label{def-graded-pseudo-trace}
\mathrm{pstr}_W (v,\tau ) &=&  \mathrm{pstr}_W (o(v) q^{L_0 - c/24}) \ = \ \sum_{\lambda \in \mathbb{C}} \mathrm{pstr}(o(v)|_{W_\lambda} q^{(L^N_0)_\lambda}  ) q^{(L_0^S)_\lambda} q^{ - c/24} \nonumber \\
&=& \sum_{\lambda \in \mathbb{C}} \sum_{j \in \mathbb{Z}_{\geq 0}} \mathrm{pstr}(o(v)|_{W_\lambda} \frac{1}{j!} (L^N_0)_\lambda^j (\log q)^j) q^{\lambda - c/24}. 
\end{eqnarray}

Note that if $W$ is irreducible then the graded notion of graded pseudo-trace is vacuous and so we define the graded pseudo-trace to be zero.  

In this work, we study two settings in which strongly interlocked modules will have well-defined graded pseudo-traces, and which includes important examples, namely the two most important $C_1$-cofinite vertex operator algebras which are not $C_2$-cofinite---the Heisenberg vertex operator algebras and the universal Virasoro vertex operator algebras.   In these settings and in the examples of these settings we study in this paper, we consider only modules induced from level zero, but analogously to \cite{Miyamoto2004}, under certain circumstances this can be extended to generalized Verma modules induced from an $A_n(V)$-module.

The first particular setting of well-defined graded pseudo-traces we give here applies to all indecomposable modules for the Heisenberg vertex operator algebra as shown in Section \ref{Heisenberg-section} and all modules induced from the level zero Zhu algebra for the universal Virasoro vertex operator algebra with generic central charge, as shown in Section \ref{Virasoro-classification-subsection} which is Case (0) for Theorem \ref{interlocked-thm2}.  This is the setting for which Theorem \ref{extra-conditions-theorem} below applies and was refered to as Setting 1 in the introduction.  The second setting applies to examples developed in Sections \ref{first-induced-module-section} and \ref{Virasoro-classification-subsection} for certain modules for the universal Virasoro vertex operator algebra at central charge $c = 1$ or $c = 25$, which is Case (1)(ii) in Theorem \ref{interlocked-thm2}. This is the setting for which Theorem \ref{Virasoro-exceptional-theorem} applies and was refered to as Setting 2 in the introduction. 

Here is the first setting and our results:

\begin{thm}\label{extra-conditions-theorem}
Let $V$ be a vertex operator algebra satisfying the following:

(i) $V = \langle v\rangle = \mathrm{span} \{ v_{-n_1} v_{-n_2} \cdots v_{-n_j}{\bf 1} \; | \; j \in \mathbb{Z}_{\geq 0}, n_1,\dots, n_j \in  \mathbb{Z}_{>0} \} $.  

(ii) $A_0(V) \cong \mathbb{C}[o(v)]$.  

For $k \in \mathbb{Z}_{>0}$, $\lambda \in \mathbb{C}$,  and $U(\lambda, k) = \mathbb{C}[x]/((x - \lambda)^k)$ for $x = o(v)$, let $\{u_1, \dots, u_k\}$ be a Jordan basis for $o(v)$ acting on $U(\lambda,k)$, and consider the $\mathcal{U}(\hat{V})$-module $M_0(U(\lambda,k))$, with submodule $\langle u_1 \rangle = \mathcal{U}(\hat{V}).u_1$.  

If for each $\ell \in \mathbb{Z}_{\geq 0}$, 
the bilinear form on $\langle u_1 \rangle(\ell) = \mathcal{U}(\hat{V})_{\ell}u_1$ defined by 
\begin{eqnarray}
< \cdot, \cdot >_\ell \ \ : \ \ \langle u_1 \rangle (\ell)  \otimes \langle u_1 \rangle (\ell) &\longrightarrow& \mathbb{C}\\
v_{-m_1}\cdots v_{-m_i} u_1 \otimes v_{-n_1} \cdots v_{-n_j}u_1 &\mapsto& <v_{-m_1}\cdots v_{-m_i} u_1,  v_{-n_1} \cdots v_{-n_j} u_1 >_\ell \nonumber \\
&& = <u_1, v_{m_i} \cdots v_{m_1} v_{-n_1} \cdots v_{-n_j}u_1 >_0 \nonumber
\end{eqnarray}
with $<u_1, u_1>_0 = 1$ 
is nondegenerate, then $W(\lambda,k) = \mathfrak{L}_0(U(\lambda,k))$ is a strongly interlocked $V$-module, is the universal $V$-module $\overline{M}_0(U(\lambda,k))$ which is equal to the Verma module $M_0(U(\lambda,k))$, and $\mathrm{pstr}_{W(\lambda,k)}(w, \tau)$ is well defined for every $w \in V$.  

 In particular, setting $W^{(j)} = \mathfrak{L}_0(\mathrm{span}\{u_1, \dots, u_j\})$ for $j = 0,\dots,k$, we have that $W^{(j)}$ is interlocked with $W^{(k-j)}$, and 
$W^{(1)} = \langle u_1 \rangle = Soc(W(\lambda,k))$ is interlocked with  $W^{(k-1)} = \langle u_1, \dots, u_{k-1} \rangle = Rad(W(\lambda,k))$.
\end{thm}

\begin{proof} 
There exists a Jordan basis for the $\mathbb{C}[x]$-module $U(\lambda,k)$ given by $\{u_1, u_2, \dots, u_k\}$ and which is a unique Jordan basis up to a nonzero scalar multiple.  
Letting $U^{(j)} = \mathrm{span} \{ u_1, \dots, u_j\} \cong U(\lambda,j)$ for $j = 1,\dots, k$, we have
\[0 =U^{(0)} \subset U^{(1)} \subset \cdots \subset U^{(k-1)} \subset U^{(k)} = U(\lambda, k)\]
and $U(\lambda, k)$ is a strongly interlocked $A_0(V)$-module with $U^{(j)}$ interlocked with $U^{(k-j)}$.  
 This chain induces a chain of submodules
\begin{align*}
0 = W^{(0)} \subset W^{(1)} \subset \cdots \subset W^{(k-1)} \subset W^{(k)} = \mathfrak{L}_0(U(\lambda, k))
\end{align*}
 where $W^{(j)} = \mathfrak{L}_0(\mathrm{span}\{u_1,\dots, u_j\}) = \mathfrak{L}_0(U^{(j)}) \cong \mathfrak{L}_0(U(\lambda, j))$. 

Since the bilinear form on $\langle u_1 \rangle (\ell)$ is nondegenerate for each $\ell \in \mathbb{Z}_{\geq0}$, $J_0(U^{(1)}) = 0$ and $W^{(1)} = M_0(U^{(1)}) = \overline{M}(U^{(1)})$ is the universal irreducible $V$-module such that $o(v)$ acts as $\lambda Id$ on the lowest weight space. 
Furthermore, since $o(v)$ also acts as $\lambda Id$ on the vector $u_j + U^{(j-1)} \in U^{(j)}/U^{(j-1)} \cong U(\lambda, 1) \cong U^{(1)}$ for each $j = 1, \dots,k$, as a $V$-module $\mathfrak{L}_0(U^{(j)}/U^{(j-1)}) \cong \mathfrak{L}_0(U^{(1)}) = \overline{M}(U^{(1)})$, and $J_0(U^{(j)}/U^{(j-1)}) = 0$. By the definition of $J_0$, this implies that the bilinear form is nondegenerate on each $u_j + W^{(j-1)}$.

Next, we claim that $W^{(j)}$ for $j = 0, \dots, k$ are all the submodules of $W^{(k)}$.  To prove this, let  $N$ be a $V$-submodule of $W^{(k)}$. For simplicity, we set $M/J= M_0(U(\lambda,k))/J_0(U(\lambda,k)) = \mathfrak{L}_0(U(\lambda, k)) = W^{(k)}$. So, the submodule $N$ of $M/J$ is of the form $E/J$ where $J\subseteq E\subseteq M$. Since $M/J$ is admissible, we can write $M/J=\coprod_{\ell=0}^{\infty}(M/J)(\ell)$ where $(M/J)(0)\neq \{0\}+J$ and $(M/J)(0)= U(\lambda,k)$. Then we have $N(\ell)=(E/J)(\ell)=(M/J)(\ell)\cap N$,
and $N(0) \subset (M/J)(0) \cong U(\lambda,k)$ as a $\mathbb{C}[x]$-submodule.

 Thus $N(0) = U^{(j)}$ for some $j = 0, \dots, k$, and $W^{(j)} = \mathcal{U}(\hat{V}).U^{(j)} + J \subset N$.  We claim that $N = W^{(j)}$.  Indeed, suppose $w \in N \smallsetminus W^{(j)} \subset W^{(k)}$. Without loss of generality, we can assume $w$ is a homogeneous element of some degree $d$.  Thus $w = \sum_{\ell = 1}^k R_{d,\ell} u_\ell$ for $R_{d,\ell}{\bf 1} \in \mathcal{U}(\hat{V})_d$ and $R_{d,\ell} u_\ell \neq 0$ for some $\ell \in \{j+1, \dots, k\}$, say $R_{d,r} u_r \neq 0$ for $r \in \{j+1, \dots, k\}$. But since the bilinear form is nondegenerate on $u_r + W^{(r-1)}$, this implies that there exists $T_{d,r} {\bf 1} \in V_d$, such that  $<T_{d,r} u_r, R_{d,r} u_r> = <u_r, T_{d,r}^\dag R_{d,r} u_r >_0 \neq 0$ where we use the notation $(v_{n_1} \cdots v_{n_m})^\dag = v_{-n_m} \cdots v_{-n_1}$.  But this would imply that  $T_{d,r}^\dag w = \sum_{\ell = 1}^k T_{d,r}^\dag R_{d,\ell} u_\ell$ is both in $N(0) = U^{(j)} = \mathrm{span} \{u_1, \dots, u_j\}$ and has a nonzero component $u_r$ for $r \geq j+1$, which is impossible.  Thus no such $w \in N \smallsetminus W^{(j)}$ exists, and $N = W^{(j)}$, showing that the $W^{(j)}$ for $j = 0, \dots, k$ are all the submodules of $W^{(k)}$.

Next note that each quotient $W^{(j)}/W^{(j-1)} = \langle u_j \rangle + W^{(j-1)}$ for $j = 1, \dots, k$ is an irreducible $V$-module such that $o(v)$ acts as $\lambda Id$ on the lowest weight space.  However, up to isomorphism, there is a unique such $V$-module $W^{(1)}$ since there is a unique such $\mathbb{C}[o(v)]$-module $U(\lambda,1) \cong U^{(1)}$.  Thus $W^{(j)}/W^{(j-1)} \cong W^{(1)}$.  

Furthermore, for each $\lambda \in \mathbb{C}$ and $j = 1, \dots, k$, up to isomorphism there is a unique $\mathbb{C}[o(v)]$-module $U(\lambda, j)$, and thus there is a unique $V$-module $W^{(j)} \subset W^{(k)}$ such that $o(v)$ acts as a $j \times j$ Jordan block on the degree 0 subspace.  Thus  
\begin{eqnarray*}
 W^{(k)}/W^{(k-j)} =  \langle u_1, \dots, u_k\rangle/\langle u_1, \dots, u_{k-j} \rangle = \langle u_{k-j + 1}, \dots, u_k \rangle + W^{(k-j)} \cong \langle u_1, \dots, u_j \rangle = W^{(j)}.
\end{eqnarray*}
Substituting $k - j$ for $j$ above, this proves that $W^{(j)}$ is interlocked with $W^{(k-j)}$, proving that $W(U(\lambda,k)) = W^{(k)}$ is strongly interlocked.

In particular,  since $\langle u_1 \rangle$ is a unique irreducible submodule of $W^{(k)} = W(\lambda,k)$ and $\langle u_1, \dots, u_{k-1} \rangle$ is a unique maximal submodule of $W(\lambda, k)$, we have that $Soc(W(\lambda,k)) = \langle u_1, \dots, u_{k-1} \rangle$ and $Rad(W(\lambda,k)) = \langle u_1, \dots, u_{k-1} \rangle$.

It is left to prove that $W^{(k)}$ has well-defined graded pseudo-traces.  Indeed, by property {\it (i)} of $V$, for $d \in \mathbb{Z}_{\geq 0}$, there is a PBW basis $\mathcal{B}_d {\bf 1}$ for $V_d$ with $\mathcal{B}_d = \{B_1, \dots, B_{p_d}\}$.  Then     \begin{equation}  
\mathcal{B}_d(U^{(k)}) = \{B_1u_1, B_2u_1, \dots, B_{p_d}u_1, B_1u_2, B_2 u_2, \dots, B_{p_d}u_2, \dots, B_1u_k, B_2 u_k,\dots, B_{p_d}u_k \}  
\end{equation}
is a strongly interlocked basis for $W^{(k)}(d)$ such that for any $\sigma = o(u)$ for $u \in V$, then the matrix representation of $\sigma|_{W^{(k)}(d)}$ with respect to this basis, denoted $[\sigma|_{W^{(k)}(d)}]_{\mathcal{B}_d(U^{(k)})}$, is of the form (\ref{Our-decomp}), and any strongly interlocked basis for $W^{(k)}(d)$ is of this form by the uniqueness of the Jordan basis for $U^{(k)} \cong U(\lambda,k)$ up to a scalar multiple of the identity.  Thus choices of PBW bases, $\{\mathcal{B}_d\}_{d \in \mathbb{Z}_{\geq 0}}$ for $V$ give strongly interlocked families of bases $\{\mathcal{B}_d(U^{(k})\}_{d \in \mathbb{Z}_{\geq 0}}$ for $W^{(k)}$,  and any strongly interlocked family of bases is of this form.

 Moreover, for any other strongly interlocked basis for $W^{(k)} (d)$, there is a corresponding 
 PBW basis $\mathcal{B}'_d{\bf 1} = \{B'_1, \dots, B'_{p_d}\}.{\bf 1}$ for $V_d$.  Then the  corresponding matrix representation $[\sigma|_{W^{(k)}(d)}]_{\mathcal{B}'_d}$ of $\sigma|_{W^{(k)}(d)}$ with respect to this basis will be related by the change of basis matrix with $P$ on the diagonal where $P$ is the change of basis matrix from $\mathcal{B}_d$ to $\mathcal{B}'_d$. We denote this block diagonal change of basis matrix by $P_{block}$.  Thus if $\mathrm{pstr} [\sigma|_{W^{(k)}(d)}]_{\mathcal{B}_d(U^{(k)})} = \mathrm{tr} B$ then 
\[ \mathrm{pstr} [\sigma|_{W^{(k)}(d)}]_{\mathcal{B}'_d(U^{(k)})} = \mathrm{pstr} P_{block}^{-1} [\sigma|_{W^{(k)}(d)}]_{\mathcal{B}_d(U^{(k)})} P_{block}  = \mathrm{tr} P^{-1} B P = \mathrm{tr} B, \]
and the pseudo-trace of $\sigma|_{W^{(k)}(d)}$ is well defined.  Thus the graded pseudo-trace of $W^{(k)}$ for the grade preserving $V$-module operator $\sigma$ is also well defined.   
\end{proof}

Next we give the second setting in which we assume we know a priori that $W$ is a strongly interlocked $V$-module and applies to Case (1)(ii) as defined in Section \ref{first-induced-module-section} and applied in Section \ref{Case(1)(ii)-subsection}:

\begin{thm}\label{Virasoro-exceptional-theorem}
Let $V$ be a vertex operator algebra, and $\lambda \in \mathbb{C}$, such that the following hold:

\smallskip

(i) There is a unique irreducible $V$-module of conformal weight $\lambda$, denoted $W^{(1)}(\lambda)$.

(ii) The irreducible $A_0(V)$-module $U(\lambda) = W^{(1)}(\lambda)(0)$ is one dimensional.  

(iv) $W(\lambda)$ is  a strongly interlocked $V$-module of conformal weight $\lambda \in \mathbb{C}$.

\smallskip

\noindent
Then $W(\lambda)$ has well-defined graded pseudo-traces. 
\end{thm}

\begin{proof}
By (i), the bijection between irreducible $A_0(V)$-modules and irreducible $V$-modules implies there exists an irreducible $A_0(V)$-module $U(\lambda)$ such that $W^{(1)}(\lambda) = \mathfrak{L}_0(U(\lambda)) = M_0(U(\lambda))/J_0(U(\lambda))$, and this $U(\lambda)$-module is unique such that $L_0$ acts as scalar multiplication by $\lambda$.  Then  $W^{(1)}(\lambda)(0) = U(\lambda)$ and $W^{(1)}(\lambda)(d) = \mathcal{U}(\hat{V})_d. U(\lambda)/J_0(U(\lambda))$ for each $d \in \mathbb{Z}_{> 0}$.  

For simplicity, for the rest of the proof we will denote $W^{(1)}(\lambda)$ by $W^{(1)}$, $W(\lambda)$ by $W$, $U(\lambda)$ by $U$, and $J_0(U(\lambda))$ by $J$.

Since $W$ is strongly interlocked, $W$ satisfies (\ref{chain}) for some $k \in \mathbb{Z}_{> 0}$ with $W^{(1)} = W^{(1)}(\lambda)$, and there exists a strongly interlocked family of bases for $W$ such that every grade preserving $V$-module map $\sigma$ restricted to the weight space $W_\mu = W(\lambda)_\mu$ is of the form (\ref{Our-decomp}) with respect to this basis.

In particular, for each $d \in \mathbb{Z}_{\geq 0}$, there is a basis $\mathcal{B}_d = \{B_1u_1 + J, B_2u_1 +J, \dots, B_{j_d}u_1+J \}$ for $W^{(1)}(d)$ where $\{u_1 + J\}$ is a basis for $W^{(1)}(0)$ such that any grade preserving module map restricted to $W^{(1)}(d)$ with respect to this basis, denoted $[\sigma|_{W^{(1)}(d)}]_{\mathcal{B}_d}$ has matrix representation $A$.  Then since $W$ is interlocked, letting $u_i + J + W^{(i-1)}$ be a choice of basis for $W^{(i)}/W^{(i-1)}$ for $i = 1, \dots, k$, then 
\begin{multline}
\mathcal{B}_d^k = \{B_1u_1 +J, B_2u_1 +J, \dots, B_{j_d}u_1+J, B_1u_2+J, B_2 u_2+J, \dots, B_{j_d}u_2+J, \dots, \\
B_1u_k+J, B_2 u_k+J,\dots, B_{j_d}u_k+J \}     
\end{multline} 
is a strongly interlocked family of bases for $W$, and any grade preserving module map restricted to $W(d)$ with respect to this basis is of the form (\ref{Our-decomp}).  

 Moreover, for any other strongly interlocked  basis for $W$ at degree $d$, there is a corresponding 
 basis $\tilde{\mathcal{B}}_d = \{\tilde{B}_1u_1 + J, \dots, \tilde{B}_{j_d}u_1 +J\}$ for $W^{(1)}(d)$.  Then the  corresponding matrix representation $[\sigma|_{W(d)}]_{\tilde{\mathcal{B}}^k_d}$ of $\sigma|_{W(d)}$ with respect to the corresponding  basis will be related to $[\sigma|_{W^{(1)}(d)}]_{\mathcal{B}_d}$ by the change of basis matrix consisting of blocks of $P$ on the diagonal where $P$ is the change of basis matrix from $\mathcal{B}_d$ to $\tilde{\mathcal{B}}_d$. We denote this block diagonal change of basis matrix by $P_{block}$.  Thus if $\mathrm{pstr} [\sigma|_{W(d)}]_{\mathcal{B}^k_d} = \mathrm{tr} B$ then 
\[ \mathrm{pstr} [\sigma|_{W(d)}]_{\tilde{\mathcal{B}}^k_d} = \mathrm{pstr} P_{block}^{-1} [\sigma|_{W(d)}]_{\mathcal{B}^k_d} P_{block}  = \mathrm{tr} P^{-1} B P = \mathrm{tr} B, \]
and the pseudo-trace of $\sigma|_{W(d)}$ is well defined.  Thus the graded pseudo-trace of $W$ for the grade preserving $V$-module operator $\sigma$ is also well defined.   
\end{proof}

Importantly, we have the following  properties of graded pseudo-traces for strongly interlocked modules that have well-defined graded pseudo-traces, i.e. graded pseudo-traces that are invariant under a change of strongly interlocked family of bases.

\begin{thm}\label{log-thm} Let $V$ be a vertex operator algebra, and let $W =\bigoplus_{\lambda \in \mathbb{C}} W_{\lambda}$ be a strongly interlocked generalized $V$-module.  Then  when graded pseudo-traces $\mathrm{pstr}_W(v,\tau)$ are well defined, we have that the following hold:

(1) {\bf linearity:} If $\alpha, \beta$ are weight-preserving module maps on $W$, then 
\[\mathrm{pstr}((\alpha + \beta)|_{W_{\lambda}}) = \mathrm{pstr}(\alpha|_{W_{\lambda}}) + \mathrm{pstr}( \beta|_{W_{\lambda}}),\quad \mbox{and} \quad \mathrm{pstr}((\mu \alpha)|_{W_{\lambda}}) = \mu (\mathrm{pstr}(\alpha|_{W_{\lambda}}))\]
for $\mu \in \mathbb{C}$, and thus, for instance linearity holds for 
$\mathrm{pstr}_W(v, \tau)$ for $v\in V$.

(2) {\bf symmetry:} If $\alpha, \beta$ are weight-preserving  module maps on $W$, then 
\[\mathrm{pstr}((\alpha \beta)|_{W_{\lambda}}) = \mathrm{pstr}((\beta \alpha)|_{W_{\lambda}}),\]
and thus, for instance,
\begin{equation}\label{simple-symmetry}\mathrm{pstr}_W(o(u)o(v)q^{L_0}) = \mathrm{pstr}_W(o(v)o(u)q^{L_0})
\end{equation}
for all $u,v \in V$. 

(3) {\bf logarithmic derivative property:} 
For all $v \in V$
\[\mathrm{pstr}_W(o(\omega)o(v)q^{L_0}) = \frac{1}{2\pi i} \frac{d}{d\tau} \mathrm{pstr}_W (o(v)q^{L_0}) = q \frac{d}{dq} \mathrm{pstr}_W (o(v) q^{L_0}),\]
where $\omega$ is the conformal element of $V$.
\end{thm}

\begin{proof}
 Assume $W$ is a strongly interlocked $V$-module with well-defined graded pseudo-traces. 

(1) Follows immediately from the definition and the fact that the trace is linear.

(2) Using Proposition \ref{matrix-form-prop}, for every pair of weight-preserving linear maps $\alpha$ and $\beta$ on $W$, and for every weight $\lambda$, we choose a strongly interlocked basis for $W_\lambda$ such that $\alpha$ and $\beta$ are expressed in the form (\ref{Our-decomp}), with blocks denoted by $A_\alpha, C_{1,\alpha}, \dots, C_{k-1,\alpha}$ and  $B_\alpha$ for $\alpha|_{W_\lambda}$ and
$A_\beta, C_{1,\beta}, \dots, C_{k-1,\beta}$ and  $B_\beta$ for $\beta|_{W_\lambda}$.
Then
\begin{eqnarray*}
\mathrm{pstr}(( \alpha \beta)|_{W_\lambda}) &=& \mathrm{tr} (A_\alpha B_\beta + C_{1,\alpha}C_{k-1,\beta} + C_{2, \alpha} C_{k-2,\beta} + \cdots + C_{k-1,\alpha}C_{1, \beta} + B_{\alpha} A_\beta) \\
&=& \mathrm{tr}(A_\alpha B_\beta) + \sum_{j = 1}^{k-1} \mathrm{tr} (C_{j,\alpha}C_{k-j,\beta}) + \mathrm{tr}(B_\alpha A_\beta) \\
&=& \mathrm{tr}(B_\beta A_\alpha) + \sum_{j = 1}^{k-1} \mathrm{tr} (C_{k-j,\beta}C_{j,\alpha}) + \mathrm{tr}(A_\beta B_\alpha )\\
&=& \mathrm{tr}(A_\beta B_\alpha + C_{1,\beta}C_{k-1,\alpha} + C_{2, \beta} C_{k-2,\alpha} + \cdots + C_{k-1,\beta}C_{1, \alpha} + B_{\beta} A_\alpha)\\
&=& \mathrm{pstr}(( \beta \alpha)|_{W_\lambda}).
\end{eqnarray*}
Furthermore, since we assume that $W$ has well-defined graded pseudo-traces, for any other choice of strongly interlocked basis for $W_\lambda$, we again have $\mathrm{pstr}((\alpha \beta)|_{W_\lambda}) = \mathrm{pstr} ((\beta \alpha)_{W_\lambda})$. 
This gives the first part of (2). Eqn.\ (\ref{simple-symmetry}) follows from the definition of $\mathrm{pstr}_W(o(u)o(v)q^{L_0}$, i.e.
\begin{eqnarray*}
\mathrm{pstr}_W(o(u)o(v)q^{L_0}) &=& \sum_{\lambda\in\mathbb{C}}\sum_{j\in\mathbb{Z}_{\geq 0}}\mathrm{pstr} \big(o(u)o(v)|_{W_{\lambda}}\frac{1}{j!}(L_0^N)^j_{\lambda}(\log q)^j \big)q^{\lambda} \\
&=& \sum_{\lambda\in\mathbb{C}}\sum_{j\in\mathbb{Z}_{\geq 0}}\mathrm{pstr} \big(o(v)o(u)|_{W_{\lambda}}\frac{1}{j!}(L_0^N)^j_{\lambda}(\log q)^j \big)q^{\lambda} \\
&=& \mathrm{pstr}_W(o(v)o(u)q^{L_0}).
\end{eqnarray*}

(3) Note that by linearity of trace, we have that $q\frac{d}{dq} \mathrm{tr} (X(q)) = \mathrm{tr} (q\frac{d}{dq}X(q))$, and thus by linearity the same holds for the pseudo-trace. 

Using the notation $o(\omega) = L_0 = L_0^S + L_0^N$,  and $o(v)|_{W_\lambda} = o(v)_\lambda$, by linearity, we have 
\begin{eqnarray*}
\lefteqn{\mathrm{pstr}_W\big(o(\omega)o(v)q^{L_0} \big) }\\
&=& \sum_{\lambda\in\mathbb{C}}\sum_{j\in\mathbb{Z}_{\geq 0}}\mathrm{pstr} \big(o(\omega)o(v)|_{W_{\lambda}}\frac{1}{j!}(L_0^N)^j_{\lambda}(\log q)^j \big)q^{\lambda} \\
&=& \sum_{\lambda\in\mathbb{C}}\sum_{j\in\mathbb{Z}_{\geq 0}}\mathrm{pstr} \big(((L_0^S)_\lambda + (L_0^N)_\lambda) o(v)_{\lambda}\frac{1}{j!}(L_0^N)^j_{\lambda}(\log q)^j \big)q^{\lambda}\\
&=& \sum_{\lambda\in\mathbb{C}}\sum_{j\in\mathbb{Z}_{\geq 0}}\mathrm{pstr} \big((\lambda o(v)_\lambda + (L_0^N)_\lambda o(v)_\lambda) \frac{1}{j!}(L_0^N)^j_{\lambda}(\log q)^j \big)q^{\lambda} \\ 
&=& \sum_{\lambda\in\mathbb{C}}\sum_{j\in\mathbb{Z}_{\geq 0}}\mathrm{pstr} \big((\lambda o(v)_\lambda\frac{1}{j!}(L_0^N)^j_{\lambda}(\log q)^j +  o(v)_\lambda \frac{1}{j!}(L_0^N)^{j+1}_{\lambda}(\log q)^j \big)q^{\lambda}   \\ 
&=& \sum_{\lambda\in\mathbb{C}}\sum_{j\in\mathbb{Z}_{\geq 0}}\mathrm{pstr} \big((\lambda o(v)_\lambda\frac{1}{j!}(L_0^N)^j_{\lambda}(\log q)^j \big)q^\lambda  
+ \sum_{\lambda\in\mathbb{C}}\sum_{j\in\mathbb{Z}_{\geq 0}}\mathrm{pstr} \big( o(v)_\lambda \frac{1}{j!}(L_0^N)^{j+1}_{\lambda}(\log q)^j \big)q^{\lambda} \\
&=& \sum_{\lambda\in\mathbb{C}} \sum_{j\in\mathbb{Z}_{\geq 0}}\mathrm{pstr} \big(( o(v)_\lambda\frac{1}{j!}(L_0^N)^j_{\lambda}(\log q)^j \big)\lambda q^\lambda  + \sum_{\lambda\in\mathbb{C}}\sum_{n\in\mathbb{Z}_{> 0}}\mathrm{pstr} \big( o(v)_\lambda \frac{1}{(n-1)!}(L_0^N)^{n}_{\lambda}(\log q)^{n-1} \big)q^{\lambda} \\
&=& \sum_{\lambda\in\mathbb{C}} \sum_{j\in\mathbb{Z}_{\geq 0}}\mathrm{pstr} \big(( o(v)_\lambda\frac{1}{j!}(L_0^N)^j_{\lambda}(\log q)^j \big)q \frac{d}{dq} q^\lambda  
 + \sum_{\lambda\in\mathbb{C}}\sum_{n\in\mathbb{Z}_{\geq 0}}\mathrm{pstr} \big( o(v)_\lambda \frac{n}{n!}(L_0^N)^{n}_{\lambda}(\log q)^{n-1} \big)q^{\lambda} \\
&=& \sum_{\lambda\in\mathbb{C}} \sum_{j\in\mathbb{Z}_{\geq 0}}\mathrm{pstr} \big(( o(v)_\lambda\frac{1}{j!}(L_0^N)^j_{\lambda}(\log q)^j \big)q \frac{d}{dq} q^\lambda  
+ \sum_{\lambda\in\mathbb{C}}\sum_{j\in\mathbb{Z}_{\geq 0}}\mathrm{pstr} \big( o(v)_\lambda \frac{1}{j!}(L_0^N)^{j}_{\lambda}q \frac{d}{dq} \left( (\log q)^{j}\right) \big)q^{\lambda} \\
&=&  \sum_{\lambda\in\mathbb{C}}\sum_{j\in\mathbb{Z}_{\geq 0}} 
 q \frac{d}{dq} \left(\mathrm{pstr} \big( o(v)|_{W_{\lambda}}\frac{1}{j!}(L_0^N)^j_{\lambda}(\log q)^j \big)\right) q^{\lambda} \\
& & \quad  + \sum_{\lambda\in\mathbb{C}}\sum_{j\in\mathbb{Z}_{\geq 0}} \mathrm{pstr} \big( o(v)|_{W_{\lambda}}\frac{1}{j!}(L_0^N)^j_{\lambda}(\log q)^j \big) q \frac{d}{dq} q^{\lambda}\\
 &=&  q \frac{d}{dq} \sum_{\lambda\in\mathbb{C}}\sum_{j\in\mathbb{Z}_{\geq 0}} \mathrm{pstr} \big( o(v)|_{W_{\lambda}}\frac{1}{j!}(L_0^N)^j_{\lambda}(\log q)^j \big) q^{\lambda}\\
 &=& q \frac{d}{dq} \mathrm{pstr}_W (o(v) q^{L_0}).
\end{eqnarray*}
\end{proof}

\begin{rem} Note that the symmetry property relies on the fact that $W$ is not just interlocked, but is strongly interlocked.  Thus this condition is precisely what allows us to dispense with much of the machinery in \cite{Miyamoto2004} using a symmetric linear map defined on the higher level Zhu algebras while maintaining the important properties of the graded pseudo-trace in the strongly interlocked setting when they are well defined. 
\end{rem}

Linearity, symmetry, and the logarithmic derivative property, along with the property
\[ [L_0, v_m] = (\wt v - m - 1)v_m, \]
are the conditions analogous to those used in \cite{Z}, in the case of graded traces, and in \cite{Miyamoto2004}, in the case of graded-pseudo traces defined with respect to a map $\phi$, to further develop key properties of graded traces and graded pseudo-traces in the sense of Zhu and Miyamoto, respectively.  

In \cite{BOHY}, we use Theorem \ref{log-thm}  to prove additional properties of the graded pseudo-traces for strongly interlocked modules we defined here, similar to those properties studied in \cite{Z} and \cite{Miyamoto2004} in the setting of graded traces and graded-pseudo traces in the $C_2$-cofinite setting.

In this paper, we will show that all the interlocked modules (in the sense of Definition \ref{interlocked-def}) classified in this work are strongly interlocked, and  their graded pseudo-traces are well defined and satisfy the properties of Theorem \ref{log-thm} as well as these further properties we study in \cite{BOHY}.

\section{The Heisenberg vertex operator algebras and graded pseudo-traces for indecomposable modules}\label{Heisenberg-section}

\subsection{The Heisenberg vertex algebra and associated vertex operator algebras}

We denote by $\mathfrak{h}$ a one-dimensional abelian Lie algebra spanned by $\alpha$ with a bilinear form $< \cdot, \cdot >$ such that $< \alpha, \alpha > = 1$, and by
\[
\hat{\mathfrak{h}} = \mathfrak{h}\otimes \C[t, t^{-1}] \oplus \C \mathbf{k}
\]
the affinization of $\mathfrak{h}$ with bracket relations
\[
[a_m, b_n] = m < a, b>\delta_{m+n, 0}\mathbf{k}, \quad \mbox{for $a, b \in \mathfrak{h}$},
\]
\[
[\mathbf{k}, a_m] = 0,
\]
where we define $a_m = a \otimes t^m$ for $m \in \mathbb{Z}$ and $a \in \mathfrak{h}$.

Set
\[
\hat{\mathfrak{h}}^{+} =\mathfrak{h} \otimes   t\C[t] \qquad \mbox{and} \qquad \hat{\mathfrak{h}}^{-} = \mathfrak{h} \otimes t^{-1}\C[t^{-1}].
\]
Then $\hat{\mathfrak{h}}^{+}$ and $\hat{\mathfrak{h}}^{-}$ are abelian subalgebras of $\hat{\mathfrak{h}}$. Consider the induced $\hat{\mathfrak{h}}$-module given by 
\[
M(1) \ = \ \mathcal{U}(\hat{\mathfrak{h}})\otimes_{\mathcal{U}(\C[t]\otimes \mathfrak{h} \oplus \C c)} \C{\bf 1} \ \simeq \ \mathcal{U}(\hat{\mathfrak{h}}^-).{\bf 1} \ \simeq \ \mathcal{S}(\hat{\mathfrak{h}}^{-}) \qquad \mbox{(linearly)},
\]
where $\mathcal{U}(\cdot)$ and $\mathcal{S}(\cdot)$ denote the universal enveloping algebra and symmetric algebra, respectively, $\mathfrak{h} \otimes \C[t]$ acts trivially on $\mathbb{C}\mathbf{1}$ and $\mathbf{k}$ acts as multiplication by $1$. Then $M(1)$ is a vertex algebra, often called the {\it vertex algebra associated to the rank one Heisenberg}, or the {\it rank one Heisenberg vertex algebra}, or the {\it one free boson vertex algebra}. Here, the Heisenberg Lie algebra in question is precisely $\hat{\mathfrak{h}} \smallsetminus \mathbb{C}\alpha_0$.

Any element of $M(1)$ can be expressed as a linear combination of elements of the form
\begin{equation}\label{generators-for-V}
\alpha_{-n_1}\cdots \alpha_{-n_j}{\bf 1}, \quad \mbox{with} \quad  n_1 \geq \cdots \geq n_j \geq 1.
\end{equation}

The vertex algebra $M(1)$ admits a one-parameter family of conformal elements, each giving $M(1)$ the structure of a vertex operator algebra, given by 
\[ \omega^a = \frac{1}{2} \alpha_{-1}^2{\bf 1} + a \alpha_{-2} \mathbf{1}, \quad \mbox{for  $a \in \mathbb{C}$,} \] 
and then $M(1)$ with the conformal element $\omega^a$ is a vertex operator algebra with central charge $c = 1 - 12 a^2$.  We denote this vertex operator algebra with conformal element $\omega^a$ for $a \in \mathbb{C}$ by $M_a(1)$.   

Writing $Y (\omega^a,x) = \sum_{k \in \mathbb{Z}} L^a_k  x^{-k-2}$ we have that the representation of the Virasoro algebra corresponding to the vertex operator algebra structure $(M_a(1), Y, \mathbf{1}, \omega^a)$ is given as follows:  
For $n$ even
\begin{equation}\label{Ln-even}
L^{a}_{-n} = \sum_{k = n/2 + 1}^\infty \alpha_{-k}\alpha_{-n+k} + \frac{1}{2} \alpha_{-n/2}^2 + a (n-1) \alpha_{-n},
\end{equation} 
and for $n$ odd, we have
\begin{equation}\label{Ln-odd}
L^{a}_{-n} = \sum_{k = (n+1)/2}^\infty \alpha_{-k}\alpha_{-n+k}  + a (n-1) \alpha_{-n} .\end{equation}

As a vertex operator algebra, $M_a(1)$ is simple and not $C_2$-cofinite. In addition, $M_a(1)$ has infinitely many inequivalent irreducible modules which can be easily classified (see \cite{LL}), and infinitely many inequivalent indecomposable non simple modules which we discuss in the next subsection. 

\begin{rem}
A straightforward calculation shows that 
for any choice of conformal vector $\omega^a$, the quotient vector space $M_a(1)/C_1(M_a(1))=span_{\C}\{\vac, \alpha_{-1}\vac\}$ is finite dimensional and thus $M_a(1)$ is $C_1$-cofinite for all $a\in \mathbb{C}$.
\end{rem}

\subsection{Modules for $M_a(1)$}

From \cite{LL}, we have that for every highest weight irreducible $M_a(1)$-module $W$ there exists $\lambda \in \C$ such that
\[
W \cong M_a(1) \otimes_{\mathbb{C}} \Omega_{\lambda},
\]
where $\Omega_{\lambda}$ is the one-dimensional $\mathfrak{h}$-module such that $\alpha_{0}$  acts as multiplication by $\lambda$. We denote these irreducible $M_a(1)$-modules by  $M_a(1, \lambda)= M_a(1)\otimes_{\mathbb{C}} \Omega_{\lambda}$, so that $M_a(1, 0) = M_a(1)$.  If we let $v_\lambda \in \Omega_\lambda$ such that $\Omega_{\lambda} = \C v_{\lambda}$, then for instance,
\begin{eqnarray*}
L^{a}_{-1} v_\lambda &=& \alpha_{-1}\alpha_0 v_{\lambda} \ = \ \lambda \alpha_{-1}v_{\lambda}, \label {L(-1)-on-v-lambda}\\
L^{a}_0 v_\lambda &=& \Big( \frac{1}{2}\alpha_0^2  - a\alpha_0 \Big) v_\lambda \ = \ \Big( \frac{\lambda^2}{2} - a\lambda \Big) v_{\lambda}. \label {L(0)-on-v-lambda}
\end{eqnarray*}

Note that $M_a(1, \lambda)$ is admissible, i.e.,  $\mathbb{Z}_{\geq 0}$-gradable with 
\[M_a(1, \lambda) = \coprod_{\ell \in \mathbb{Z}_{\geq 0}} M_a(1, \lambda)(\ell) \]
and $M_a(1, \lambda)(0) \neq 0$ where $M_a(1, \lambda)(\ell)$ is the eigenspace of eigenvectors of weight $\ell + \frac{\lambda^2}{2} - a\lambda$.  And since $M_a(1)$ is simple these $M_a(1, \lambda)$ are faithful $M_a(1)$-modules. These irreducible modules $M_a(1,\lambda)$ are often called the {\it Fock modules}, and are the unique irreducible highest weight $M_a(1)$-modules generated by $v_\lambda$ satisfying $\alpha_0 v_\lambda = \lambda v_\lambda$.

More generally, the indecomposable modules have been classified, for instance in \cite{Milas}. In particular, we have

\begin{prop}[\cite{Milas}]\label{Milas-prop}  Let $W$ be an indecomposable generalized $M_a(1)$-module.  Then as an $\hat{\mathfrak{h}}$-module  
\[W \cong M_a(1) \otimes \Omega(W) \]
where $\Omega(W) = \{w \in W \; | \; \alpha_n w = 0 \mbox{ for all $n>0$} \}$ is the vacuum space of $W$.   
\end{prop}

\begin{rem}
In terms of the functors $\Omega_n$ for $n \in \mathbb{Z}_{\geq 0}$, defined by Eqn.\ (\ref{Omega-def}), if $W$ is an indecomposable $M_a(1)$-module, then 
\[\mathbf{1} \otimes \Omega(W) = \Omega_0(W) .\]
\end{rem}

In \cite{FZ} it was shown that the level zero Zhu algebra for $M_a(1)$ is given by 
\[A_0(M_a(1)) \cong  \mathbb{C}[\alpha_{-1} {\bf 1}].\]
In \cite{BVY-Heisenberg}, the level one Zhu algebra for $M_a(1)$ was calculated and in  \cite{AB-Heisenberg}, the level two Zhu algebra for $M_a(1)$ was calculated, giving the first example of a level two Zhu algebra for a vertex operator algebra.  
In \cite{AB-Heisenberg} the structure of the level $n$ Zhu algebra was conjectured for all $n\geq 3$, and in \cite{DGK} this conjecture was proved.  In 
particular, in this case of the Heisenberg vertex operator algebra $M_a(1)$, all of the representation theory of $M_a(1)$ is captured by the level zero Zhu algebra $A_0(M_a(1)) \cong \mathbb{C}[x]$ by Proposition  \ref{Milas-prop}, and the results of \cite{BVY-Heisenberg, AB-Heisenberg}, and \cite{DGK}.  Thus this is a setting in which in Theorem \ref{mainthm} the requirement that no nonzero submodule of an $A_n(M_a(1))$-module $U$ factor through $A_{n-1}(M_a(1))$ is superfluous, whereas for other vertex operator algebras, such as for
the universal Virasoro vertex operator algebra, this extra condition is indeed necessary.

\subsection{Explicit structure of the indecomposable modules  for $M_a(1)$ and their graded dimensions}\label{Heisenberg-example}

The indecomposable modules for $A_0(M_a(1)) \cong \mathbb{C}[x]$ for $x$ corresponding to $\alpha_{-1}{\bf 1} + O_0(M_a(1))$ are given by 
\[U(\lambda, k) =  \mathbb{C}[x]/ ((x - \lambda)^k )\]
for $\lambda \in \C$ and $k \in \mathbb{Z}_{>0}$, and 
\[
\mathfrak{L}_0(U(\lambda, k)) \cong M_a(1) \otimes_{\mathbb{C}} \Omega(\lambda, k),
\]
where $\Omega(\lambda,k) \cong U(\lambda, k)$ is a $k$-dimensional vacuum space. Since $\alpha_{-1}{\bf 1} + O_0(M_a(1))$ acts via its zero mode $\alpha_0$ on the module $U(\lambda, k)$ via the action of $x$ on $\mathbb{C}[x]/((x - \lambda)^k)$, we have that  $\alpha_0$  
acts as a $k \times k$ Jordan block given by
\begin{equation}\label{alpha(0)-Jordan-block} 
\alpha_0|_{\Omega(\lambda,k)} = \left[ \begin{array}{cccccc}
\lambda & 1 & 0 & \cdots  & 0 & 0\\
0 & \lambda & 1 & \cdots  & 0 & 0 \\
\vdots & \vdots & \vdots & \ddots & \vdots & \vdots\\
0 & 0 & 0 & \cdots  & \lambda & 1 \\
0 & 0 & 0 & \cdots  & 0 & \lambda
\end{array}
\right]
\end{equation}
with respect to some Jordan basis. 
 
For convenience we introduce the following notation. 

\noindent
{\bf Notation:} Let $D_{m,j}$, $m \in \mathbb{Z}_{>0}$ and $j = 0,\dots, m-1$, be the $m \times m$ matrix with 1's on the $j$th super-diagonal and zeros everywhere else, i.e., 1's in the $(i, i + j)$ position and zeros elsewhere. So $D_{m,0} = I_m$, and $D_{m,m-1}$ is the matrix with a 1 in the $(1,m)$-position and zeros elsewhere.  

Note that 
\[D_{m,i} D_{m,j} = D_{m, i + j} \]
for $0 \leq i,j \leq m$ where $D_{m,r} = 0$ if $r\geq m$.  

For instance, in this notation $\alpha_0|_{\Omega(\lambda,k)} = \lambda I_k + D_{k,1}$.  

Let $w \in \Omega(\lambda, k )$ be such that $\alpha_nw = 0$ for $n >0$, and such that 
\[
(\alpha_0 - \lambda I)^{k-1} w \neq 0; \quad  (\alpha_0 - \lambda I)^{k}w = 0,
\]
for $I$ the identity map. Then a Jordan basis for $\alpha_0$ acting on $\Omega(\lambda,k) \cong U(\lambda, k)$, and realizing the Jordan block (\ref{alpha(0)-Jordan-block}) is given by 
\begin{equation}\label{Heisenberg-jordan-basis} \{u_j = (\alpha_0 - \lambda I_k)^{k-j}w \; | \; j = 1, \dots, k\},
\end{equation}
where $I_k$ is the $k \times k$ identity matrix.
We denote the indecomposable $M_a(1)$-module induced from the $A_0(M_a(1))$-module $U(\lambda, k)$ by $W(a,\lambda, k)$. 
 That is 
 \[ W(a,\lambda,k) = \mathfrak{L}_0(U(\lambda,k)) \cong M_a(1) \otimes \Omega(\lambda,k),\]
 and $W(a,\lambda, k) = M_0(U(\lambda,k))/J$.

\begin{rem} \label{Heis-M=Mbar}
We note that in this case $\overline{M}_0(U(\lambda,k))=M_0(U(\lambda,k))$ due to the fact that $\overline{M}_n(U(\lambda,k))$ is a \textit{restricted} module for the Virasoro Lie algebra in the sense of Remark 5.1.6 in \cite{LL}. Namely, for any $w \in \overline{M}_0(U(\lambda,k))$ we have that for any $a\in \mathbb{C},$ the Virasoro operators defined in Equations (\ref{Ln-even})-(\ref{Ln-odd}) satisfy $L^a_n w=0$ for $n$ sufficiently large. By Theorem 6.1.7 in \cite{LL} this implies that $\overline{M}_0(U(\lambda,k))$ is an $M_a(1)$-module. Thus, in light of Remark \ref{L-a-V-module-remark} $W_A=0$ and 
$\overline{M}_0(U(\lambda,k))=M_0(U(\lambda,k)).$ This shows that $M_0(U(\lambda,k))$ is the universal $M_a(1)$-module with degree zero space $U(\lambda,k)$.
\end{rem}

 In particular, $J = 0$, and $W(a,\lambda, k)$ is the universal $M_a(1)$-module with zero degree space  isomorphic to $U(\lambda,k)$.

 In addition, $W(a,1,1) = M_a(1)$, and the $W(a,\lambda,k)$ for $\lambda \in \mathbb{C}$ and $k \in \mathbb{Z}_{>0}$ give all the indecomposable modules for $M_a(1)$. In particular, $W(a, \lambda, 1)$ give all the irreducible $M_a(1)$-modules.

Note then that the zero mode of $\omega^a$ which is given by \begin{equation}\label{define-L(0)} 
L^a_0 = \sum_{m \in \mathbb{Z}_{>0}} \alpha_{-m} \alpha_m + \frac{1}{2} \alpha_0^2 - a \alpha_0, 
\end{equation}
acts on $\Omega(\lambda, k)$ such that the only eigenvalue is $\frac{1}{2} \lambda^2 - a \lambda$ (which is the lowest conformal weight of  $M_a(1) \otimes \Omega(\lambda, k)$). Hence, $L^a_0$ with respect to a Jordan basis for $\alpha_0$ acting on $\Omega(\lambda, k) = W(0)$ is given by 
\begin{equation}\label{L_0-miyamoto}
L^a_0|_{W(0)} \ = \ \big( \frac{1}{2}  \alpha_0^2 - a \alpha_0\big)|_{W(0)} \ = \  \left( \frac{1}{2} \lambda^2 - a \lambda \right) I_{k} + (\lambda - a)D_{k,1} + \frac{1}{2} D_{k,2},
\end{equation}
with semisimple component $S_0 = L^{a,S}_0|_{W(0)} =  \left( \frac{1}{2} \lambda^2 - a \lambda \right) I_{k}$ and nilpotent component $N_0 = L^{a,N}_0|_{W(0)} = (\lambda - a)D_{k,1} + \frac{1}{2} D_{k,2}$.

There is a typo in the matrix given for the nilpotent component $L^a_0 - ( \frac{1}{2} \lambda^2 - a\lambda) I_{k}$ in \cite{BVY} and \cite{BVY-Heisenberg}. The $(i,i + 2)$-entries were given as $\frac{1}{2} - a$ instead of $\frac{1}{2}$.

Also note that $L_0^a$ is diagonalizable if and only if:  (i) $k = 1$ which corresponds to the case when $M_a(1) \otimes \Omega(\lambda, k) = M_a(1) \otimes \Omega_\lambda$ is irreducible; or (ii) $k = 2$ and $\lambda = a$.

These $M_a(1) \otimes \Omega(\lambda, k)$ exhaust all the indecomposable generalized $M_a(1)$-modules, and the $\mathbb{Z}_{\geq 0}$-grading of $M_a(1) \otimes \Omega(\lambda, k)$ is explicitly given by
\[W(a,\lambda,k) =  M_a(1) \otimes \Omega(\lambda, k) = \coprod_{\ell \in \mathbb{Z}_{\geq 0}} M_a(1)(\ell) \otimes \Omega( \lambda, k) = \coprod_{\ell \in \mathbb{Z}_{\geq 0}} W(a, \lambda, k) (\ell) \]
where $M_a(1)(\ell) = M_a(1)_\ell$ is both the degree $\ell$ space and the weight $\ell$ space of the vertex operator algebra $M_a(1)$.  Thus $W(a,\lambda, k)(\ell) = M_a(1)(\ell) \otimes \Omega( \lambda, k)$ is the space of generalized eigenvectors of degree $\ell$ and
weight $\ell + \frac{1}{2}\lambda^2 - a \lambda$ with respect to $L_0^{ a}$. 

Since each $W(a,\lambda, k)$ is a highest weight module, and in fact $W(a,\lambda,k) = \mathcal{U}(\hat{\mathfrak{h}}^-).\Omega(\lambda, k)$, we have a PBW basis for the action of $\mathcal{U}(\hat{\mathfrak{h}}^-) = \coprod_{\ell \in \mathbb{Z}_{\geq 0} }\mathcal{U}(\hat{\mathfrak{h}}^-)(\ell)$ given by Eqn.\ (\ref{generators-for-V}), so that
\[ \dim M_a(1)_\ell = \dim M_a(1) (\ell) = \mathfrak{p}(\ell)\]
with $\mathfrak{p}(\ell)$ the number of integer partitions of $\ell$. This implies that
\[\dim W(a,\lambda, k)(\ell) = \dim (M_a(1)(\ell) \otimes \Omega(\lambda,k)) = \dim M_a(1)(\ell) \times \dim \Omega(\lambda,k) = \mathfrak{p}(\ell)\times k. \]

Therefore, the generalized graded dimension of $W(a,\lambda, k) = M_a(1) \otimes \Omega(\lambda, k)$ is given by
\begin{eqnarray*}
Z_{W(a,\lambda, k)} ({\bf 1}, \tau)   &=& q^{-c/24} \sum_{\ell \in \mathbb{Z}_{\geq 0}} (\mathrm{dim} \,W(a,\lambda,k)(\ell ))\, q^{\ell + \frac{1}{2} \lambda^2 - a \lambda}\\
&=&  q^{\frac{1}{2} \lambda^2 - a \lambda - (1-12a^2)/24}  \sum_{\ell \in \mathbb{Z}_{\geq 0}} k \mathfrak{p}(\ell) \, q^{\ell} \ = \   q^{\frac{1}{2} (\lambda^2+ a^2) -  a \lambda} k \, \eta(q)^{-1}\\
&=& q^{\frac{1}{2} (\lambda- a)^2} k \, \eta(q)^{-1}
\nonumber
\end{eqnarray*}
where $\eta(q)$ is the Dedekind $\eta$-function
\begin{equation}\label{eta}
\eta(q):=q^{1/24}\prod_{j\in \mathbb{Z}_{\geq 0} }\left(1-q^j\right).
\end{equation}

Therefore, in particular,
\[ Z_{W(a,\lambda,k)}({\bf 1}, \tau)  = k Z_{W(a,\lambda, 1)}({\bf 1}, \tau ) =  kZ_{M_a(1)}({\bf 1}, \tau) .\]
where $Z_{M_a(1)} (\vac,\tau)$ denotes the graded dimension of $M_a(1)$.
More general graded traces for $M_0(1)$ with $v \neq {\bf 1}$ are studied in, for instance \cite{DMN-Heisenberg, MT1, MT2}.

\subsection{Proof that all indecomposable $M_a(1)$-modules are strongly interlocked}

We now apply Theorem \ref{extra-conditions-theorem} to $V = M_a(1)$. We note:

(i) $M_a(1) = \langle \alpha_{-1} {\bf 1} \rangle$.

(ii) $A_0(M_a(1)) \cong \mathbb{C}[\alpha_0]$.

We define the adjoint operator  \begin{eqnarray}
(\cdot)^\dag  \ \ :  \ \ \mathcal{U}(\hat{\mathfrak{h}}) &\longrightarrow& \mathcal{U}(\hat{\mathfrak{h}})\\
\alpha_{n_1} \alpha_{n_2} \cdots \alpha_{n_j}  &\mapsto& \alpha_{-n_j} \alpha_{-n_{j-1}} \cdots \alpha_{-n_1} \nonumber \\
\mathbf{k} &\mapsto& \mathbf{k} \nonumber
\end{eqnarray}
and extending linearly, we define the bilinear form \begin{equation}\label{Heisenberg-form}
< T_\ell {\bf 1}, R_\ell {\bf 1}> = <{\bf 1}, T_\ell^\dag R_\ell {\bf 1}> \qquad \mbox{ for $T_\ell, R_\ell \in \mathcal{U}(\hat{\mathfrak{h}}^-)(\ell)$}
\end{equation}
with $< {\bf 1}, {\bf 1} > = 1$.

Given $R_\ell \in \mathcal{U}(\hat{\mathfrak{h}}^-)(\ell)$, there exists some $T_\ell \in \mathcal{U}(\hat{\mathfrak{h}}^-)(\ell)$ such that $T_\ell^\dag R_\ell u_1 \neq 0$.  Thus this bilinear form is nondegenerate.  Extending this bilinear form to $\langle u_1 \rangle = \mathcal{U}(\hat{\mathfrak{h}}^-).u_1$ for $u_1 \in U(\lambda, k) = \mathbb{C}[\alpha_0]/((\alpha_0 - \lambda)^k)$, such that $\alpha_0u_1 = \lambda u_1$ as in Eqn.\ (\ref{Heisenberg-jordan-basis}), 
we have the nondegenerate bilinear form 
\begin{eqnarray}
< \cdot, \cdot >_\ell \ \ : \ \ \langle u_1 \rangle (\ell)  \otimes \langle u_1 \rangle (\ell) &\longrightarrow& \mathbb{C}\\
\alpha_{-m_1}\cdots \alpha_{-m_i} u_1 \otimes \alpha_{-n_1} \cdots \alpha_{-n_j}u_1 &\mapsto& <\alpha_{-m_1}\cdots \alpha_{-m_i} u_1,  \alpha_{-n_1} \cdots \alpha_{-n_j} u_1 >_\ell \nonumber \\
&& = <u_1, \alpha_{m_i} \cdots \alpha_{m_1} \alpha_{-n_1} \cdots \alpha_{-n_j}u_1 >_0 \nonumber
\end{eqnarray}
with $<u_1, u_1>_0 = 1$.

Thus we have the following Corollary to Theorem \ref{extra-conditions-theorem}:

\begin{cor}\label{extra-conditions-cor}
 The $A_0(M_a(1))$-modules $U(\lambda, k) = \mathbb{C}[x]/((x-\lambda)^k )$  and the $M_a(1)$-modules $W = W(a,\lambda,k) = \mathfrak{L}_0(U(\lambda, k)) = M_0(U(\lambda,k))$ are  strongly interlocked for all $\lambda \in \mathbb{C}$ and $k\in \mathbb{Z}_{>0}$, and $\mathrm{pstr}_W(v, \tau)$ is well defined for every $v \in M_a(1)$. 

In particular, letting $\{u_1, \dots, u_k\}$ be a Jordan basis for $o(\alpha_{-1} {\bf 1}) = \alpha_0$ acting on $U(\lambda,k)$, and setting $W^{(j)} = \mathfrak{L}_0(\mathrm{span}\{u_1, \dots, u_j\})$ for $j = 0,\dots,k$, we have that $W^{(j)}$ is interlocked with $W^{(k-j)}$, and  
$\langle u_1 \rangle = Soc(W(a, \lambda,k))$ is interlocked with  $\langle u_1, \dots, u_{k-1} \rangle = Rad(W(a,\lambda,k))$.

 Furthermore, each $o(v)$ acting on $W(a,\lambda, k)(\ell)$ for any $v \in M_a(1)$ has a $k \mathfrak{p}(\ell) \times k\mathfrak{p}(\ell)$ matrix realization in this basis so that $o(v)|_{W(a,\lambda,k)(\ell)}$ 
can be decomposed as in Eqn.\ (\ref{Our-decomp}), with $A = o(v)|_{Soc(W(a,\lambda, k))(\ell)}$ a $ \mathfrak{p}(\ell) \times \mathfrak{p}(\ell)$ matrix, since $\dim Soc(W)(\ell) = \mathfrak{p}(\ell)$. 
\end{cor}

\subsection{Graded pseudo-traces for indecomposable  $M_a(1)$-modules}

We are now ready to compute some of the graded pseudo-traces for the indecomposable reducible $M_a(1)$-modules. Fix $W = W(a,\lambda, k)$, and for convenience, we let $S_\ell = L^{a,S}_0|_{W(\ell)}$ and
$N_\ell = L^{a,N}_0|_{W(\ell)}$  denote the semisimple and nilpotent parts of $L^a_0|_{W(\ell)}$, respectively, so that
\begin{eqnarray*}
\mathrm{pstr}_W(v, \tau) &=&  \sum_{\ell \in \mathbb{Z}_{\geq 0}} \mathrm{pstr} (o(v)|_{W(\ell)}  \, q^{N_\ell})  q^{S_\ell - c/24} \ = \ q^{\frac{1}{2}( \lambda - a)^2  - 1/24 }  \sum_{\ell \in \mathbb{Z}_{\geq 0}} q^\ell \mathrm{pstr} ( o(v) |_{W(\ell)} \, q^{N_\ell} ) . 
\end{eqnarray*} 

Note that for convenience we are organizing our calculation around the degree $\ell$ of $W(a, \lambda, k)$ rather than the weight, using the fact that the weight spaces and degree spaces are related by a uniform shift, i.e., $W(a, \lambda, k)_{\frac{1}{2} \lambda^2 - a \lambda + \ell} = W(a, \lambda, k)(\ell)$ for $\ell \in \mathbb{Z}_{\geq 0}$. 

To determine $N_\ell$ and then $q^{N_\ell}$, we recall the notation  $D_{m,j}$ introduced in Section \ref{Heisenberg-example} for the $m \times m$ matrix with 1's on the $j$th super-diagonal and zeros everywhere else. 

To determine $N_\ell$ and $q^{N_\ell}$, for each $\ell \in \mathbb{Z}_{\geq 0}$, we first give some low degree examples:

{\bf Degree 0:}  We have $Soc(W)(0) = \mathbb{C}u_1$ and $Rad(W)(0) = \mathrm{span}\{u_1, \dots, u_{k-1}\}$ so that $\{u_1, \dots, u_k\}$ itself is a  strongly interlocked basis, and  $L_0^a|_{W(0)}$ in this basis is given by Eqn.\ (\ref{L_0-miyamoto}).  Thus   $S_0 = (\frac{1}{2} \lambda^2  - a \lambda)I_k$ and $N_0 =  (\lambda - a) D_{k,1} + \frac{1}{2} D_{k,2}$.  Then 
\[N_0^j  \ = \  \sum_{r = 0}^j \binom{j}{r} 2^{-r}  (\lambda - a)^{j-r} D_{k,1}^{j-r} D_{k,2}^r  \ = \   \sum_{r = 0}^j \binom{j}{r} 2^{-r}  (\lambda - a)^{j-r} D_{k,j+r},\]
where these terms are zero if $j + r \geq k$.  
Therefore 
\[q^{N_0}  \ = \  \sum_{j \in \mathbb{Z}_{\geq 0}} \frac{1}{j!}N_0^j (\log q)^j \ = \  \sum_{j =0}^{k-1} \sum_{r = 0}^j\frac{1}{j!}  \binom{j}{r} 2^{-r}  (\lambda - a)^{j-r} D_{k,j+r}  (\log q)^j  \label{N_0}, \]
where terms are zero if $j + r \geq k$.

E.g. If $k = 2$: 
\[q^{N_0} \ = \ I_2 + (\lambda - a) D_{2,1} (\log q)  \ = \ \left[ \begin{array}{cc}
1 & (\lambda - a) \log q \\
 0 & 1
\end{array} \right] \  = \ \left[ \begin{array}{cc}
A & B \\
 0 & A
\end{array} \right]    \]
where $A$ is $1 \times 1$, since $ \mathfrak{p}(0) = 1$.  Then 
\[\mathrm{pstr} \ q^{N_0} = \mathrm{tr}(B) =  (\lambda - a) \log q,\] 
which is the trace of the $D_{k,1}$ term in $q^{N_0}$.

For $k = 3$:
\begin{eqnarray*}
q^{N_0} &=&  I_3 + ((\lambda - a) D_{3,1} + 2^{-1} D_{3,2}) (\log q) + \frac{1}{2} (\lambda - a)^2 D_{3,2} (\log q)^2  \\
&=&  \left[ \begin{array}{ccc}
1 & (\lambda - a) \log q  & \frac{1}{2} \log q + \frac{1}{2} (\lambda - a)^2(\log q)^2 \\
0 & 1 &  (\lambda - a) \log q \\
0 & 0 & 1
\end{array} \right] =   \left[ \begin{array}{ccc}
A & C &B \\
0 & A &  C \\
0 & 0 & A
\end{array} \right]
\end{eqnarray*} 
where again $A$ is 1 x 1. Then 
\begin{eqnarray*}
\mathrm{pstr} \, q^{N_0} = \mathrm{tr}(B) =    \frac{1}{2} \log q + \frac{1}{2} (\lambda - a)^2 (\log q)^2
\end{eqnarray*}
which is the trace of the $D_{3,2}$ term in $q^{N_0}$.

In general, we will have $\mathrm{pstr} \, q^{N_0}  =   \mathrm{tr}(B)$
where $B$ is the 1 x 1 matrix corresponding to the coefficient of $D_{k,k-1}$ in $q^{N_0}$ given by Eqn.\ (\ref{N_0}). Thus
\[ \mathrm{pstr} \, q^{N_0}  = \mathrm{tr} (B) = \sum_{j = 0}^{k-1} \frac{1}{j!} \binom{j}{k-j-1} 2^{-k+j + 1} (\lambda - a)^{2j - k + 1}(\log q)^j.\]

{\bf Degree 1:}  At degree 1, i.e., weight $\frac{1}{2} \lambda^2 - a \lambda + 1$, we have that $M_a(1)(1) = \mathrm{span} \{ \alpha_{-1}{\bf 1} \}$.   Thus $W(1) = \alpha_{-1} {\bf 1} \otimes \Omega (\lambda,k)$, and $\alpha_0|_{W(1)} = \lambda I_k + D_{k,1}$, so that 
\begin{eqnarray*}
L^a_0|_{W(1)} &=& \big( \alpha_{-1} \alpha_1 + \frac{1}{2} \alpha_0^2 - a \alpha_0 \big)|_{W(1)} \ = \ I_k + \frac{1}{2} (\lambda I_k + D_{k,1})^2 - a (\lambda I_k + D_{k,1}) \\
& =& I_k + \frac{1}{2} (\lambda^2I_k + 2\lambda D_{k,1} + D_{k,2}) - a\lambda I_k - aD_{k,1} \\
&=& (\frac{1}{2} \lambda^2 - a \lambda + 1)I_k + (\lambda - a)D_{k,1} + \frac{1}{2} D_{k,2}.
\end{eqnarray*}
Therefore $S_1 =(\frac{1}{2} \lambda^2 - a\lambda +1)I_k$ and $N_1 = (\lambda - a)D_{k,1} + \frac{1}{2} D_{k,2}$. It follows that 
\[
\mathrm{pstr}\, q^{N_1}
=  \sum_{j = 0}^{k-1} \frac{1}{j!} \binom{j}{k-j-1} 2^{-k+j + 1} (\lambda - a)^{2j - k + 1}(\log q)^j. \]

{\bf Degree 2:} Here we have $M_a(1)(2) = \mathrm{span} \{ \alpha_{-1}^2 {\bf 1}, \alpha_{-2}{\bf 1}\}$, so that for a weight-preserving linear operator expressed in terms of a strongly interlocked basis, the corresponding $A$ and $B$ submatrices for $Soc(W)$ and $W/Rad(W)$, respectively, will both be $2 \times 2$ matrices. For instance, a  strongly interlocked  basis for $W(a,\lambda,k)(2)$ is given by 
\[\{\alpha^2_{-1}u_1, \alpha_{-2}u_1, \alpha_{-1}^2u_2, \alpha_{-2}u_2, \dots, \alpha_{-1}^2u_k, \alpha_{-2}u_k\},\]  with the first two terms spanning $Soc(W)_{\lambda + 2} = Soc(W)(2)$ and all but the last two terms spanning $Rad(W)_{\lambda + 2} = Rad(W)(2)$.  Then in this basis, $\alpha_0|_{W(2)} = \lambda I_{2k} + D_{2k,2}$, so that 
\begin{eqnarray*}
L_0^a|_{W(2)} 
&=& \big( \alpha_{-1} \alpha_1 + \alpha_{-2} \alpha_2 + \frac{1}{2}  \alpha_0^2 - a \alpha_0 \big)|_{W(2)}\\
&=& 2I_{2k} + \frac{1}{2} \Big(\lambda^2I_{2k} + 2 \lambda D_{2k,2} + D_{2k,4}\Big) - a\lambda  I_{2k} - a D_{2k,2} \\
&=& \Big(\frac{1}{2} \lambda^2 - a \lambda + 2\Big) I_{2k} + (\lambda -a)D_{2k,2} + \frac{1}{2}D_{2k, 4}. 
\end{eqnarray*}
Thus $N_2 = (\lambda -a)D_{2k,2} + \frac{1}{2}D_{2k, 4}$ and
\begin{eqnarray*}
q^{N_2} &=& \sum_{j = 0}^{2k-1} \frac{1}{j!}  \sum_{r = 0}^j \binom{j}{r} 2^{-r} (\lambda - a)^{j-r} D_{2k,2}^{j-r} D_{2k,4}^r (\log q)^j \\
&=&  \sum_{j = 0}^{2k-1} \frac{1}{j!}  \sum_{r = 0}^j \binom{j}{r} 2^{-r} (\lambda - a)^{j-r} D_{2k,2(j-r)} D_{2k,4r} (\log q)^j\\
&=& \sum_{j = 0}^{2k-1} \frac{1}{j!}  \sum_{r = 0}^j \binom{j}{r} 2^{-r} (\lambda - a)^{j-r} D_{2k,2(j+r)} (\log q)^j.
\end{eqnarray*}

Then the $B$ matrix is the $\mathfrak{p}(2) \times \mathfrak{p}(2) = 2 \times 2$ upper right submatrix of $q^{N_2}$, and the diagonal of this matrix is $I_2$ times the coefficient of $D_{2k,2k-2}$ in $q^{N_2}$. 
Therefore 
\begin{eqnarray*}
\mathrm{pstr}\, q^{N_2} &=& \mathrm{tr} ( B) \ = \ \mathrm{tr}\Bigg(\bigg(  \sum_{j = 0}^{2k-1} \frac{1}{j!} \binom{j}{k-j-1} 2^{-k+j + 1} (\lambda - a)^{2j - k + 1}(\log q)^j \bigg) I_2 \Bigg) \\
&=&  2 \sum_{j = 0}^{2k-1} \frac{1}{j!} \binom{j}{k-j-1} 2^{-k+j + 1} (\lambda - a)^{2j - k + 1}(\log q)^j, 
\end{eqnarray*}
where in fact the sum is nontrivial only for $0\leq j \leq k-1$.

Continuing in this manner, we see that more generally, for degree $\ell \in \mathbb{Z}_{\geq 0}$:

{\bf Degree $\ell$:} Here we have $\dim M_a(1)(\ell) = \mathfrak{p}(\ell)$ 
so that for a grade-preserving linear operator expressed in terms of a  strongly interlocked basis, the corresponding $A$ and $B$ submatrices for $Soc(W)$ and $W/Rad(W)$, respectively, will both be $\mathfrak{p}(\ell) \times \mathfrak{p}(\ell)$.  So for instance, using lexicographical ordering on the basis for $M_a(1)(\ell)$, in the corresponding basis for $W(\ell) = M_a(1) (\ell) \otimes \Omega(\lambda,k) $ we have  $\alpha_0|_{W(\ell)} = \lambda I_{k\mathfrak{p}(\ell)} + D_{k\mathfrak{p}(\ell),\mathfrak{p}(\ell)}$, so that
\begin{eqnarray*}
L_0^a|_{W(\ell)} &=& \ell I_{k\mathfrak{p}(\ell)} + \frac{1}{2} \left( \lambda^2 I_{k\mathfrak{p}(\ell)} + 2 \lambda D_{k \mathfrak{p}(\ell), \mathfrak{p}(\ell)} + D_{k\mathfrak{p}(\ell), 2\mathfrak{p}(\ell)} \right)- a\lambda I_{k \mathfrak{p}(\ell)} - aD_{k \mathfrak{p}(\ell), \mathfrak{p}(\ell)}\\
&=& \Big(\frac{1}{2}\lambda^2 - a \lambda + \ell\Big)I_{k \mathfrak{p}(\ell)} + (\lambda - a)D_{k \mathfrak{p}(\ell), \mathfrak{p}(\ell)} + \frac{1}{2} D_{k \mathfrak{p}(\ell), 2\mathfrak{p}(\ell)}.
\end{eqnarray*}
Thus $N_\ell = (\lambda -a)D_{k\mathfrak{p}(\ell),\mathfrak{p}(\ell)} + \frac{1}{2}D_{k\mathfrak{p}(\ell),2\mathfrak{p}(\ell)}$ and 
\begin{eqnarray}\label{Nell}
q^{N_\ell} &=& \sum_{j = 0}^{k\mathfrak{p} (\ell)-1} \frac{1}{j!}  \sum_{r = 0}^j \binom{j}{r} 2^{-r} (\lambda - a)^{j-r} D_{k\mathfrak{p}(\ell),\mathfrak{p}(\ell)}^{j-r} D_{k\mathfrak{p}(\ell),2 \mathfrak{p}(\ell)}^r (\log q)^j \\
&=& \sum_{j = 0}^{k\mathfrak{p}(\ell)-1} \frac{1}{j!}  \sum_{r = 0}^j \binom{j}{r} 2^{-r} (\lambda - a)^{j-r} D_{k\mathfrak{p}(\ell),(j-r)\mathfrak{p}(\ell)} D_{k\mathfrak{p}(\ell),2r\mathfrak{p}(\ell)} (\log q)^j \nonumber \\
&=& \sum_{j = 0}^{k\mathfrak{p}(\ell)-1} \frac{1}{j!}  \sum_{r = 0}^j \binom{j}{r} 2^{-r} (\lambda - a)^{j-r} D_{k\mathfrak{p}(\ell),(j+r)\mathfrak{p}(\ell)} (\log q)^j. \nonumber 
\end{eqnarray}
Then the $B$ matrix is the $\mathfrak{p}(\ell) \times \mathfrak{p}(\ell)$ upper right submatrix of $q^{N_\ell}$, and the diagonal of this matrix is $I_{\mathfrak{p}(\ell)}$ times the coefficient of $D_{k\mathfrak{p}(\ell),(k-1)\mathfrak{p}(\ell)}$ in $q^{N_\ell}$. 
Therefore 
\begin{eqnarray*}
\mathrm{pstr}\, q^{N_\ell} &=& \mathrm{tr} ( B) \  = \  \mathrm{tr}\left(\Bigg(  \sum_{j = 0}^{k\mathfrak{p}(\ell)-1} \frac{1}{j!} \binom{j}{k-j-1} 2^{-k+j + 1} (\lambda - a)^{2j - k + 1}(\log q)^j \Bigg) I_{\mathfrak{p}(\ell)} \right) \\
&=&  \mathfrak{p}(\ell) \sum_{j = 0}^{k\mathfrak{p}(\ell)-1} \frac{1}{j!} \binom{j}{k-j-1} 2^{-k+j + 1} (\lambda - a)^{2j - k + 1}(\log q)^j. 
\end{eqnarray*}

Thus in general, we have that if $W=W(a,\lambda, k)$ then
\begin{eqnarray}
\lefteqn{ \mathrm{pstr}_{W}({\bf 1}, \tau) \ = \ \mathrm{pstr}_W(q^{L_0 - c/24}) 
 \ = \ \sum_{\ell \in \mathbb{Z}_{\geq 0}} \mathrm{pstr} (q^{N_\ell}) q^{S_\ell - c/24}  }\\
&=&  q^{\frac{1}{2} (\lambda - a )^2 -1/24} \sum_{\ell \in \mathbb{Z}_{\geq 0}} q^\ell \Bigg( \sum_{j = 0}^{\mathfrak{p}(\ell) k-1} \frac{\mathfrak{p}(\ell)}{j!} \binom{j}{k-j-1} 2^{-k+j + 1} (\lambda - a)^{2j - k + 1}(\log q)^j\Bigg). \nonumber
\end{eqnarray}

But note that the binomial coefficient is zero unless $\frac{k-1}{2} \leq j \leq k-1$. Thus we have the following theorem.

\begin{thm}\label{Heisenberg-pseudo-trace-thm} Let $W = W(a, \lambda, k) = M_a(1) \otimes \Omega(\lambda, k)$ be the $M_a(1)$-module induced from the $A(M_a(1)) = \mathbb{C}[x]$-module $U = \mathbb{C}[x]/((x-\lambda)^k)$, for fixed $k \in \mathbb{Z}_{>0}$ and $\lambda, a \in \mathbb{C}$.  Then the graded pseudo-trace of $W$  associated to the vacuum, i.e. the 0-point function, is given by 
\begin{eqnarray*}
\lefteqn{\mathrm{pstr}_{W}({\bf 1}, \tau) } \\
&=&  q^{\frac{1}{2} (\lambda-a)^2 -1/24} \sum_{\ell \in \mathbb{Z}_{\geq 0}} q^\ell \mathfrak{p}(\ell)\Bigg( \sum_{j = \lceil \frac{k-1}{2} \rceil}^{k-1} \frac{1}{j!} \binom{j}{k-j-1} 2^{-k+j + 1} (\lambda - a)^{2j - k + 1}(\log q)^j\Bigg)  \nonumber \\
&=&  q^{\frac{1}{2} (\lambda- a)^2 }\eta(q)^{-1}\Bigg( \sum_{j = \lceil \frac{k-1}{2} \rceil}^{k-1} \frac{1}{j!} \binom{j}{k-j-1} 2^{-k+j + 1} (\lambda - a)^{2j - k + 1}(\log q)^j\Bigg)  .  \nonumber
\end{eqnarray*}
where $\lceil \frac{k-1}{2} \rceil$ denotes the closest integer greater than or equal to $\frac{k-1}{2}$, and formally $\log q = 2 \pi i \tau$.  
\end{thm}

\begin{cor} \label{cor:vanish}
For all $a\in \mathbb{C}$ and $m\in \mathbb{Z}_{\geq 1} $,  the $M_a(1)$ module $W(a,a,2m)=M_a(1) \otimes \Omega(a, 2m)$ induced from the $A(M_a(1)) = \mathbb{C}[x]$-module $U = \mathbb{C}[x]/((x-a)^{2m})$ has trivial graded pseudo-trace associated to the vacuum vector.
\end{cor}

Next, we calculate another graded pseudo-trace for the $M_a(1)$-module $W(a,\lambda,k)$ for $a, \lambda \in \mathbb{C}$ and $k \in \mathbb{Z}_{>0}$, and then also  note that by the logarithmic derivative property Theorem \ref{log-thm} (3), we can easily determine $\mathrm{pstr}_W(\omega, \tau)$. 

{\bf Graded pseudo-trace for the strong generator -- $\mathrm{pstr}_W (\alpha_{-1}{\bf 1},\tau)$:} To calculate $\mathrm{pstr}_W (\alpha_{-1}{\bf 1},\tau)$, we first need to calculate the pseudo-traces of $o(\alpha_{-1}{\bf 1})|_{W(\ell)} q^{N_\ell}$ for each $\ell \in \mathbb{Z}_{\geq 0}$. From Eqn.\ (\ref{Nell}), we have 
\begin{eqnarray*}
\lefteqn{o(\alpha_{-1}{\bf 1})|_{W(\ell)} q^{N_\ell} \ = \  \alpha_0|_{W(\ell)} q^{N_\ell}}\\
&=&  ( \lambda I_{k\mathfrak{p}(\ell)} + D_{k\mathfrak{p}(\ell), \mathfrak{p}(\ell)})\sum_{j = 0}^{k\mathfrak{p}(\ell)-1} \frac{1}{j!}  \sum_{r = 0}^j \binom{j}{r} 2^{-r} (\lambda - a)^{j-r} D_{k\mathfrak{p}(\ell),(j+r)\mathfrak{p}(\ell)} (\log q)^j \\
&=& \sum_{j = 0}^{k\mathfrak{p}(\ell)-1} \frac{1}{j!}  \sum_{r = 0}^j \binom{j}{r} 2^{-r} (\lambda - a)^{j-r} \left( \lambda D_{k\mathfrak{p}(\ell),(j+r)\mathfrak{p}(\ell)} +  D_{k\mathfrak{p}(\ell),(j+r + 1)\mathfrak{p}(\ell)}\right) (\log q)^j.
\end{eqnarray*}
The pseudo-trace of this is the trace of the matrix $B$ given by the upper right hand $\mathfrak{p}(\ell) \times \mathfrak{p}(\ell)$ matrix. Thus we are interested in the coefficient of $D_{k\mathfrak{p}(\ell),  (k-1)\mathfrak{p}(\ell)}$.  

It follows that 
\begin{eqnarray*}
\lefteqn{\mathrm{pstr}_W (\alpha_{-1}{\bf 1},\tau) \  = \  \sum_{\ell \in \mathbb{Z}_{\geq 0}} \mathrm{pstr} (\alpha_0 |_{W(\ell)} q^{N_\ell}) q^{S_\ell - c/24} } \\
&=& q^{\frac{1}{2} (\lambda - a)^2}  \eta(q)^{-1} \Bigg(  \sum_{j = \lceil{\frac{k-1}{2}}\rceil}^{k-1} \frac{\lambda}{j!} \binom{j}{k-j-1} 2^{-k+j + 1} (\lambda - a)^{2j - k + 1}(\log q)^j\\
& & \quad  + \sum_{j = \lceil\frac{k-2}{2} \rceil }^{k-2} \frac{1}{j!} \binom{j}{k-j-2} 2^{-k+j + 2} (\lambda - a)^{2j - k + 2}(\log q)^j\Bigg).
\end{eqnarray*}

{\bf Graded pseudo-trace of the conformal vector -- $\mathrm{pstr}_W  (\omega^a,\tau)$:}  As an application of Theorem \ref{log-thm} (3), we have that
\begin{eqnarray}
\mathrm{pstr}_W(\omega^a, \tau) \ = \ q^{ -(1-12a^2)/24} q \frac{d}{dq} q^{ (1-12a^2)/24} \mathrm{pstr}_W({\bf 1} , \tau)\nonumber.
\end{eqnarray} 

In future work, we plan to extend the results of, for instance, \cite{DMN-Heisenberg, MT1, MT2} on graded traces for $n$-point functions to the graded pseudo-trace setting.

\section{The universal Virasoro vertex operator algebras--The setting}

In this Section we recall the relevant facts about the Virasoro universal vertex operator algebras $V_{Vir}(c,0)$ for $c \in \mathbb{C}$ necessary to analyze the indecomposable modules that are generated from their level zero Zhu algebras, which will comprise Sections 6 and 7. 

We start by recalling 
from \cite{FF, FF2, F, AF, As}, the relevant facts about the Virasoro algebra highest weight modules, which have been extensively studied in the literature and will be used heavily in Sections 6 and 7 where we give our main results. In particular, we recall important results regarding the Verma modules $M(c,h)$ for $c,h \in \mathbb{C}$, their maximal submodules $T(c,h)$, and the singular vectors that generate them. We discuss the Shapavolov form and the determinant of the Gram matrix in terms of the curves $\Phi_{r,s}(c,h)$ in $\mathbb{C}^2$ as defined in \cite{FF}.

\subsection{The universal Virasoro vertex operator algebra, its level zero Zhu algebra, and Verma modules}

Let $\mathcal{L}$ be the Virasoro algebra with central charge $\mathbf{c}$, that is, $\mathcal{L}$ is the vector space with basis $\{L_n \,|\, n\in\Z\}\cup \{\mathbf{c}\}$ and with bracket relations 
\begin{align*}
[L_m,L_n]=(m-n)L_{m+n}+\frac{m^3-m}{12} \delta_{m+n,0} \, \textbf{c},\quad\quad  [\textbf{c},L_m]=0  
\end{align*}  
for $m,n\in \Z$.  

Let $\mathcal{L}^{\geq 0}$ be the Lie subalgebra with basis $\{ L_n  \,|\, n\geq 0 \} \cup \{\mathbf{c}\}$, and let $\mathbb{C}_{c,h}$ be the $1$-dimensional $\mathcal{L}^{\geq 0}$-module where $\mathbf{c}$ acts as $c$ for some $c\in \mathbb{C}$, $L_0$ acts as $h$ for some $h\in \mathbb{C}$, and $L_n$ acts trivially for $n\geq 1$. The Verma module is the induced $\mathcal{L}$-module
\begin{equation}\label{V-module-M}
M(c,h)= \mathcal{U}(\mathcal{L})\otimes_{ \mathcal{U}(\mathcal{L}^{\geq 0})} \mathbb{C}_{c,h} .
\end{equation}

\begin{defn}
 Denote the image of $L_n$ in $M(c,h)$ by $L_n$ again and let $\mathbf{1}_{c,h} := 1 \in \mathbb{C}_{c,h}$. Define 
\[ V_{Vir}(c,0):= M(c,0)/\langle L_{-1}\mathbf{1}_{c,0}\rangle.
\] 

It was proved by Frenkel and Zhu in \cite{FZ} that for any $c\in \mathbb{C}$ the space $V_{Vir}(c,0)$ has a natural vertex operator algebra structure with vacuum vector $1=\mathbf{1}_{c,0}$, and conformal element $\omega=L_{-2}\mathbf{1}_{c,0}$, satisfying  $Y(\omega,x)= \sum_{n\in\Z } L_nx^{-n-2}$.  In addition, for each $h\in\C$, we have that $M(c,h)$ is an ordinary $V_{Vir}(c,0)$-module with $\mathbb{Z}_{\geq 0}$-gradation 
\begin{equation}\label{grade-M}
M(c,h)=\coprod_{\ell \in \mathbb{Z}_{\geq 0}} M(c,h)(\ell)
\end{equation}
where $M(c,h)(\ell)$ is the $L_0$-eigenspace with eigenvalue $\ell+ h$, so that  $M(c,h)(\ell)$ has degree $\ell$ and weight $\ell + h$.  

The vertex operator $V_{Vir}(c,0)$ is called the \emph{universal Virasoro vertex operator algebra} of central charge $c$. 
\end{defn}

The simple quotient  of $M(c,0)$, denoted $L(c,0)$, is a quotient of $V_{Vir}(c,0)$ 
 and therefore, it  admits a vertex operator algebra structure as well. Note that all $L(c,0)$- and $V_{Vir}(c,0)$-modules are $\vir$-modules.  Another important result of \cite{FZ} is that every  $\vir$-module of central charge $c$ is a $V_{Vir}(c,0)$-module.

Recall, following e.g., \cite{LL, KarL}, that $V_{Vir}(c,0)$ is spanned by vectors of the form
\begin{equation}\label{span}
L_{-n_1} \cdots L_{-n_m} \mathbf{1}_{c,0} \qquad \mbox{for $n_1 \geq \cdots \geq n_m \geq 2$ and $m \in \mathbb{Z}_{\geq 0}$} .
\end{equation}

It was shown in \cite{W} that
\begin{equation}\label{level-zero-iso}
A_0(V_{Vir}(c,0)) \cong \mathbb{C}[x]
\end{equation}
for $x = L_{-2} \mathbf{1}_{c,0} + O_0(V_{Vir}(c,0))$.

By \cite{FZ}, there is a bijection between isomorphism classes of irreducible $\mathbb{Z}_{\geq0}$-gradable $V_{Vir}(c,0)$-modules and irreducible $\mathbb{C}[x]$-modules which is given by $L(c,h)
\leftrightarrow \mathbb{C}[x]/(x - h)$ where $T(c,h)$ is the largest proper submodule of $M(c,h)$, and $L(c,h) = M(c,h)/T(c,h)$.   
For any central charge $c\in \C$,  $V_{Vir}(c,0)$ is not $C_2$-cofinite and in fact admits infinitely many nonisomorphic irreducible modules. However, a straightforward computation shows $V_{Vir}(c,0)/C_1(V_{Vir}(c,0))=span_{\C}\{ {\bf 1}_{c,0}, \omega\}$ and thus $V_{Vir}(c,0)$ is $C_1$-cofinite for any $c\in\C$.

It was proved in \cite{FF} (see also \cite{W}) that if $c$ is not of the form 
\begin{equation}\label{cpq}
c_{p,q} = 1 - 6\frac{(p-q)^2}{pq} \quad \mbox{for any $p,q \in \{2,3,4,\dots\}$ with $p$ and $q$ relatively prime,}
\end{equation}
then $T(c, 0) = \langle L_{-1} \mathbf{1}_{c,0} \rangle$, implying $V= V_{Vir}(c,0) = L(c,0)$, is simple as a vertex operator algebra.  
It was also shown in \cite{W, FF}  that for $c = c_{p,q}$, $T(c,  0) \neq \langle L_{-1} \mathbf{1}_{c,0} \rangle$ and thus in this case $V_{Vir}(c_{p,q}, 0)$ is not simple as a vertex operator algebra.

\begin{defn} 
${}$

\begin{itemize}
    \item[(i)]  A non zero vector $v\in M(c,h)$ is called a {\it singular} vector if $L_nv=0$ for $n>0$.
    \item[(ii)]  A non zero vector $v\in M(c,h)$ is a {\it conformal singular vector} if it is singular and satisfies $L_0v=(\ell+h) v$ for some $\ell \in \mathbb{Z}_{\geq 0}$. 
\end{itemize}
\end{defn} 
 Note that if $v$ is a conformal singular vector with $L_0v=(\ell+h) v$ then, $\mathrm{Vir}(c,0).v$ is a proper submodule isomorphic to $M(c,h+\ell)$. Since $M(c,h)_0$ is $1$-dimensional, the sum of all proper submodules is again a proper submodule and so $M(c,h)$ has a maximal proper submodule $T(c,h)$, and the quotient $L(c,h)= M(c,h)/T(c,h)$ is irreducible. In general either $T(c,h)=0$ or $T(c,h)$ is generated by up to $2$ conformal singular vectors \cite{FF}. We recall important results on conformal singular vector formulas following the exposition in,  for instance \cite{FF} and \cite{AF}, however, we note the following.
 
 \begin{rem} In \cite{FF}, the notation is in terms of $e_i = -L_i$ and is a study of Verma modules and thus highest weight modules, i.e., with only a finite number of nontrivial positive weight  spaces, whereas we are interested in the vertex operator algebra structure, i.e., ``positive energy modules" for the Virasoro, and thus with only a finite number of nontrivial negative weight spaces, or in other words, lowest weight modules.  Thus our results are in terms of $L_{-i}$ instead of $L_i$.  The minus signs in $-e_{-i} = L_{-i}$  cancel each other, and thus the analysis and in particular the singular vectors theory introduced below is identical.
 
 In \cite{AF}, the notation is stated to be that $e_i = - L_{-i}$, but this notation results in their Lie algebra structure being the opposite Lie algebra bracket structure for the usual Virasoro algebra. Nevertheless, their singular vector results hold in our notational setting.  
 \end{rem}

\begin{defn}\label{define-c-h}
For $r,s\in \mathbb{Z}_{>0}$ and $t \in \mathbb{C}^\times$, define
\begin{equation} \label{eq:virpara}
 c=c(t)=13 + 6t + 6t^{-1}, \qquad
h=h_{r,s}(t)=\frac{1-r^2}{4}t + \frac{1-rs}{2}+\frac{1-s^2}{4}t^{-1},
\end{equation}
and let $\Phi_{r,s}(c,h)$ denote the curve in the $\mathbb{C}^2(c,h)$ plane defined by these parametric equations.
\end{defn}
 Note that when $r=s$ the curve $\Phi_{r,s}(c,h)$ is the line $h=\frac{c-1}{24}(1-r^2)$. Note also that $\Phi_ {r,s}(c,h)=\Phi_{s,r}(c,h)$ and that otherwise all the curves $\Phi_{r,s}(c,h)$ are different.  In fact, in \cite{FF}, these curves are given (once a typo is fixed) by 
 \begin{multline}\label{Phi-r-s}
\Phi_{r,s}(c,h) = \Big( h + \frac{1}{24}(r^2 - 1) (c-13) + \frac{1}{2} (rs - 1) \Big)\Big( h + \frac{1}{24}(s^2 - 1) (c-13) \\
+ \frac{1}{2} (rs - 1) \Big) 
 + \frac{(r^2 - s^2)^2}{16}. 
 \end{multline}
As noted in \cite{FF}, once $c$ is fixed to be of the form $c(t)=13+6t+6t^{-1}$ for $t\in \mathbb{C}^\times$, the roots with respect to $h$ of the equation $\Phi_{r,s}(c,h)=0$ are given by $h_{r,s}(t)$ and $h_{r,s}(t^{-1})$.

 \begin{thm}[\cite{RW1, RW2, K, FF2}] \label{reducible}
 
The module $M(c,h)$ is reducible if and only if $(c,h)$ belongs to the union of the curves $\Phi_{r,s} (c,h)$ with $r,s>0$. If $(c,h)\in\Phi_{r,s}(c,h)$ with $r,s>0$ and $(c,h)\notin \Phi_{r',s'}(c,h)$ for any $r',s'>0$ such that $r's'<rs$ then $M(c,h)$ has a singular vector in $M(c,h)(rs)$ and no singular vectors in $M(c,h)(\ell)$ with $\ell<rs$.
\end{thm}

Next, we recall the following results about singular vectors in reducible Verma modules following the exposition in \cite{AF}.

\begin{prop}\label{singular-prop} \cite{F}
The Verma module $M(c(t), h_{r,s}(t))$ contains a singular vector of the form $S_{r,s}(t)\vac_{c,h}$ where $r\geq s$ and
\begin{align}\label{S}
S_{r,s}(t)=L_{-1}^{rs}+\sum_{I\in \mathcal{I}_{r,s}} P_{r,s}^{I}(t)L_{-I},   
\end{align}
 where  $\mathcal{I}_{r,s}=\{(i_1, \dots, i_k) \in \mathbb{Z}_{>0}^k \; | \;  k \in \mathbb{Z}_{>0}, \ (i_1, \dots, i_k)\neq (1,\dots, 1) \ \mbox{and} \   i_1+\cdots + i_k=rs\}$,   $P_{r,s}^{I}(t)$ are polynomials in $t$, and $L_{-I} = L_{-i_1} \cdots L_{-i_k}$ for $I = (i_1, \dots, i_k) \in \mathcal{I}$.

\begin{rem}
Note that we fixed $r\geq s$ in the expression above to account for the symmetry in the conformal weights $h_{r,s}=h_{s,r}$ when $r$ and $s$ are exchanged.
\end{rem}

\end{prop}
A proof of this fact is contained in \cite{F}. Moreover, we have the following important property of conformal singular vectors as shown in \cite{FF}.

\begin{prop}\cite{FF} \label{propcases}
Any conformal singular vector in $M(c,h)$ is a non zero vector proportional to either $S_{r,s}(t)\vac_{c,h}$ or to $S_{\tilde{r},\tilde{s}}(t)S_{r,s}(t)\vac_{c,h}$ 
 with $(c,h)\in\Phi_{r,s}(c,h)$ and $(c, h + rs)\in \Phi_{\tilde{r},\tilde{s}}(c, h + rs)$, for some $r,s, \tilde{r}, \tilde{s} \in \mathbb{Z}_{>0}$.

In particular, letting $T(c,h)$ denote the maximal proper submodule of $M(c,h)$, we have three cases:

Case (0):  $M(c,h)$ has no  singular vectors, and $L(c,h) = M(c,h)/T(c,h) = M(c,h)$.

Case (1):  $T(c,h)$ is generated by a singular vector at some degree $d$ of the form $S_{r,s}(t) {\bf 1}_{c,h}$ with $r,s>0$ and $rs = d$, so that $L(c,h) = M(c,h)/ \langle S_{r,s}(t) {\bf 1}_{c,h}\rangle$. 

Case (2): $T(c,h)$ is generated by exactly two singular vectors at some degrees $d$ and $d' \neq d$, and the singular vectors are of the form $S_{r,s}(t) {\bf 1}_{c,h}$ and $S_{r',s'}(t) {\bf 1}_{c,h}$, respectively, with $r,s, r', s'>0$ and $rs = d$ and $r's'=d'$, so that \[L(c,h) = M(c,h)/ \langle S_{r,s}(t) {\bf 1}_{c,h}, S_{r',s'}(t){\bf 1}_{c,h} \rangle.\] 
\end{prop}

\begin{rem} 
In \cite{FF} Feigin and Fuchs divide these cases in a different way. The correspondence with our notation is as follows
\begin{center}
\begin{tabular}{ |c|c|c|c|c| } 
\hline
 Case &Case 0 & Case 1 & Case 2 \\
\hline
\multirow{3}{4em}{Feigin and Fuchs subcases involved } & $I$, &$II_+, $& $III_{-}, $ \\ 
& $II_{0}$, & $III^{00}_{-}, $ & $III_{+}. $\\ 
& $II_{-}$. & $III^{0}_{-}, $ & \ \\ 
& \ & $III^{00}_{+}, $& \ \\
& \ & $III^{0}_{+}. $ & \ \\
\hline
\end{tabular}
\end{center}
\end{rem}

The conformal singular vectors can be computed by hand for small values of $r$ and $s$ (see  \cite{FF, AF}):
 \begin{align*}
    &S_{1,1}(t)=L_{-1}\\
    &S_{2,1}(t)=L_{-1}^2+ tL_{-2} \\ 
    &S_{3,1}(t)=L_{-1}^3+4tL_{-2}L_{-1}+(4t^2 + 2t)L_{-3}\\
    &S_{4,1}(t)=L_{-1}^{4}  +10tL_{-2}L_{-1}^2+9t^2L_{-2}^2+(24t^2 + 10t)L_{-3}L_{-1}+ (36t^3+24t^2+ 6t)L_{-4}\\
    &S_{2,2}(t)= L_{-1}^4+2(t+t^{-1})L_{-2}L_{-1}^2+((t+t^
    {-1})^2-4)L_{-2}^2\\
    &\ \ \quad  \quad \quad \quad \ \ \ \ \ +( 2(t+t^{-1})+6)L_{-3}L_{-1} +(3(t+t^{-1})+6)L_{-4}.
 \end{align*}

 More generally, we have the following formula for $S_{r,s}(t)$ proved in \cite{AF}.
 
 \begin{prop} [Theorem 1.2. \cite{AF}]  The polynomial coefficients in the singular vector formula (\ref{S}) are given by
 \begin{align}
 S_{r,s}(t)=(r-1)!^{2s}L_{-r}^st^{(r-1)s}+\cdots + (s-1)!^{2r}L_{- s}^rt^{-(s-1)r}, 
 \end{align}
 where $\cdots$ denotes the  terms of intermediate degrees in $t$.
 \end{prop}

\begin{rem}\label{remark-L(-1)}
Note that any conformal singular vector in $M(c,h)$ may be expressed as a linear combination of Poincar\'{e}--Birkhoff--Witt-ordered monomials in the $L_i$, $i<0$, acting on the  vector $\vac_{c,h}$. A crucial fact for what follows is that the coefficient of $L_{-1}^{rs}$, is never $0$ (irrespective of the chosen order). Here, $rs$ is the degree of the conformal singular vector which is the conformal weight of the conformal singular vector minus that of $\vac_{c,h}$. 
\end{rem}

\begin{defn}\label{Hc}
For fixed $c = c(t) \in \mathbb{C}$ as in Eqn.\ (\ref{eq:virpara}), set
\begin{equation}\label{H_c-notation}
H_c = \left\{h \mid M(c,h)\ \text{is reducible}\right\}.
\end{equation}
By Theorem \ref{reducible}, we have $H_c = \{h_{r,s}(t) \mid r,s \in \mathbb{Z}_{>0}\}$ for $h_{r,s}(t)$ given by Eqn.\ (\ref{eq:virpara}).

The set $H_c$ of conformal weights is often referred to as the \emph{extended Kac table}, and more generally if $(c,h) \in \mathbb{C}^2$ are such that $H_c \neq \emptyset$ then $(c,h)$ is said to be in the extended Kac table.
\end{defn}

 Note that the term ``exteded Kac table" is used  because the original Kac table pertains to the Virasoro minimal models $L(c_{p,q}, 0)$. More precisely, the minimal models are the irreducible vertex operator algebras $L(c,0) = M(c,0)/T(c,0)$ which are a (nontrivial) quotient of $V_{Vir}(c_{p,q},0) = M(c_{p,q},0)/\langle L_{-1} {\bf 1}_{c_{p,q},0} \rangle$ for these particular central charges $c=c_{p,q}$. In this case, the Kac table $\bar{H}_{r,s}$ lists the irreducible representations of these rational theories and is given by the subset of $h_{r,s}(t)$ with $t=-p/q$, for $p,q \in \mathbb{Z}_{\geq2}$ and $\gcd\{p,q\}=1$, $r=1,\dots, q-1$, and $s=1,\dots, p-1$.

 In particular, we have the following Corollary to Proposition \ref{singular-prop} and Remark \ref{remark-L(-1)}.
 
\begin{cor} \label{cor_rep_L(-1)}
If $(c,h)$ is in the extended Kac table, i.e.,  $M(c,h) \neq L(c,h)$ and we are in the setting of Case (1) or Case (2), and $S_{r,s}(t) {\bf 1}_{c,h}$ is a singular vector at degree $rs=d$, then $L_{-1}^d {\bf 1}_{c,h}$ can be expressed as a linear combination of terms of the form $L_{-d}^{j_d} L_{-d+1}^{j_{d-1}} \cdots L_{-2}^{j_2} L_{-1}^{j_1}{\bf 1}_{c,h}$ for $j_1, \dots, j_d \in \mathbb{Z}_{\geq 0}$ and with $j_1 <d$ such that $j_1 + 2j_2 + \cdots d j_d = d$. In other words, $L_{-1}^d {\bf 1}_{c,h}$ can be expressed as a linear combination of terms of the form $L_{-n_1} L_{-n_2} \cdots L_{-n_i}$ such that $(n_1, \dots , n_i)$ is a partition of $rs=d$ that does not involve the partition of $d$ parts that are equal to 1, i.e., $(n_1, n_2, \dots, n_i) = (n_1, n_2, \dots, n_d) = (1,1,\dots, 1)$.
\end{cor}

\subsection{The Shapovalov form and the determinant of the Gram matrix}\label{shapovalov-subsection} Many of the results above rely on the analysis of the Shapovalov form for the Verma modules $M(c,h)$ of the Virasoro algebra. In addition, the Shapovalov form, and extensions of it to broader classes of modules, will be used in our analysis of the indecomposable modules induced from the level zero Zhu algebra for $V_{Vir}(c, 0)$.  

The universal enveloping algebra for the Virasoro algebra, $\mathcal{U}(\mathcal{L})$ has an adjoint operator 
\begin{eqnarray}
(\cdot)^\dag  \ \ :  \ \ \mathcal{U}(\mathcal{L}) &\longrightarrow& \mathcal{U}(\mathcal{L})\\
L_{n_1} L_{n_2} \cdots L_{n_m}  &\mapsto& L_{-n_m} L_{-n_{m-1}} \cdots L_{-n_1} \nonumber \\
\mathbf{c} &\mapsto& \mathbf{c} \nonumber
\end{eqnarray}
extended linearly.  
Writing $\mathcal{L} = \mathcal{L}^- \oplus \mathbb{C}L_0 \oplus \mathcal{L}^+$ with $\mathcal{L}^\pm = \mathrm{span}_\mathbb{C} \{ L_n \, | \, \pm n >0 \}$, respectively, let $\mathcal{U}(\mathcal{L})^\pm = \mathcal{U}(\mathcal{L}^\pm)$, respectively. Thus, for instance $M(c,h) \cong \mathcal{U}(\mathcal{L})^-.\mathbf{1}_{c,h}$ as vector spaces.

The Shapovalov form on a Verma module $M(c,h)$, denoted $< \cdot , \cdot >_{c,h} : M(c,h) \times M(c,h) \longrightarrow  \mathbb{C}$ is defined to be the contravariant (i.e., symmetric with respect to $(\cdot)^\dag$) bilinear form uniquely determined by
\begin{equation} \label{shap}
  < T_1.\mathbf{1}_{c,h} , T_2.\mathbf{1}_{c,h} >_{c,h} \ = \ < \mathbf{1}_{c,h}, T_1^\dag T_2. \mathbf{1}_{c,h} >_{c,h}, \ \mbox{for $T_1, T_2 \in \mathcal{U}(\mathcal{L})^-$, } 
\end{equation} 
and 
\[< \mathbf{1}_{c,h} , \mathbf{1}_{c,h} >_{c,h} = 1.\]

\begin{rem} Note the fact that $M(c,h)$ is a lowest weight module for the Virasoro Lie algebra implies that the Shapovalov form can be expressed in terms of a PBW basis for $\mathcal{U}(\mathcal{L})^-$.
\end{rem}

The determinant of the Shapovalov form restricted to the degree $d$ space of $M(c,h)$ with respect to the $\mathbb{Z}_{\geq 0}$-grading of $M(c,h)$ given by (\ref{grade-M}) is sometimes called the {\it Kac determinant}. The following was proved in \cite{FF} and is the main tool used in  the proof of the results given in  Theorem \ref{reducible}.

\begin{prop}[\cite{FF}]\label{determinant-prop} The Shapovalov form satisfies the following 
\begin{equation}\label{det}
\det < \cdot, \cdot >|_{M(c,h)_\ell} = \gamma_\ell^{1/2} \prod_{\alpha,\beta \in \mathbb{Z}_{>0}, \ \alpha \beta\leq \ell} (h - h_{\alpha,\beta})^{\mathfrak{p}(\ell - \alpha \beta)} 
\end{equation} 
where $\gamma_\ell$ is a nonzero constant independent of $h$, $c = c(t)$, the $h_{\alpha, \beta}$ are given by (\ref{eq:virpara}) and $\mathfrak{p}(n)$ denotes the integer partitions of $n$ for $n\in \mathbb{Z}_{\geq 0}$. In particular, the determinant (\ref{det}) is zero for some $\ell \in \mathbb{Z}_{\geq 0}$ if and only if $M(c,h) = M(c(t), h_{ \alpha,\beta})$ for some $\alpha,\beta \in  \mathbb{Z}_{>0}$ and $t \in \mathbb{C}^\times$ satisfying (\ref{eq:virpara}), i.e., for $M(c,h)$ with $h \in H_c$ as defined by (\ref{H_c-notation}).  

Moreover 
\begin{equation}\label{det-Phi}
(\det < \cdot, \cdot >|_{M(c,h)_\ell} )^2 = \gamma_\ell \prod_{m= 1}^\ell  \prod_{\alpha|m}  \Phi_{\alpha,m/\alpha}(c,h)^{\mathfrak{p}(\ell - m)} = \gamma_\ell \prod_{\alpha,\beta \in \mathbb{Z}_{>0}, \ \alpha \beta \leq \ell} \!\!\Phi_{\alpha,\beta}(c,h)^{\mathfrak{p}(\ell - \alpha \beta)}
\end{equation}
for some nonzero constant $\gamma_\ell$.
\end{prop}

We note that $\mathcal{U}(\mathcal{L})^- = \coprod_{\ell \in \mathbb{Z}_{>0}} \mathcal{U}(\mathcal{L})(\ell)$ with 

\begin{multline} \label{p(d)}
\mathcal{U}(\mathcal{L})(\ell) =  \mathrm{span}_{\mathbb{C}} \{ L_{-n_1} L_{-n_2} \cdots L_{-n_m} \; | \; m \in \mathbb{Z}_{>0}, \ n_1, \dots, n_m \in \mathbb{Z}_{>0}, \\ \mbox{and} \ n_1 + \cdots + n_m = \ell \}.
\end{multline}
Since $M(c,h)$ is a lowest weight module, we can fix a PBW basis for $\mathcal{U}(\mathcal{L})(\ell)$, say  $\mathcal{B}_d = \{B_1, \cdots, B_{\mathfrak{p}(\ell)} \}$,  and we can write the Gram matrix of the Shapovalov form on $M(c, h)(\ell)$ with respect to this basis as
\begin{equation}\label{define-A} \mathcal A_\ell(c,h) = ( < B_i.\mathbf{1}_{c,h} , B_j.\mathbf{1}_{c,h} > )_{1 \leq i, j \leq \mathfrak{p}(\ell) } .
\end{equation}

\begin{rem}\label{det-remark}
{\em Since we have the symmetry $\Phi_{\alpha,\beta}(c,h) = \Phi_{\beta,\alpha}(c,h)$, from Proposition \ref{determinant-prop}, in particular, from Eqn. (\ref{det-Phi}), we have that 
\[\det \mathcal{A}_\ell (c,h) = \gamma_\ell^{1/2} \prod_{0<\alpha<\beta , \ \alpha \beta \leq \ell}  \Phi_{\alpha,\beta}(c,h)^{\mathfrak{p}(\ell -\alpha\beta)} \prod_{1\leq \alpha \leq \sqrt{\ell}} \left(\sqrt{\Phi_{\alpha,\alpha}(c,h)}\right)^{\mathfrak{p}(\ell -\alpha^2)}  ,\]
where 
\[\sqrt{\Phi_{\alpha,\alpha}(c,h)} = h + \frac{\alpha^2 -1}{24}(c-1) .\]}
\end{rem}

\section{The structure of the indecomposable $V_{Vir}(c,0)$-modules induced from the level zero Zhu algebra}\label{first-induced-module-section}

In this Section we analyze the structure of the indecomposable modules for the Virasoro algebra induced from its level zero Zhu algebra. 
This Section gives the most technical results of this paper, and gives the results necessary to classify which $V_{Vir}(c,0)$-modules induced from the level zero Zhu algebra are interlocked, which is done in Section 7.

To aid in the understanding of and shed light on the technicalities, we provide some concrete examples in Subsections 6.3 and 6.6.

\subsection{Inducing modules $U(c,h,k)$ for the level zero Zhu algebra and characterizing $J(c,h,k)$ as the kernel of a certain family of matrices $\{\mathfrak{A}_{\ell}^{(k)}\}_{\ell \in \mathbb{Z}_{\geq 0}}$}

We denote by $U(c,h,k)$ the indecomposable $A_0(V_{Vir}(c,0))$-module 
\begin{equation}\label{level-zero-indecomp}
U(c,h, k) = \mathbb{C}[x]/((x - h)^{k}), \qquad \mbox{for $c,h \in \mathbb{C}$ and $k \in \mathbb{Z}_{>0}$,}
\end{equation}
under the isomorphism given in (\ref{level-zero-iso}).

Using the induction functor $\mathfrak{L}_0$ from $A_0(V_{Vir}(c,0))$-modules to $V_{Vir}(c,0)$-modules, we let $W(c,h,k) = $ 
$\mathfrak{L}_0(U(c,h,k))$ as defined in Eqn.\  (\ref{defLnfunctor}).

Then the $\mathbb{Z}_{\geq 0}$-grading of $W(c,h,k)$ is given by, for fixed $\ell > 0$,
\begin{eqnarray*}
W(c,h,k)(\ell) &=& \mathfrak{L}_0(U(c,h,k))(\ell) = (M_0(U(c,h,k))/J(c,h,k))(\ell) \\
&=& M_0(U(c,h,k))(\ell)/ J(c,h,k)(\ell)
\end{eqnarray*}
where 
\begin{multline} 
M_0(U(c,h,k))(\ell) = 
\{ L_{-n_1} L_{-n_2} \cdots L_{-n_m} u \; | \; m \in \mathbb{Z}_{>0}, \,  n_1, \dots, n_m \in \mathbb{Z}_{>0},\\
 n_1 + \cdots + n_m = \ell , \,  \, \mathrm{and} \, u \in U(c,h,k)\}, 
\end{multline}
and 
\[J(c,h,k)(\ell) = M_0(c,h,k)(\ell) \cap J(c,h,k),\]
for $J(c,h,k) = J_0(U(c,h,k))$ given by Eqn.\ \ref{J-def} where in the present setting $U = U(c,h,k)$.
Thus dim $W(c,h,k)(\ell) \leq \mathfrak{p}(\ell)$. Recall the PBW basis (\ref{p(d)}) for $\mathcal{U}(\mathcal{L})(\ell)$ given by $\mathcal{B}_\ell = \{B_1, \dots, B_{\mathfrak{p}(\ell)} \}$. 
Then  since $M_0(U(c,h,k))$ 
 is a lowest weight module, $M_0(U(c,h,k))(\ell) = (\mathrm{span}_\mathbb{C} \mathcal{B}_\ell).U(c,h,k)$, 
and 
\[\mathrm{dim} \,M_0(U(c,h,k))(\ell) = k \mathfrak{p}(\ell).\] 

\begin{rem} \label{Vir-M=Mbar}
We note that in this case, as in the case of the Heisenberg vertex algebra, $\overline{M}_0(U(c,h,k))=M_0(U(c,h,k))$ due to the properties of the induced module. More precisely, we have that $\overline{M}_n(U(c,h,k))$ is a \textit{ restricted} module for the Virasoro Lie algebra in the sense of Lepowsky-Li Remark 5.1.6 \cite{LL},  since for any $w \in \overline{M}_n(U)$ we have that $L_n w=0$ for $n$ sufficiently large. By Theorem 6.1.7 in \cite{LL} this implies that $\overline{M}_n(U(c,h,k))$ is a $V_{Vir}(c,0)$ module. Thus, in light of Remark \ref{L-a-V-module-remark} $W_A=0$ and 
$\overline{M}_n(U(c,h,k))=M_n(U(c,h,k)).$ This shows that $M_0(U(c,h,k))$ is the universal $V_{Vir}(c,0)$-module with degree zero space $U(c,h,k)$.
\end{rem}

Since $U(c,h,k) = W(c,h,k)(0)$, we have a singular vector $w$ for the Virasoro algebra (i.e., $L_n w = 0$ if $n>0$) such that  
\[
(L_0 -h I)^{k-1} w \neq 0; \;\;\; (L_0 - h I)^{k}w = 0,
\]
and $W(c,h, k)(0) = \mathrm{span}_\mathbb{C} \{ (L_0 - h I)^{k - i} w \; | \; i =  1, \dots, k \}$.  

So $L_0-h I $ acts on $W(c,h, k)(0)$ for the Jordan basis $u_i = (L_0 - h I)^{k - i}w$ for $i = 1, \dots,  k$ so as to have zeros on the diagonal and 1's on the super-diagonal. 
That is $L_0$ with respect to this Jordan basis $u_1, \dots, u_k$ is, using the notation of Section 4.3
\begin{equation}\label{Jordan-block}
L_0|_{W(c,h,k)(0)} = h Id_{k} + D_{k,1} .
\end{equation}

Note that, in particular, by induction on $m$ we have that 
\begin{equation}\label{L_0-on-u}
L_0^m u_j = \sum_{i = 0}^m h^{m-i} \binom{m}{i} u_{j-i} ,
\end{equation}
 where we take $u_{j-i} = 0$ if $j-i<1$. \\

{\bf Notation:} To simplify notation, we let 
\begin{align} \label{defR_d}
\mathcal{R}_\ell:=\mathrm{span}_\mathbb{C} \mathcal{B}_\ell,
\end{align}
and we will from now on write $M(c,h,k)$ for $M_0(U(c,h,k))$.

\vskip.5cm

The structure of $W(c,h,k) = \coprod_{\ell \in \mathbb{Z}_{\geq 0}} W(c,h,k)(\ell) = M(c,h,k)/J(c,h,k)$ is determined by which vectors in $M(c,h,k)(\ell) = \mathcal{R}_\ell.U(c,h,k)$ are also in $J(c,h,k)(\ell) = M(c,h,k)(\ell) \cap J(c,h,k)$, for each $\ell \in \mathbb{Z}_{\geq 0}$, and this is determined by whether lowering a vector in $\mathcal{R}_\ell.U(c,h,k)$ back down to $W(c,h,k)(0)= U(c,h,k)$ annihilates the vector.  The lowering operators are given by
\[
\mathcal{R}_\ell^\dag = \mathrm{span}_{\mathbb{C}} \{L_{n'_1}L_{n'_2} \cdots L_{n'_m} \; | \; m \in  \mathbb{Z}_{>0}, \, n'_1, \dots, n'_m \in \mathbb{Z}_{>0} , \, n'_1 + \cdots + n'_m = \ell\}\]
and thus  using the PBW basis for $\mathcal{R}_\ell$, we can form a PBW basis  
$\mathcal{B}_\ell^\dag = \{B_1^\dag, \dots, B_{\mathfrak{p}(\ell)}^\dag \}$ for $\mathcal{R}_\ell^\dag$. Note that $\mathcal{R}^\dag_\ell = \mathrm{span}_\mathbb{C} \mathcal{B}^\dag_\ell$, we have that 
$v + J(c,h,k)(\ell) = 0 + J(c,h,k)(\ell)$  if and only if $\mathcal{R}_\ell^\dag.v = 0$, i.e., if and only if $B^\dag_iv = 0$ for all $1\leq i \leq \mathfrak{p}(\ell)$. 

This implies that to determine $J(c,h,k)(\ell) = \mathcal{R}_\ell.U(c,h,k) \cap J(c,h,k)$, we must determine when $\mathcal{R}_\ell^\dag . \mathcal{R}_\ell. u = 0$ for $u \in U(c,h,k)  = \mathrm{span}_\mathbb{C} \{ u_1, \dots, u_k\}$.   Thus we must consider the $k\mathfrak{p}(\ell) \times k\mathfrak{p}(\ell)$ linear system given by the action of $ \mathcal{B}_\ell^\dag$ on the $k\mathfrak{p}(\ell)$ vectors in $\mathcal{B}_\ell.U(c,h,k)$ giving $k\mathfrak{p}(\ell)$ equations in $k\mathfrak{p}(\ell)$ unknowns in terms of the basis $\{u_1, \dots, u_k \}$.  This implies that $\mathcal{R}_\ell.U \cap J  = 0$ if and only if the determinant of this $k\mathfrak{p}(\ell
) \times k\mathfrak{p}(\ell)$ linear system is nonzero, and more generally the rank of this linear system is the dimension of $W(c,h,k)(\ell)$.  

More precisely, in terms of the Jordan basis for $U = U(c,h,k)$, given by $\mathcal{B}_U = \{ u_1, \dots, u_k\}$, and a basis for $\mathcal{R}_\ell$ given by $ \mathcal{B}_\ell = \{B_1, \dots, B_{\mathfrak{p}(\ell)} \}$, then 
$\mathcal{R}_\ell.U = M_0(U)(\ell)$ has a corresponding basis given by 
\begin{eqnarray}\label{basis}
 \mathcal{B}_{\mathcal{R}_\ell.U} &=& \{ B_1u_1, \dots, B_1u_k, B_2u_1, \dots, B_2u_k, \dots, B_{\mathfrak{p}(\ell)}u_1, \dots,B_{\mathfrak{p}(\ell)}u_k \}\\
 &=&  \{B_1^1, B_1^2,\dots, B_1^k, B_2^1, B_2^2, \dots, B_2^k, \dots, B_{\mathfrak{p}(\ell)}^1, B_{\mathfrak{p}(\ell)}^2, \dots, B_{\mathfrak{p}(\ell)}^k  \},\nonumber
 \end{eqnarray} 
where we have set the notation $B_i^j = B_iu_j$ for $i = 1, \dots, \mathfrak{p}(\ell)$ and $j = 1, \dots, k$.

\begin{rem}
    Note that this basis (\ref{basis}) is useful for determining the $V_{Vir}(c,0)$-module induced from $U(c,h,k)$, i.e. in determining $J$, but will not be a  strongly interlocked basis for degree $\ell$ for when the module $W(c,h,k)$ is strongly interlocked. 
\end{rem}

Then $v \in \mathcal{R}_\ell.U$ is in $J$ if and only if  every linear combination of the $B_i^\dag v$ is zero, for $i$ running through $1 \leq i \leq \mathfrak{p}(\ell)$,  which is equivalent to $v$ being in the kernel of the following matrix with respect to the bases (\ref{basis}) given by
\begin{eqnarray}\label{define-A-k}
\mathfrak A_\ell^{(k)}(c,h) =  \left[ \begin{array}{cccccc} 
N_{1,1}  & N_{1,2} & \cdots & N_{1, \mathfrak{p}(d)} \\
N_{2,1} & N_{2,2} & \cdots & N_{2, \mathfrak{p}(\ell)}  \\
\vdots & \vdots & \ddots & \vdots \\
N_{\mathfrak{p}(\ell), 1} & N_{\mathfrak{p}(\ell), 2} & \cdots & N_{\mathfrak{p}(\ell), \mathfrak{p}(\ell)}
\end{array} \right]
\end{eqnarray}
where $N_{i,j}$ is the $k \times k$ matrix given by 
\begin{equation}\label{define-blocks}
N_{i,j} = \left[ [B_i^\dag B_j  u_1]_{\mathcal{B}_U} \ \ [B_i^\dag B_j  u_2]_{\mathcal{B}_U} \ \  \cdots \ \  [B_i^\dag B_ju_k]_{\mathcal{B}_U} \right] 
\end{equation}
for $1\leq i, j \leq \mathfrak{p}(\ell)$, with entries in $\mathbb{C}[c, h]$.
Note then that the Gramm matrix $ \mathcal A_\ell(c,h)$, with respect to the basis $\mathcal{B}_{\ell}$, of the degree $\ell \in \mathbb{N}$ Shapovalov form  corresponds to  $\mathfrak A_\ell^{(1)}(c,h)$, i.e., $\mathfrak A_\ell^{(k)}(c,h)$ for the case $k=1$.  In this case  $V.u_k = V.u_1$, for $V = V_{Vir}(c,0)$, is a highest weight irreducible module since $U = \mathbb{C}u_1$ and so $U_0(c, h, 1) = \mathbb{C}[x]/(x-h)$.  In other words, in the case of $k = 1$,  $W = \mathfrak{L}_0(U(c, h, 1) )\cong M(c,h)/J(c,h) = M(c,h)/T(c,h) = L(c,h)$ which is $M(c,h)$ if and only if $J(c,h) = T(c,h) = 0$. Namely, if and only if we are in Case (0) of Proposition \ref{propcases} and 
the pair $(c,h)$ is not of the form  $(c(t),h_{r,s}(t))$ as in (\ref{eq:virpara}) for $r, s \in \mathbb{Z}_{>0}$ and $t \in \mathbb{C}^\times$.  Therefore, for $k=1$ we obtain the Verma module $M(c,h)$ as an induced module from the zero level Zhu algebra if and only if the Shapovalov determinant $ \mathcal A_\ell(c,h)$ acting on $W = M(c,h)$ is nonsingular for all degree spaces and there are no conformal singular vectors in $M(c,h)$.

We summarize this in the following Proposition.

\begin{prop}\label{NulJ-prop}
For $k \in \mathbb{Z}_{>0}$, let $W(c,h,k)$ be the induced module $W(c,h,k) = \mathfrak{L}_0(U(c, h, k))$ for the $A_0(V_{Vir}(c,0))$-module given by $U(c, h, k)  \cong \mathbb{C}[x]/ ((x - h)^k )$. Then \[W(c,h,k)  = \coprod_{\ell \in \mathbb{Z}_{\geq 0}} W(c,h,k)(\ell) \cong M_0(U(c,h,k))/J(c,h,k)\] 
has $\ell$th degree space for $\ell \in \mathbb{Z}_{\geq 0}$ determined by $J(c,h,k)(\ell)= Ker \, \mathfrak A_\ell^{(k)}(c,h)$, where $\mathfrak A_\ell^{(k)}(c,h)$ is given by Eqns.\ (\ref{define-A-k})--(\ref{define-blocks}) and $M_0(U(c,h,k))(\ell)=\mathcal{R}_\ell.U(c,h,k)$.
\end{prop}

Motivated by this characterization of the space $J(c,h,k)(\ell) = Ker \, \mathfrak A_\ell^{(k)}(c,h)$ and the form of $\mathfrak A_\ell^{(k)}(c,h)$, we prove the following basic linear algebra result.

\begin{prop}\label{lin-alg-prop}
 Let $p,k \in \mathbb{Z}_{>0}$, let $A = (a_{i,j})_{1 \leq i,j \leq p}$ be a $p \times p$ matrix with entries in $\mathbb{C}$, and let $N$ be a $pk \times pk$ matrix consisting of $p^2$  block matrices of size $k \times k$ such that the $(i,j)$th block is upper triangular with the constant $a_{i,j}$  on the diagonal.  Then 
\[\det (N) = \det(A)^k.\]
\end{prop}

\begin{proof} We prove the result by induction on $p$, and by using the block form of the determinant
\[ \det \left[ \begin{array}{cc} 
M & B \\
C & D
\end{array} \right] = \det(M) \det ( D - C M^{-1} B), \quad \mbox{if $M$ is invertible.}\]

If $p = 1$, then $A = (a_{1,1})$ and $N$ is one $k \times k$ upper triangular block with $a_{1,1}$ on the diagonal.  Thus $\det A = a_{1,1}$ and $\det N = a_{1,1}^k$ and the result holds.  

Assume the result holds for $N$ a $(p-1)k \times (p-1)k$ matrix that consists of $(p-1)^2$ blocks of size $k \times k$ such that the $(i,j)$th block is upper triangular  with $a_{i,j}$ on the diagonal.  

Note that $N = (N_{i,j})_{1\leq i,j \leq p}$ where $N_{i,j} = a_{i,j}I_k + T_{i,j}$ with $I_k$ the $k \times k$ identity matrix and $T_{i,j}$ a $k \times k$ strictly upper triangular matrix.  Let $D_N$ be the $(p-1)k \times (p-1)k$ matrix given by $D_N = (N_{i,j})_{2 \leq i,j \leq  p}$, let $B_N$ be the $k \times (p-1)k$ matrix given by $B_N = (N_{1,j})_{2\leq j \leq p}$, and let $C_N$ be the $(p-1)k \times k$ matrix given by $C_N = (N_{i,1})_{2 \leq i \leq p}$.  So that we have 
\[N = \left[ \begin{array}{cc} N_{1,1} & B_N \\ 
C_N & D_N  \end{array}\right] .\] 

If $N_{1,1} = a_{1,1} I_k + T_{1,1}$ is invertible, then $N_{1,1}^{-1} = a^{-1}_{1,1} \sum_{n =0}^{k-1} (-1)^n a^{-n}_{1,1} T_{1,1}^n$ and 
\begin{eqnarray*}
 \det (N) &=& \det(N_{1,1}) \det(D_N - C_NN_{1,1}^{-1} B_N ) \ = \  a_{1,1}^k \det \Big( D_N -   \sum_{n=0}^{k-1} (-1)^n a^{-n-1}_{1,1} C_N T_{1,1}^n B_N  \Big)\\
 &=&  a_{1,1}^k \det \Big( D_N -      a_{1,1}^{-1} C_NB_N +   \sum_{n =1}^{k-1} (-1)^n a^{-n-1}_{1,1} C_N T_{1,1}^n B_N  \Big) \\
 &=& a_{1,1}^k \det \Big( D_N  -      a_{1,1}^{-1} (N_{i,1}N_{1,j})_{2\leq i,j \leq p}  +   \sum_{n =1}^{k-1} (-1)^n a^{-n-1}_{1,1} C_N T_{1,1}^n B_N  \Big) \\
 &=& a_{1,1}^k \det \left( D_N  -   a_{1,1}^{-1} (N_{i,1}N_{1,j})_{2\leq i,j \leq p}  + T  \right)
 \end{eqnarray*}
 where $T$ is a $(p-1)k \times (p-1)k$ matrix consisting of $(p-1)^2$ blocks of size $k \times k$ that are each strictly upper triangular, and $D_N -a_{1,1}^{-1} (N_{i,1}N_{1,j})_{2\leq i,j \leq p}$ is a $(p-1)k \times (p-1) k$ matrix consisting of $(p-1)^2$ blocks of size $k \times k$ that are upper triangular with the $((i-1), (j-1))$th block given by $a_{i,j} I_k - a_{1,1}^{-1} a_{i,1}a_{1,j} I_k + T'_{i,j}$ for some strictly upper triangular matrix $T'_{i,j}$.

Thus by the inductive assumption, we have that 
\begin{eqnarray*}
 \det (N) &=& a_{1,1}^k \det \left( D_N  -      a_{1,1}^{-1} (N_{i,1}N_{1,j})_{2\leq i,j \leq p}  + T  \right)\ = \  a_{1,1}^k (\det ( (a_{i,j} - a_{1,1}^{-1} a_{i,1} a_{1,j})_{2 \leq i, j \leq p}))^k\\
&=&  (a_{1,1} \det ( ( a_{i,j} - a_{1,1}^{-1} a_{i,1} a_{1,j})_{2 \leq i, j \leq p}))^k.
\end{eqnarray*}
Decomposing $A$ as follows:
\[ A = \left[ \begin{array}{cc} 
a_{1,1} & B_A\\
C_A & D_A
\end{array} \right]
\]
where $B_A$ is the $1 \times p$ matrix $B_A = (a_{1,2} \ a_{1,3} \ \cdots \ a_{1,p})$, $C_A$ is the $p \times 1$ matrix $C_A = (a_{2,1} \ a_{3,1} \ \cdots \ a_{p,1})^T$ and $D_A$ is the $(p-1) \times (p-1)$ matrix given by $D_A = (a_{i,j})_{2 \leq i, j \leq p}$, then we have that 
\begin{eqnarray*}
 \det (N) &=&  (a_{1,1} \det ( ( a_{i,j} - a_{1,1}^{-1} a_{i,1} a_{1,j})_{2 \leq i, j \leq p}))^k \ = \  (\det(a_{1,1}) \det( D_A - C_A a_{1,1}^{-1} B_A))^k \\
 &=& (\det(A))^k.
\end{eqnarray*}

If $N_{1,1} = a_{1,1} I_k + T_{1,1}$ is not invertible, i.e., if $a_{1,1} = 0$,  then we have two cases: Case I.  $N_{i,1} = a_{i,1}I_k + T_{i,1}$ is not invertible for all $1\leq i\leq p$;  Case II. $N_{i,1} = a_{i,1}I_k + T_{i,1}$ is invertible for some $2\leq i\leq p$. In the first case, this implies $a_{i,1} = 0$ for all $1 \leq j \leq p$ which implies that both the first column of $N$ and the first column of $A$ are all zeros, and thus the result holds since $\det (N) = 0 = (0)^k = (\det(A))^k$.   

If Case II holds, then $N_{i,1}$ is invertible for some $i = 2, \dots, p$, we can multiply $N$ by the permutation matrix $P_N$  given by $p^2$ blocks of size $k \times k$ that switches the $k$ rows containing the block $N_{i,1}$ with the first $k$ rows containing the block $N_{1,1}$. 
Then $\det (P_N) \det (N) = (-1)^k \det(N)$, since $P_N$ has switched $k$ rows.  

Similarly, we can multiply $A$ by the permutation matrix  $P_A$ that switches the $i$th row containing $a_{i,1}$ with the first row containing  the $a_{1,1}$ entry.  Then $\det(P_A) \det(A) = - \det(A)$.  Applying the proof above to the matrices $P_N N$ and $P_A A$, we have  
\begin{eqnarray*}
\det (N) &=& (-1)^k \det(P_N) \det(N) \ = \  (-1)^k \det(P_NN)  \ = \  (-1)^k  (\det(P_AA))^k \\
&=& (-1)^k (\det(P_A) \det(A))^k  \ =  \ (-1)^k (- \det(A))^k \ = \ (\det(A))^k.
\end{eqnarray*}
\end{proof}

\subsection{The relationship between the determinants of $\mathcal{A}_\ell$ and $\mathfrak{A}_\ell^{(k)}$ }

From Propositions \ref{NulJ-prop} and \ref{lin-alg-prop}, we have the following Theorem.

\begin{thm}\label{det-thm}  For $\ell \in \mathbb{Z}_{\geq 0}$, fix a basis $\mathcal{B}_{\ell} = \{ B_1, \dots, B_{\mathfrak{p}(\ell)} \}$ of $\mathcal{R}_\ell$.  
Let $\mathcal{A}_\ell$ be the Gram matrix of the Shapovalov form on degree $\ell$ of the Verma module  $M(c,h)$ with respect to the basis $\mathcal{B}_{\ell}$, i.e., 
\[\mathcal{A}_\ell (c,h) = ( < B_i {\bf 1}_{c,h}, B_j {\bf 1}_{c,h} > )_{1 \leq i,j \leq \mathfrak{p}(\ell)}.\]
For $k \in \mathbb{Z}_{>0}$, let  $W(c,h,k) = \mathfrak{L}_0(U(c,h,k)) = M(c,h,k)/J(c,h,k)$ be the induced module for $U(c, h, k)  \cong \mathbb{C}[x]/((x - h)^k)$ as an $A_0(V_{Vir}(c,0))$-module.  Let 
$\mathfrak{A}_\ell^{(k)}(c,h)$ be the Gram matrix of the Shapovalov form on degree $\ell$ of $M(c,h,k)$ with respect to the basis 
\[\mathcal{B}_{\mathcal{R}_\ell.U} = \{ B_1u_1, \dots, B_1u_k, B_2u_1, \dots, B_2u_k, \dots, B_{\mathfrak{p}(\ell)}u_1, \dots,B_{\mathfrak{p}(\ell)}u_k \},\] 
where $\mathcal{B}_U = \{u_1, \dots, u_k\}$ is a Jordan basis for $U(c,h,k)$.  That is $\mathfrak{A}_\ell^{(k)}(c,h)$ is given by Eqns.\ (\ref{define-A-k}) and (\ref{define-blocks}). 

Then $W(c,h,k)  = \coprod_{\ell \in \mathbb{Z}_{\geq 0}} W(c,h,k)(\ell)$ has $\ell$th degree space for $\ell \in \mathbb{Z}_{\geq 0}$ given by 
\[W(c,h,k)(\ell) =  M(c,h,k)(\ell)/J(c,h,k)(\ell) \]
where 
\[ J(c,h,k)(\ell) = Ker \, \mathfrak A_\ell^{(k)}(c,h),\] 
and we have that 
\[\det (\mathfrak A_\ell^{(k)}(c,h)) = \left( \det ( \mathcal A_\ell(c,h)) \right)^k = \gamma_\ell^k \prod_{\alpha,\beta \in \mathbb{Z}_{>0}, \ \alpha \beta \leq \ell} (h - h_{\alpha,\beta})^{k\mathfrak{p}(\ell - \alpha \beta )} . \]
\end{thm}

\begin{proof} The matrix $\mathfrak A_\ell^{(k)}(c,h)$ consists of $\mathfrak{p}(\ell)^2$ blocks of $k \times k$ upper triangular matrices with entries in $\mathbb{C}[c, h]$ since the blocks are polynomials in $L_0$ and $cI$ acting on the Jordan basis $\mathcal{B}_U = \{u_1, \dots, u_k\}$ with respect to $L_0$.  Furthermore, the $(i,j)$th block is of the form $a_{i,j} I_k + T$ where $a_{i,j}$ is the $(i,j)$th entry of $ \mathcal A_\ell$ and $T$ is strictly upper triangular. Thus the result follows by Proposition \ref{lin-alg-prop}.
\end{proof}

 \begin{rem}
In light of the result above we note that we can use the singular vector formulas, which have been extensively studied in the literature \cite{AF, FF, FF2, K}, to describe the generators of $J(c,h,k)(\ell) = Ker\, \mathfrak A_\ell^{(k)}(c,h)$ for different values of $c,h$ and $k$.
\end{rem} 

\subsection{Examples of $\mathfrak{A}_\ell^{(k)}$ for $k = 3$}\label{example-section} 

Here we give some concrete examples of the relationship between 
$\mathcal{A}_\ell(c,h)$ and $\mathfrak{A}^{(k)}_\ell(c,h)$ for 
low degrees $\ell$ in the case $k = 3$ for general $(c, h)$.  In this 
case $U(c, h, 3) = \mathrm{span}_\mathbb{C} \{u_1,u_2, u_3 \}$ where 
$u_1 = (L_0 - h I)^2w$, $u_2 = (L_0 - h I)w$, and $u_3 = w$ is a basis for $U$ and $w$ is singular (i.e., $L_nw = 0$ for $n>0$) and in addition $(L_0 - h I)^3 w = 0$. 
At degree 1, we have that  $\mathcal{R}_1 = \mathbb{C} L_{-1}$, and that $ \mathcal A_1$ is the $1 \times 1$ matrix given by
\begin{eqnarray*} \mathcal A_1 (c,h) &=& (< L_{-1}{\bf 1}_{c,h}, L_{-1}{\bf 1}_{c,h} >) \ = \ ( <  {\bf 1}_{c,h}, L_{1}L_{-1}{\bf 1}_{c,h} >) \\
&=& ( <  {\bf 1}_{c,h}, 2L_0 {\bf 1}_{c,h} >) \ = \ ( 2 h ).
\end{eqnarray*}

We note that a basis for $M(c,h,3)(1)$ is $\mathcal{B}_{\mathcal{R}_1.U} = \{B_1 u_1, B_1 u_2, B_1 u_3\}$ for $B_1 = L_{-1}$ and with respect to this basis we have that
\[\mathfrak A_1^{(3)}(c,h) =  \left[ \begin{array}{ccc}
2 h   & 2 & 0 \\
0  & 2 h & 2 \\
0 & 0 & 2 h
\end{array} \right] .\]
And, as an example of Theorem \ref{det-thm}, we have that $\det \mathfrak A_1^{(3)}(c,h) =( \det (\mathcal{A}_1)(c,h))^3$.

At degree $2$, we let $B_1 = L_{-1}^2$ and $B_2 = L_{-2}$, so that $\mathcal{R}_2 = \mathrm{span}_\mathbb{C} \{B_1, B_2\}$ and $\mathcal{R}^\dag_2 = \mathrm{span}_\mathbb{C} \{B_1^\dag, B_2^\dag\} = \mathrm{span}_\mathbb{C} \{ L_{1}^2, L_{2} \}$. Then  for $u \in U$, $L_nu = 0$ for $n>0$, and so
\begin{eqnarray*}
B^\dag_1 B_1 u \ =& L_{1}^2 L_{-1}^2 u &= \  8L_0^2u + 4L_0u  \\
B^\dag_1 B_2  u\ = & L_{1}^2 L_{-2} u &= \ 6L_0 u\\
B^\dag_2 B_1 u \ = & L_{2}L_{-1}^2 u &= \ 6L_0u\\
B^\dag_2 B_2 u \ = & L_{2} L_{-2} u &= \ 4L_0u + \frac{c}{2}u. 
\end{eqnarray*}

Thus in terms of the basis $\{B_1, B_2\}$ for $\mathcal{R}_d$, with corresponding dual basis for $\mathcal{R}_d^\dag$, we have that 
\[ \mathcal A_2 (c,h) = \left[ \begin{array}{cc}
8 h^2 + 4 h  & 6 h \\
6 h & 4 h + c/2 
\end{array} \right] .\]
And with respect to the basis 
$\mathcal{B}_{ \mathcal{R}_2.U }=\{ B_1 u_1, B_1 u_2, B_1 u_3, B_2u_1, B_2u_2, B_2 u_3 \}$ for the subspace $M(c,h,3)(2)$, we have
\begin{eqnarray*} \lefteqn{\mathfrak A_2^{(3)} (c,h) }\\
&=&
\left[ \begin{array}{ccc|ccc}
[B_1^\dag B_1 u_1]_{\mathcal{B}_U}   & [B_1^\dag B_1 u_2]_{\mathcal{B}_U}   & [B_1^\dag B_1 u_3]_{\mathcal{B}_U}  & [B_1^\dag B_2 u_1]_{\mathcal{B}_U}  & [B_1^\dag B_2 u_2]_{\mathcal{B}_U}  &  [B_1^\dag B_2 u_3]_{\mathcal{B}_U}  \\
& & \\

\big[ B_2^\dag B_1 u_1 \big]_{\mathcal{B}_U}   & [B_2^\dag B_1 u_2]_{\mathcal{B}_U}   & [B_2^\dag B_1 u_3]_{\mathcal{B}_U}  & [B_2^\dag B_2 u_1]_{\mathcal{B}_U}  & [B_2^\dag B_2 u_2]_{\mathcal{B}_U}  &  [B_2^\dag B_2 u_3]_{\mathcal{B}_U}  \\
\end{array} \right] 
 \\
 \\
&=& \left[ \begin{array}{ccc|ccc}
8 h^2 + 4 h   & 16 h + 4  & 8 & 6h & 6 & 0 \\
0 & 8 h^2 + 4 h  & 16 h + 4  & 0 & 6 h & 6\\
0 & 0 & 8 h^2 + 4 h  & 0 & 0 & 6 h \\
\mbox{\sout{ \qquad \quad  } }& \mbox{\sout{ \qquad \quad  } } & \mbox{\sout{ \qquad \quad  } } & \mbox{\sout{ \qquad \quad  } } & \mbox{\sout{ \qquad \quad  } } & \mbox{\sout{ \qquad \quad  }}  \\
6h & 6 & 0 & 4 h + c/2 & 4 & 0\\
0 & 6 h & 6 & 0 & 4 h + c/2 & 4 \\
0 & 0 & 6 h & 0 & 0 & 4h + c/2 
\end{array} \right] .
\end{eqnarray*} 
\vskip1cm
 At degree $3$ if we let $B_1 = L_{-1}^3$ and $B_2 = L_{-2}L_{-1}$ and $B_3=L_{-3}$, we have $\mathcal{R}_3 = \mathrm{span}_\mathbb{C} \{B_1, B_2, B_3\}$ and $\mathcal{R}^\dag_3 = \mathrm{span}_\mathbb{C} \{B_1^\dag, B_2^\dag, B_3^\dag\} = \mathrm{span}_\mathbb{C} \{ L_{1}^3,L_1L_2, L_3 \}$. 
In this case, we obtain
\begin{equation}\label{A_3} \mathcal A_3(c,h) = \left[ \begin{array}{ccc}
24 h(2h^2 + 3h + 1)  & 12h(3h + 1) & 24 h \\
12h(3h + 1) &  h(8h + 8 + c) & 10 h \\
24h & 10h & 6h + 2c
\end{array} \right].
\end{equation} 
\begin{rem}
In \cite{FF} Feigin and Fuchs use the basis $ \{L_{-1}^3, L_{-1} L_{-2}, L_{-3} \}$ for $\mathcal{R}_3$ and thus get a different but equivalent matrix representation for the Shapovalov form.  Our modified choice of basis simplifies computations for our purposes. We note that there is a typo in the coefficient of the determinant of $\det \mathcal{A}_3$ in \cite{FF}. From Eqn.\ (\ref{A_3}), we have that $\det \mathcal{A}_3(c,h) = 2304 \Phi_{1,1}(c,h) \Phi_{2,1}(c,h) \Phi_{3,1}(c,h).$ 
\end{rem}

Instead of building the complicated $9\times 9$ matrix  $\mathfrak{A}_3^{(3)}(c,h)$ by direct computation of its entries as prescribed by Eqns.\ (\ref{define-A-k}) and (\ref{define-blocks}), we will give a general formula for $\mathfrak{A}_\ell^{(k)}(c,h)$ for all $k\geq 1$ in terms of the entries for $\mathcal{A}_\ell(c,h)$ for $\ell\geq 1$ in the next subsection.

\subsection{The entries of $\mathfrak{A}_\ell^{(k)}(c,h)$ as derivatives of $\mathcal{A}_\ell(c,h)$}

Note that for any $T_n, R_n \in \mathcal{R}_n$ and $u \in U = M(c,h,k)(0)$, since $T_n^\dag R_n u \in M(c,h,k)(0)$, we have that 
\begin{equation}\label{define-a} T_n^\dag R_n u = a_{T_n, R_n}(c, L_0)u
\end{equation}
for  a unique polynomial $a_{T_n, R_n}(x,y) \in \mathbb{C}[x,y]$  in the formal variables $x$ and $y$, where  $a_{T_n, R_n}(c, L_0) = a_{T_n, R_n}(x,y)|_{(x,y) = (c, L_0)} \in \mathrm{End} \, M(c,h,k)(0)$.  In addition, $a_{T_n, R_n}(x,y)$ is of degree $n$ or less in $y$.

Then we have the following lemma.

\begin{lem}\label{derivative-lemma} 
Let $T_\ell, R_\ell \in \mathcal{R}_\ell$ for $\ell\geq 1$, and let $\{u_1, \dots, u_k\}$ be the Jordan basis for $U(c,h,k)$ giving the action of $L_0$ on $W(c,h,k)$ as in (\ref{Jordan-block}). For fixed $j \in \{1,\dots, k\}$, we have that as an element in $M(c,h,k)(0)$
\begin{equation}
T_\ell^\dagger R_\ell u_j =  \sum_{i=0}^{\ell} \frac{1}{i!}\left(\frac{\partial}{\partial y}\right)^i a_{T_\ell, R_\ell }(x,y) u_{j-i}\bigg|_{(x,y) = (c,h)},
\end{equation}
where $u_{j-i} = 0$ when $j-i<1$, and $a_{T_i, R_i}(x,y) \in \mathbb{C}[x,y]$ is uniquely defined by Eqn.\ (\ref{define-a}).
\end{lem} 

\begin{proof}
 Recall  Eqn.\ (\ref{L_0-on-u}). Writing 
\[a_{T_\ell, R_\ell}(c,h) =a_{T_\ell,R_\ell}(x,y)|_{(x,y) = (c,h)} =  \sum_{m=0}^\ell a_m(x) y^m |_{(x,y) = (c,h)}\]
for $a_m(x) \in \mathbb{C}[x]$, we have that for $j = 1,\dots, k$ 
\begin{eqnarray*}
 T_\ell^\dag R_\ell u_j &=& a_{T_\ell, R_\ell}(c, L_0)u_j \ = \  \sum_{m=0}^\ell a_m(c) L_0^m u_j  \nonumber \\
&=& \sum_{m=0}^\ell a_m(c) \sum_{i= 0}^m  \binom{m}{i} h^{m-i}u_{j-i} 
= \sum_{m=0}^\ell a_m(c) \sum_{i=0}^m \frac{1}{i!} \Big(\frac{\partial}{\partial y} \Big)^i y^m u_{j-i} \Big|_{y = h} \nonumber \\
&=& \sum_{i = 0}^{\ell} \frac{1}{i!} \Big(\frac{\partial}{\partial y}\Bigr)^i a_{T_\ell, R_\ell}(x,y) u_{j-i} \Big|_{(x,y) = (c,h)},
\end{eqnarray*}
giving the result.
\end{proof}

Next we introduce the following projection maps. 

\begin{defn} \label{proj-def}
For $\ell \in \mathbb{Z}_{\geq 0}$ and $1\leq j\leq k$, we define {\it the projection map onto the $u_j$ component at degree $\ell$} as the linear map $\pi_j^\ell$ satisfying 
\begin{eqnarray*}
\pi_j^\ell: M(c,h,k) &\longrightarrow & \mathcal{R}_nu_{j}\\
\sum_{i,m} R_{m}^iu_i &\mapsto & \pi_j^\ell(\sum_{i, m} R_{m}^iu_i )=R_\ell^ju_j 
\end{eqnarray*}
for any (not necessarily homogeneous) element $\sum_{i,  m} R_{ m}^i u_i=\sum_{i=1}^k\left(\sum_{m\geq 0}R_{m}^i\right)u_i \in M(c,h,k)$, i.e., $R_m^i \in \mathcal{R}_m$, and the $R_m^i$ are nonzero only for a finite number of $m \in \mathbb{Z}_{\geq 0}$.
\end{defn}

Recall the basis of $M(c,h,k)(\ell)$ given by Eqn.\ (\ref{basis}), i.e.,
\[ \{B_1^1, \dots, B_1^k, B_2^1, \dots, B_2^k, \dots, B_{\mathfrak{p}(\ell)}^1, \dots, B_{\mathfrak{p}(\ell)}^k \}\]
with $\{B_1, \dots, B_{\mathfrak{p}(\ell)}\}$ a basis of $\mathcal{R}_\ell$ and $B_i^m = B_i u_m$.

For $1\leq i,j \leq \mathfrak{p}(\ell)$, and $1 \leq m, n \leq k$, we define the form 
\begin{equation}\label{matrix-notation}
< B_i^m, B_j^n > : = \pi_m^0(a_{B_i, B_j} (c, L_0)u_n) = \pi_m^0(B_i^\dag B_j u_n).
\end{equation} 
So then the $(i,j)$th block of $\mathfrak{A}_\ell^{(k)}(c,h)$ is given by $(< B_i^m, B_j^n >)_{1 \leq m,n \leq k}$  
and 
\begin{equation} \label{matrix} 
\mathfrak{A}_\ell^{(k)}(c,h) = (< B_i^m,  B_j^n >)_{1 \leq i,j \leq \mathfrak{p}(\ell), 1\leq m,n \leq k}. \end{equation} 

We have the following corollary to Lemma \ref{derivative-lemma}.

\begin{cor}\label{structure-cor} The following gives the structure of $\mathfrak{A}_\ell^{(k)}(c,h)$ in the basis (\ref{basis}) using the notation (\ref{matrix-notation}) and (\ref{matrix}), and where $a_{B_i, B_j}(x,y) \in \mathbb{C}[x,y]$ is uniquely defined by Eqn.\ (\ref{define-a}).  

For $1 \leq i,j \leq \mathfrak{p}(\ell)$ the $(i,j)$th block of $\mathfrak{A}_\ell^{(k)}$ is given by 
\begin{eqnarray*}
< B_i^m, B_j^m > \ =&  a_{B_i, B_j}(c,h), \qquad \qquad &\mbox{for $m = 1, \dots, k$}\\
< B_i^m, B_j^{m+1} > \ =& \frac{\partial}{\partial y} a_{B_i, B_j}(x,y) \Big|_{(x,y) = (c,h)}, \qquad &\mbox{for $m = 1, \dots, k-1$}\\
< B_i^m, B_j^{m+2} > \  =& \frac{1}{2!} \Big( \frac{\partial}{\partial y} \Big)^2 a_{B_i, B_j}(x,y)\Big|_{(x,y) = (c,h)}, \qquad &\mbox{for $m = 1, \dots, k-2$}\\
\vdots \qquad \ \ 
=&  \qquad \vdots \\
< B_i^m , B_j^{m+n} > \  
=& \frac{1}{n!}\Big( \frac{\partial}{\partial y}\Big)^n a_{B_i, B_j}(x,y)\Big|_{(x,y) = (c,h)}, \qquad &\mbox{for $m = 1, \dots, k-n$}\\
\vdots \qquad \ \ 
=& \qquad \vdots\\
< B_i^1, B_j^{k} > \   
=& \frac{1}{k!} \Big( \frac{\partial}{\partial y} \Big)^k a_{B_i, B_j}(x,y) \Big|_{(x,y) = (c,h)}, 
\end{eqnarray*}
and 
\begin{equation}\label{zeros}
< B_i^m, B_j^n > \  = 0 \qquad \mbox{for $1\leq n < m \leq k$. }
\end{equation}
In particular: 
\begin{itemize}
\item Each $(i,j)$th block of $\mathfrak{A}_\ell^{(k)}(c,h)$ is upper triangular;

\item Each $(i,j)$th block has the $(i,j)$th entry of  $\mathcal{A}_\ell(c,h)$ on the diagonal;

\item Each $(i,j)$th block has on the $n$th super-diagonal $D_{k,n}$ of that block, for $n = 1,\dots, k-1$,  $\frac{1}{n!} \left(\frac{\partial}{\partial y}\right)^n$ of the $(i,j)$th entry of $\mathcal{A}_\ell(x,y)$  evaluated at $(x,y) = (c,h)$. 
\end{itemize}
\end{cor}

\subsection{The relationship between the derivatives of $\mathcal{A}_\ell(c,h)$ and $J(c,h,k)(\ell)=Ker \, \mathfrak{A}_\ell^{(k)}(c,h) \subset M(c,h,k)(\ell)$}

Recall that the Gram matrix $A_\ell (c,h)$ of the Shapovalov form at degree $\ell$ has entries that are polynomials in $c$ and $h$.  
For clarity, we replace $c$ and $h$ with formal variables $x$ and $y$. We do this since in this section we relate the partial 
derivatives of $\mathcal{A}_\ell(x,y)$ with respect to $y$ to the entries of $\mathfrak{A}_\ell^{(k)}(x,y)$, and the partial 
derivatives of $\det \mathcal{A}_\ell(x,y)$ with respect to $y$, evaluated at 
particular $(x,h) = (c,h)$, to $J(c,h,k)(\ell) = Ker \, \mathfrak{A}_\ell^{(k)}(c,h) \subset M(c,h,k)(\ell)$ to 
determine the structure of $W(c,h,k) (\ell) = M(c,h,k)(\ell) /J(c,h,k)(\ell)$.  In particular, we have the following Proposition.

\begin{prop}\label{derivative-equations-prop} Fix a basis $\mathcal{B}_\ell $ for $\mathcal{R}_\ell$, and let $\mathcal{A}_\ell(c,h)$ denote the Gram matrix of the Shapovalov form with respect to this basis. For $v\in M(c,h,k)(\ell)$, write $v = \sum_{j = 1}^k R_\ell^ju_j$ for $R_\ell^j \in \mathcal{R}_\ell$. Let $[R_\ell^j]_{\mathcal{B}_\ell}$ denote the coordinate vector of $R_\ell^j$ in the $\mathcal{B}_\ell$ basis for $j = 1,\dots, k$. Then $v \in J(c,h,k)(\ell) = Ker \, \mathfrak{A}_\ell^{(k)}(c,h)$ if and only if all of the following equations hold
\begin{eqnarray}
0 &=& \mathcal{A}_\ell (c,h) [R_\ell^k]_{\mathcal{B}_\ell} \label{der1}\\
0 &=& \mathcal{A}_\ell (c,h) [R_\ell^{k -1} ]_{\mathcal{B}_\ell} +  \frac{\partial}{\partial y} \mathcal{A}_\ell (x,y)[R_\ell^k]_{\mathcal{B}_\ell} \Big|_{(x,y) = (c,h)} \label{der2} \\
0 &=& \mathcal{A}_\ell (c,h) [R_\ell^{k -2} ]_{\mathcal{B}_\ell} + \Big( \frac{\partial}{\partial y} \mathcal{A}_\ell (x,y)[R_\ell^{k-1}]_{\mathcal{B}_\ell}  +  \frac{1}{2!} \Big( \frac{\partial}{\partial y} \Big)^2\mathcal{A}_\ell (x,y)[R_\ell^{k}]_{\mathcal{B}_\ell} \Big) \Big|_{(x,y) = (c,h)} \\
 \vdots &=& \vdots \nonumber \\
 0 &=& \mathcal{A}_\ell (c,h) [R_\ell^{k - m} ]_{\mathcal{B}_\ell} + \Big(\frac{\partial}{\partial y} \mathcal{A}_\ell (x,y)[R_\ell^{k - m+1} ]_{\mathcal{B}_\ell}   +  \frac{1}{2!} \Big( \frac{\partial}{\partial y} \Big)^2\mathcal{A}_\ell (x,y)[R_\ell^{k-m+2}]_{\mathcal{B}_\ell} \label{der-n}\\
 & & \quad +  \cdots +  \frac{1}{m!} \Big( \frac{\partial}{\partial y} \Big)^m\mathcal{A}_\ell (x,y)[R_\ell^{k}]_{\mathcal{B}_\ell} \Big) \Big|_{(x,y) = (c,h)}\nonumber  \\
\vdots &=& \vdots \nonumber \\
0 &=& \mathcal{A}_\ell (c,h) [R_\ell^1 ]_{\mathcal{B}_\ell} + \Big( \frac{\partial}{\partial y} \mathcal{A}_\ell (x,y)[R_\ell^2]_{\mathcal{B}_\ell} + \frac{1}{2!} \Big( \frac{\partial}{\partial y} \Big)^2\mathcal{A}_\ell (x,y)[R_\ell^{3}]_{\mathcal{B}_\ell}\label{der-last} \\
& & \quad +  \cdots + \frac{1}{k!} \Big( \frac{\partial}{\partial y} \Big)^k\mathcal{A}_\ell (x,y)[R_\ell^{k}]_{\mathcal{B}_\ell} \Big) \Big|_{(x,y) = (c,h)}. \nonumber 
\end{eqnarray}
\end{prop}
\begin{proof}
By definition of $J(c,h,k) = Ker \, \mathfrak{A}_\ell^{(k)}$ and linear independence of $u_1, \dots, u_k$, we have that $v \in J(c,h,k)$ if and only if $\pi_j^0(T^\dag_\ell v) = 0$ for all $j = 1, \dots, k$ and for all $T_\ell \in \mathcal{R}_\ell = \mathrm{span} \, \mathcal{B}_\ell$.  The result follows by Lemma \ref{derivative-lemma}, Corollary \ref{structure-cor}, and the fact that by Eqn.\ (\ref{define-A}), 
\[\mathcal{A}_\ell   = ( < B_i.\mathbf{1}_{c,h} , B_j.\mathbf{1}_{c,h} > )_{1 \leq i, j \leq \mathfrak{p}(\ell) }  = (a_{B_i, B_j}(c,h) )_{1 \leq i, j \leq \mathfrak{p}(\ell) } , \]
with Eqn.\ (\ref{der1}) corresponding to the requirement that $\pi_k^0(T^\dag_\ell v) = 0$, Eqn.\ (\ref{der2}) corresponding to the requirement that $\pi_{k-1}^0(T^\dag_\ell v) = 0$, etc. 
\end{proof} 

To analyze solutions to the equations characterizing $J(c,h,k)$ given in Proposition \ref{derivative-equations-prop}, we prove the following 

\begin{prop}\label{Phi-derivative-prop} 
Suppose the maximal submodule $T(c,h) \subset M(c,h)$  is either generated by one singular vector $S_{r,s}(t) {\bf 1}_{c,h}$ or two singular vectors, $S_{r,s}(t) {\bf 1}_{c,h}$ and $S_{r',s'}(t) {\bf 1}_{c,h}$ with  $rs + 1<r's'$. Then the following hold for the Gram matrix $\mathcal{A}_\ell (x,y)$ in formal variables $(x,y)$ of the Shapovalov form: 

(i) $\left. \det \mathcal{A}_{\ell}(x,y) \right|_{(x,y) = (c(t), h_{r,s})} =  0$  if and only if $\ell\geq rs$.

(ii) For $\ell = rs$ or $rs + 1$
\[
\left.  \Big(  \frac{\partial}{\partial y} \det \mathcal{A}_{\ell}(x,y) \Big)\right|_{(x,y) = (c(t), h_{r,s})}  =  0 \quad \mbox{ if and only if $r \neq s$ and $t = \pm 1$.} \]
\end{prop}

\begin{proof}
Recall from Proposition \ref{determinant-prop} and Remark \ref{det-remark}, that 
\begin{equation}\label{det-proof}
\det \mathcal{A}_\ell(c,h) = \gamma_{\ell}^{1/2}\prod_{0<\alpha< \beta ,\  \alpha \beta \leq \ell}  \Phi_{ \alpha,\beta}(c,h)^{\mathfrak{p}(\ell-\alpha \beta)}  \prod_{1\leq  \alpha \leq \sqrt{\ell}} \left( \sqrt{\Phi_{ \alpha, \alpha}(c,h)}\right)^{\mathfrak{p}(\ell - \alpha^2)},
\end{equation}
where 
\[\sqrt{\Phi_{ \alpha, \alpha}(c,h)} = h + \frac{ \alpha^2 -1}{24}(c-1) .\]
Thus when $\ell \geq d=rs$, $\det \mathcal{A}_\ell (c(t), h_{r,s}(t)) = 0$ since $\Phi_{r,s}(c(t), h_{r,s}(t)) = 0$ and $\Phi_{r,s}(x,y)$ divides $ \det \mathcal{A}_\ell(x,y)$, whereas if $\ell<rs$, since
$S_{r,s}(t){\bf 1}_{c,h}$ is the lowest degree vector in $T(c,h)$, we have that $(c,h) = (c(t),h_{r,s}(t)) \in\Phi_{r,s}(c,h)$, and  $(c,h)  = (c(t),h_{r,s}(t)) \notin \Phi_{\alpha,\beta}(c,h)$ for $\alpha \beta <rs$, so that $\det \mathcal{A}_\ell(c,h) \neq 0$, proving $(i)$.

In order to prove $(ii)$, we consider the two cases: $r \neq s$ and $r = s$, for  $(c,h) = (c(t), h_{r,s}(t))$. If $r \neq s$, we write $\det \mathcal{A}_\ell(c,h)$ as 
\[\det \mathcal{A}_\ell(c,h) = f(x,y) \Phi_{r,s}(x,y)^{\mathfrak{p}(\ell - rs)}\big|_{(x,y) = (c,h)} \]
where 
\[f(x,y) = \gamma_{\ell}^{1/2}\prod_{\substack{0<\alpha< \beta ,\  \alpha \beta \leq \ell\\
 (\alpha, \beta) \neq (r,s)}}  \Phi_{ \alpha,\beta}(x,y)^{\mathfrak{p}(\ell-\alpha \beta)}  \prod_{1\leq  \alpha \leq \sqrt{\ell}} \left( \sqrt{\Phi_{ \alpha, \alpha}(x,y)}\right)^{\mathfrak{p}(\ell - \alpha^2)}.\]
Then $f(x,y)|_{(x,y) = (c,h)} \neq 0$ since $\Phi_{\alpha, \beta}(x,y)|_{(x,y) = (c,h)} \neq 0$ if $(\alpha, \beta)\neq (r,s)$ by Theorem \ref{reducible}, and the fact that we are assuming that if $M(c,h)$ falls in Case (2) of Proposition \ref{propcases}, i.e., $T(c,h) =\langle S_{r,s}(t){\bf 1}_{c,h}, S_{r',s}(t) {\bf 1}_{c,h} \rangle$, then $r's'>\ell + 1$. 

Thus by the product rule and the fact that $\Phi_{r,s}(x,y)|_{(x,y) = (c(t), h_{r,s}(t))} = 0$, 
\begin{eqnarray*}
\lefteqn{\Big(\frac{\partial}{\partial y} \det A_\ell (x,y) \Big)\Big|_{(x,y) = (c(t), h_{r,s}(t))}  }\\
&
=& \Big(\frac{\partial}{\partial y} f(x,y) \Big) \Phi_{r,s}(x,y)^{\mathfrak{p}(\ell - rs)} + f(x,y) \frac{\partial}{\partial y} \Phi_{r,s}(x,y)^{\mathfrak{p}(\ell - rs)} \Big|_{(x,y) = (c,h)} \\
&=& f(x,y) \frac{\partial}{\partial y} \Phi_{r,s}(x,y)^{\mathfrak{p}(\ell - rs)} \Big|_{(x,y) = (c,h)},
\end{eqnarray*}
which is zero if and only if one of the following holds: Either $\mathfrak{p}(\ell - rs) >1$ (i.e.,  $\ell > rs + 1$); or $\mathfrak{p}(\ell - rs) = 1$ (i.e., $\ell = rs$ or $rs + 1$) and
\[\Big(\frac{\partial}{\partial y}\Phi_{r,s}(x,y)\Big) \Big|_{(x,y) = (c(t), h_{r,s}(t))}   = 0.\]

Observe that from Eqn.\ (\ref{Phi-r-s}) we have 
\begin{eqnarray*}
\frac{\partial}{\partial y}\Phi_{ r,s} (x,y)\Big|_{(x,y) = (c, h)} &=&  \Big( h + \frac{1}{24} (s^2 - 1)(c-13) + \frac{1}{2}(rs - 1)\Big) \\
& & \quad + \Big( h + \frac{1}{24} (r^2 - 1)(c-13) + \frac{1}{2}(rs - 1)\Big)   \\
&=& 2h + \frac{1}{24}(r^2 + s^2 - 2)(c-13) + rs -1 .\\
\end{eqnarray*}
Thus letting $(c,h) =(c(t), h_{r,s}(t)) \in \Phi_{r,s}(c,h)$, Eqns.\ (\ref{eq:virpara}) imply that
\begin{eqnarray*}
\frac{\partial}{\partial y} \Phi_{ r,s} (x,y)\Big|_{(x,y) = (c(t), h_{r,s}(t))} 
&=&  2\Big( \frac{1-r^2}{4}t + \frac{1-rs}{2}+\frac{1-s^2}{4}t^{-1}\Big) \\
& & \quad + \frac{1}{24} (r^2 + s^2 - 2)(6t + 13 + 6t^{-1} - 13) + rs - 1\\
&=& \Big( \frac{1-r^2}{2} + \frac{r^2 + s^2 - 2}{4}\Big) t + \Big( \frac{1-s^2}{2} + \frac{r^2 + s^2 - 2}{4}\Big) t^{-1}\\
&=& \frac{1}{4t} (s^2-r^2) (t^2 - 1).
\end{eqnarray*}
Therefore in the case when $r \neq s$ this is zero if and only if $t = \pm 1$.

Whereas, if $r = s$, we have that 
\[\frac{\partial}{\partial y} \det A_\ell (x,y) \Big|_{(x,y) = (c(t), h_{r,s}(t))} =0  ,\]
if and only if $\mathfrak{p}(\ell - r^2)>1$ (i.e.,  $\ell > r^2 + 1$)
since 
\[\frac{\partial}{\partial y} \sqrt{\Phi_{r,r}(x,y)} = 1 \neq 0 \quad \mbox{for all $(x,y) = (c,h)$.}  \]

\end{proof}

Next, we recall the following result from linear algebra.

{\bf Jacobi's Formula:}  Let $A(y)$ be an $n \times n$ matrix with entries in $\mathbb{C}[y]$. Then, its derivative satisfies 
\[\frac{\partial}{\partial y}\det A(y) = \mathrm{tr} \left( adj (A (y)) \frac{\partial}{\partial y} A(y) \right) , \]
where $adj( A)$ denotes the adjugate (also called the classical adjoint) of $A$, i.e., the transpose of the cofactor matrix of $A$.

In the following Lemma we note some additional basic linear algebra results that will be used in Proposition \ref{Jacobi-lemma}

\begin{lem} \label{orth} Let $A$ be an $n\times n$ matrix with complex coefficients of rank $n-1$.  Let $\{\vec{s}\}$ be a basis for $Ker A$  and let $\{\vec{z}\}$ be a basis for $Ker A^T$. Then,

(i) $A\vec{b}=\vec{d}$ has a solution $\vec{b}\in \mathbb{C}^n$ if and only if $\vec{d}^{\, T} \vec{z} = 0$.

(ii) The adjugate matrix of $A$ has rank $1$ and is of the form \[adj(A)=\mu \ \vec{s}\  \vec{z}^{\ T}  \quad \mbox{
for a nonzero constant $\mu.$}\]
\end{lem}

\begin{proof}
For part $(i)$ of the Lemma, we note that 
the system $A\vec{b}=\vec{d}$ has a solution if and only if $\vec{d}$ is in the column space of $A$, which we denote by $Range(A)$. Since $Range(A)=(Ker A^*)^\bot$ and $KerA^*$ is given by $\vec{\bar{z}}$, the complex conjugate of $\vec{z}$, we obtain that $\vec{d}$ is in $Range(A)$ if and only if $\langle\vec{d}, \vec{\bar{z}}\rangle_H=\vec{d}^{\, T}  \vec{z}=0$, where $\langle \cdot,\cdot  \rangle_H$ denotes the usual Hermitian form. Therefore, we have that $A\vec{b}=\vec{d}$ has a solution if and only if $\vec{d}^{\, T}  \vec{z} = 0$.

For part $(ii)$ of the Lemma, since $A \ adj (A)=det(A) I=0$ we have that each column of $adj(A)$ is in the kernel of $A$ and therefore, a scalar multiple of $\vec{s}$. This implies that $rk(adj(A))\leq 1$. Moreover, since $A$ has rank $n-1$, we know that at least one of its cofactors is nonzero and therefore $adj(A)
\neq 0$. Thus, we obtain $rk(adj(A))=1$.  Because $adj(A)$ is a rank one matrix, it is of the form $\vec{u} \vec{v}^T$ for some vectors $\vec{u}$ and $\vec{v}$. The fact that the columns of $adj(A)$ are all scalar  multiples of $\vec{s}$ together with $adj(A) \neq 0$ implies we can take $\vec{u}$ to be a nonzero scalar multiple of $\vec{s}$. On the other hand, from $adj (A) A=0$, it follows that the rows of $adj(A)$ are in the left kernel of $A$ and are therefore scalar multiples of $\vec{z}$. Thus, we have that $adj(A)=\mu \vec{s} \  \vec{z}^{\ T} $ for a nonzero scalar $\mu \in \mathbb{C}$.
\end{proof}

Next, we use Jacobi's Formula together with the Lemma above to prove the following general linear algebra results.

\begin{prop}\label{Jacobi-lemma} Let $A(y)$ be an $n \times n$ matrix with entries in $\mathbb{C}[y]$ for $y$ a formal variable, and with $\det A(y) \neq 0$.  Assume  that when evaluated at $y = h$,  for a fixed $h \in \mathbb{C}$,  we have $\det A(h) = 0$ with the rank of $A(h)$ equal to $n-1$, and $Ker \, A(h) = \mathbb{C}\vec{s}$.  

Then, the equation 
\[A(y) \Big|_{y = h}  \vec{b} = -  \Big(\frac{\partial}{\partial y} A(y)\Big) \Big|_{y = h} \vec{s}\]
has a solution $\vec{b} \in \mathbb{C}^n$ if and only if 
\[ \frac{\partial}{\partial y} \det A(y)\Big|_{y = h} = 0.\]
\end{prop}

\begin{proof} 
Although in our applications of this Lemma $A$ will depend on two variables and thus we have written the statement of the Lemma in terms of partial derivatives, for the purposes of the proof, for shorthand we write $A'(h)$ for the partial derivative with respect to $y$ evaluated at $h$.

Using Lemma \ref{orth} $(i)$, we have that under the hypothesis of the Proposition, $A(h)\vec{b}=-A'(h)\vec{s}$  will have a solution if and only if 
$ \vec{z} ^{\ T} A'(h) \vec{s}=0.$ 
Note that  on the one hand, we have that 
\begin{align}\label{zA's}
 \vec{z} ^{\ T} A'(h) \vec{s}= \sum_{1\leq i,j\leq n} z_i[A'(h)]_{i, j}s_j.
\end{align}
On the other hand, by Lemma \ref{orth} $(ii)$, we know that 
$$adj(A)=\mu \ \vec{s} \  \vec{z}^{\ T}= \mu\begin{bmatrix}
    \vert & \vert &\vert &\vert \\
  z_1\vec{s}   &  z_2\vec{s} &\cdots &   z_n \vec{s}  \\
    \vert & \vert &\vert &\vert
\end{bmatrix}$$
which implies that $[adj (A)]_{i,j}=\mu \ s_i   z_j$.
In particular, using (\ref{zA's}) we have that 
\begin{eqnarray*} 
  \mathrm{tr}(adj (A(h)) A'(h))&=& \sum_{1 \leq i,j \leq n} [adj(A)]_{i,j} [A'(h)]_{j,i} \  = \ \sum_{1 \leq i,j \leq n} \mu \ s_i   z_j [A'(h)]_{j,i} \nonumber \\
  & =&  \sum_{1 \leq i,j \leq n} \mu \    z_j [A'(h)]_{j,i}s_i \ = \ \mu \  \vec{z}^{\ T}  A'(h) \vec{s}. 
\end{eqnarray*}
Thus $A(h)\vec{b}=-A'(h)\vec{s}$ has a solution if and only if $tr(adj(A(h)) A'(h))=0$. Finally, using the Jacobi formula, we have that $tr(adj(A(y)) A'(y))=\frac{d}{dy} det(A(y))$ and therefore,  the system $A(h)\vec{b}=-A'(h)\vec{s}$ has a solution if and only if $\frac{d}{dy} det(A(y))|_{y=h}=0$.
\end{proof}

\begin{rem} Although we have stated the results given in Lemma \ref{orth} and Proposition \ref{Jacobi-lemma} for $A(y)$ having polynomial entries in $\mathbb{C}[y]$, which is the context in which we will use them, the proofs follow through for entries that are analytic functions in a neighborhood of $y$.
\end{rem}

For the next Theorem and elsewhere, we will make use of the following Lemma.

\begin{lem}\label{S-lemma} Suppose the maximal submodule $T(c,h) \subset M(c,h)$ is nonzero and $S_{r,s}{\bf 1}_{c,h}$ is the lowest degree vector in $T(c,h)$. Let $v \in M(c,h,k)(\ell)$, with $\ell \geq rs$ such that one of the following hold:   Case (1) holds and $T(c,h) = \langle S_{r,s}(t){\bf 1}_{c,h}\rangle$; or Case (2) holds and $T(c,h) = \langle S_{r,s}(t){\bf 1}_{c,h}, S_{r',s'}(t){\bf 1}_{c,h} \rangle$ and $\ell < r's'$.  Writing $v = \sum_{i = 1}^j R_\ell^iu_i$, for $R_\ell^1, \dots, R_\ell^j \in \mathcal{R}_\ell$ with $R_\ell^j \neq 0$ for some $1 \leq j \leq k$, it follows that   
if $v \in J(c,h,k)(\ell) = Ker \, \mathfrak{A}_{\ell}^{(k)}(c,h)$, 
then $\pi_j^\ell(v) = R_{\ell-rs}S_{r,s}(t) u_j$ for some $R_{\ell-rs} \in \mathcal{R}_{\ell-rs}$.
\end{lem}

\begin{proof}
Since $v \in Ker \, \mathfrak{A}_\ell^{(k)}$ if and only if $T_\ell^\dag v = 0$ for all $T_\ell \in \mathcal{R}_\ell$, we must have that 
$\pi_i^0(T_\ell^\dag v) = 0$ for all $i = 1 , \dots, j$. In particular, we must have that $\pi_j^0(T_\ell^\dag v) = 0$.

 However, by Lemma \ref{derivative-lemma}, $\pi_j^0(T_\ell^\dag v) = \pi_j^0(T_\ell^\dag R_\ell^ju_j) = a_{T_\ell, R_\ell}(c,h) u_j$.  
 Thus we must have $a_{B_i,R_\ell}(c,h) = 0$ for all $B_i \in \mathcal{B}_\ell$, which implies $R_\ell^j{\bf 1}_{c,h} \in Ker \, \mathcal{A}_\ell(c,h)$.  Thus $R_\ell^j{\bf 1}_{c,h} \in T(c,h)$ and since $\ell<r's'$ if $T(c,h)$ is generated by two singular vectors, this implies $R_\ell^j {\bf 1}_{c,h} \in \langle S_{r,s}(t){\bf 1}_{c,h}\rangle$, and so $R_\ell^ju_j \in \langle S_{r,s}(t) u_j
\rangle$, giving the result.   
\end{proof}

Now we are ready to give a characterization of the degree $d = rs$ subspace of $W(c,h,k) = M(c,h,k)/J(c,h,k)$ in Case (1) or (2) when $T(c,h)$ is nonzero and contains $S_{r,s}(t) {\bf 1}_{c,h}$ with $rs  = d$ by determining aspects of the submodules $J(c,h,k) \subset M(c,h,k)$.

\begin{thm}\label{degree-d-theorem} Suppose the maximal submodule $T(c,h) \subset M(c,h)$ is nonzero and $S_{r,s}(t) {\bf 1}_{c,h}$ is the lowest degree vector,  of degree $d=rs$ in $T(c,h)$. Let $W(c,h,k) = \mathfrak{L}_0(U(c,h,k)) = M(c,h,k)/J(c,h,k)$, for $k \in \mathbb{Z}_{>0}$.

Then, one of the following holds:

(i) If either $t \neq \pm 1$,  or $t = \pm 1$ and $r = s$, then for $c = c(t)$ and $h = h_{r,s}(t)$, we have that 
\[J(c,h,k)(d) = Ker \, \mathfrak{A}_{d}^{(k)} (c,h) = \mathrm{span} \{S_{r,s}(t)u_1\}.\]

(ii) If $t = \pm 1$, and $r \neq s$, then for $c = c(\pm1)$ and $h = h_{r,s}(\pm1)$, we have that either there exists a positive integer $\kappa^{\pm}_{r,s} \geq 2$ such that for each $1\leq n\leq \kappa^\pm_{r,s}$, there exist $R_d^1, \dots, R_d^n \in \mathcal{R}_d$ with $R_d^n \neq 0$ that give a solution to the equation
\begin{equation}\label{kappa}
\sum_{j=1}^{n} \frac{1}{(j-1)!}\Big(\frac{\partial}{\partial y} \Big)^{j-1} \mathcal{A}_{d} (x,y)[R_d^j]_{\mathcal{B}_d}\Big|_{(x,y) = (c,h)} = 0,
\end{equation}
and such that if $n>\kappa^\pm_{r,s}$, then there is no solution to Eqn.\ (\ref{kappa}) with $R_d^n \neq 0$, or no such positive integer $\kappa_{r,s}^\pm$ exists and solutions to Eqn.\ (\ref{kappa}) exist for all $n \in \mathbb{Z}_{>0}$, in which case we write $\kappa_{r,s}^\pm = \infty$. Note that  $\kappa_{r,s}^\pm$ depends on $t = \pm1$, respectively, and $r,s$.

Then, for $n \leq \kappa_{r,s}^\pm$, up to a scalar multiple, a solution $R^1_d, \dots, R_d^n$ to Eqn.\ (\ref{kappa}) can be written as 
\[R_{r,s}^{\pm,n-1}, \  R_{r,s}^{\pm, n - 2}, \ \dots, \ R_{r,s}^{\pm, 2}, \ R_{r,s}^{\pm,1}, \ S_{r,s}(\pm 1),\] 
for $t = 1$ or $t = -1$, respectively, for some  $R_d^j = R_{r,s}^{\pm, n - j} \in \mathcal{R}_d$, for $j = 1, \dots, n-1$, and $R_d^n = S_{r,s}(\pm 1)$. In addition, 
we have that $J(c,h,k)(d)  =  Ker \, \mathfrak{A}_{d}^{(k)}(c,h)$ is given by
\begin{eqnarray*}
J(c,h,k)(d)  &=& \mathrm{span}\{S_{r,s}(
\pm 1)u_1,  \ S_{r,s}(\pm 1)u_2 + R^{\pm, 1}_{r,s}u_1, \ S_{r,s}(\pm 1)u_3 + R^{\pm, 1}_{r,s} u_2 + R^{\pm, 2}_{r,s} u_1,  \dots  \nonumber \\ 
& & \quad \dots, \ S_{r,s}(\pm1 )u_k + R_{r,s}^{\pm, 1}u_{k-1} + \cdots + R_{r,s}^{\pm, k-1}u_1    \} \nonumber \\
&=& \mathrm{span} \{ S_{r,s}(\pm1)u_n +  \sum_{i = 1}^{n-1} R_{r,s}^{\pm, i}u_{n-i} \; | \; n = 1, \dots, k \} 
\end{eqnarray*}
if $k\leq \kappa^\pm_{r,s}$ for $\kappa^\pm_{r,s}$ finite and for any $k$ if $\kappa^\pm_{r,s}$ is infinite, whereas for $\kappa^\pm_{r,s}$ finite and $k> \kappa^\pm_{r,s}$, then 
\[J(c,h,k)(d) = J(c,h,\kappa^\pm_{r,s})(d).\]
\end{thm}

\begin{proof}
Let  $v \in M(c,h,k)(d)$ for $d = rs$. Then $v = \sum_{j = 1}^k R_d^ju_j$ for some $R_d^1, \dots, R_d^k \in \mathcal{R}_d$, and by Proposition \ref{derivative-equations-prop}, $v \in Ker \, \mathfrak{A}_{d}^{(k)}$ is equivalent to the  Eqns.\ (\ref{der1})-(\ref{der-last}) holding for $\ell = d = rs$. 

If, for some $1\leq n \leq k,$ we have that $R_d^n \neq 0$ and $R^j_d = 0$ for $j>n$, this is equivalent to  Eqns.\ (\ref{der-n})-(\ref{der-last}) holding for $m=k-n$. Eqn.\ (\ref{der-n})  with $m=k-n$ implies that $[R_d^n]_{\mathcal{B}_d} \in Ker \, \mathcal{A}_d(c,h)$ which is $\langle S_{r,s}(t) \rangle$, and thus $R_d^nu_n \in \langle S_{r,s}(t) u_n
\rangle$, so that up to a scalar multiple, $R_d^n = S_{r,s}(t)$.  

The next equation in Proposition \ref{derivative-equations-prop} implies 
\begin{equation}\label{b-derivative}   \mathcal{A}_d(x,y) [R_d^{n-1}]_{\mathcal{B}_d} +  \left( \frac{\partial}{\partial y} \mathcal{A}_d(x,y)\right)  [S_{r,s}(t)]_{\mathcal{B}_d} \Big|_{(x,y) = (c(t), h_{r,s}(t))} = 0, 
 \end{equation}
must hold. 

Letting $A(y) = \mathcal{A}_d(x,y)|_{x = c(t)}$ and $h = h_{r,s}(t)$, in Proposition \ref{Jacobi-lemma}, we have that Eqn.\ (\ref{b-derivative}) has a solution $\vec{b} = [R_d^{n-1}]_{\mathcal{B}_d}$ if and only if  
\[ \frac{\partial}{\partial y} \det \mathcal{A}_d(x,y) \Big|_{(x,y) = (c(t), h_{r,s}(t))} = 0. \]
Therefore, by Proposition \ref{Phi-derivative-prop},  
there is a solution to  Eqn.\ (\ref{b-derivative}) if and only if $t = \pm 1$ and $r \neq s$.  This proves part $(i)$ of the Theorem.  

Continuing with part $(ii)$, we note that if $t = \pm1$ and $r \neq s$, 
then Eqn.\ (\ref{b-derivative}) does have a solution, 
and since in this case, the rank of $\mathcal{A}_d(c,h)$ is $\mathfrak{p}(d) -1$, 
this solution, which we will denote by $R_{r,s}^{\pm, 1} = R_d^{n-1}$, for $t = \pm 1$, respectively, is unique up to $Ker\, \mathcal{A}_d(c,h) = \mathbb{C}S_{r,s}(\pm1)$.  
Thus we have that 
\[\mathrm{span} \{ S_{r,s}(\pm1)u_1, \ S_{r,s}(\pm1)u_2 + R_{r,s}^{\pm, 1}u_1 \} \subseteq 
 J(c,h,k)\] if $k \geq 2$, and these give all the vectors in $J(c,h,k)$ of the form $T_d^1u_1 + T_d^2u_2$ for $T_d^1, T_d^2 \in \mathcal{R}_d$. In addition, this proves that if some maximal $\kappa^\pm_{r,s}$ exists, then $\kappa^\pm_{r,s}\geq 2$ for $t= \pm 1$ and $r \neq s$.

For the rest of part $(ii)$ of the Theorem, we have that if $t = \pm 1$, $r \neq s$, $k>2$, and $v \in Ker \, \mathfrak{A}_d^{(k)}(c,h) = J(c,h,k)$ then up to a scalar multiple
\begin{equation}\label{v}
v = S_{r,s}(\pm 1)u_n +  R_{r,s}^{\pm, 1}   u_{n-1} + \sum_{j = 1}^{n-2} R_d^j u_j,
\end{equation}
for some $1\leq n\leq k$, with $[R_{r,s}^{\pm, 1}]_{\mathcal{B}_{d}}$ the solution to Eqn. (\ref{b-derivative}) at $t = \pm 1$, respectively. Moreover, $v$ would have to  satisfy the next equation in Proposition \ref{derivative-equations-prop} which is
\begin{multline}\label{second-stage} \Big( \mathcal{A}_d (x,y) [R_d^{n -2} ]_{\mathcal{B}_d} +  \frac{\partial}{\partial y} \mathcal{A}_d (x,y)[R_{r,s}^{\pm, 1} ]_{\mathcal{B}_d}  \\
+  \frac{1}{2!} \Big( \frac{\partial}{\partial y} \Big)^2\mathcal{A}_d (x,y)[S_{r,s}(\pm1)]_{\mathcal{B}_d} \Big) \Big|_{(x,y) = (c(\pm1),h_{r,s}(\pm1))} = 0.
\end{multline}
 
If Eqn.\ (\ref{second-stage}) has no solution, then $\kappa^{\pm}_{r,s} = 2$, and there are no more linearly independent vectors in $J(c,h,k)(d)$.  If Eqn.\ (\ref{second-stage}) has a solution then $\kappa^{\pm}_{r,s} \geq 3$ or is infinite, and if $k\geq 3$, then $S_{r,s}(\pm1)u_3 + R_{r,s}^{\pm, 1}u_2 + R_{r,s}^{\pm, 2}u_1 \in J(c,h,k)(d)$ for $R^{\pm, 2}_{r,s} = R_d^{n-2}$ in Eqn.\ (\ref{second-stage}), for $t = \pm 1$, respectively. By Lemma \ref{S-lemma}, and the fact that $\mathcal{A}_d(c,h)$ has rank $\mathfrak{p}(d) - 1$, it follows that any  vector of the form $R_d^3u_3 + R_d^2u_2 +R_d^1u_1$ in $J(c,h,k)(d)$ is in $\mathrm{span} \{S_{r,s}(\pm1)u_1, \ S_{r,s}(\pm1)u_2 + R_{r,s}^{\pm, 1}u_1, \ S_{r,s}(\pm1)u_3 + R_{r,s}^{\pm, 1}u_2 + R_{r,s}^{\pm,2} u_1 \}$ for $k \geq 3$.

Proceeding in this manner, we induct on $n$ in Eqn.\ (\ref{v}). Assume that $\kappa_{r,s}^\pm \geq 3$, possibly infinite, and for $3\leq n< \kappa_{r,s}^\pm$, the vector $v$ as in Eqn.\ (\ref{v}) is in $J(c,h,k)$ for $k\geq n$ so that $v = S_{r,s}(\pm1)u_n + \sum_{j=1}^{n-1} R_{r,s}^{\pm, j}u_{n-j}$ is a solution to 
\[\sum_{j = 1}^{n} \frac{1}{(j-1)!} \Big(\frac{\partial}{\partial y} \Big)^{j-1} \mathcal{A}_{d} (x,y)[R_d^j]_{\mathcal{B}_d}\Big|_{(x,y) = (c,h)} = 0, \]
and 
\begin{multline}\label{J-pm-n} 
\mathrm{span} \{S_{r,s}(\pm1)u_1, \ S_{r,s}(\pm1)u_2 + R_{r,s}^{\pm, 1}u_2, \ \dots, \\ 
S_{r,s}(\pm 1)u_n + R^{\pm, 1}_{r,s} u_{n-1} + \cdots + R_{r,s}^{\pm, n-1}u_1\} \subseteq J(c,h,k)(d),
\end{multline} 
for $k\geq n$ and $J(c,h,k)(d)$ is equal to the span of the first $k$ vectors in this set if $k<n$. 

Then consider the equation
\begin{equation}\label{induction} 
\sum_{j = 1}^{n+1} \frac{1}{(j-1)!} \Big(\frac{\partial}{\partial y} \Big)^{j-1} \mathcal{A}_{d} (x,y)[R_d^j]_{\mathcal{B}_d}\Big|_{(x,y) = (c,h)} = 0. 
\end{equation} 
If there is no solution to Eqn.\ (\ref{induction}), then 
$\kappa^{\pm}_{r,s}$ is finite and is equal to $n$. If there is a solution, i.e., there exists a $R_{r,s}^{\pm, n} \in \mathcal{R}_d$ such that 
\begin{equation}\label{v-n+1}
v = S_{r,s}(\pm1)u_{n+1} + \sum_{i=1}^n R_{r,s}^{\pm, i}u_{n+1-i}
\end{equation}
is a solution to Eqn. (\ref{induction})  
with $R_d^j = R_{r,s}^{\pm, n+1 - j}$, for $j = 1, \dots, n$ and $R_d^{n+1} = S_{r,s}(\pm 1)$, then $\kappa_{r,s}^\pm \geq n+1$.  

If there is no solution to (\ref{induction}), then the inclusion given by Eqn.\ (\ref{J-pm-n}) is an equality and the statement of the Theorem holds with $\kappa_{r,s}^\pm = n$. 

If there is a solution to Eqn.\ (\ref{induction}), then $\kappa_{r,s}^\pm \geq n+1$,  
and the vector Eqn.\ (\ref{v-n+1}) is also in $J(c,h,k)(d)$  as long as $k \geq n +1$. Continuing in this manner, the proof of part $(ii)$ follows in a finite number of steps for a given $k$.  
\end{proof}

\begin{rem}
It is of interest to note that if $t = 1$, then $c = 25$, and if $t = -1$, then $c = 1$, which are known as the degenerate minimal model central charges \cite{Mil2, OH}.  This will be one of the settings where we will show that some of the modules $W(c,h,k)$ are interlocked and others are not depending on $k$ and $\kappa_{r,s}^\pm$. In addition, we give a concrete example of Theorem \ref{degree-d-theorem} for $t = -1$, $r = 2$, $s = 1$, i.e., $c = 1$ and $h = h_{2,1}(-1) = 1/4$, in which case $\kappa^-_{2,1} = 2$.
\end{rem}

Proposition \ref{determinant-prop}, Remark \ref{Vir-M=Mbar}, Theorems \ref{det-thm}, \ref{degree-d-theorem}, and  \ref{J-theorem} combine to give the following Theorem describing the submodule $J(c,h,k)$ of $M(c,h,k)$, and thus $W(c,h,k) = M(c,h,k)/J(c,h,k)$ up through at least the degree of the lowest singular vector of $M(c,h)$.  

\begin{cor}\label{J-theorem}
Let $W(c,h, k)(0)= \mathrm{span}\{u_1, \dots, u_k\} \cong U(c,h,k) = \mathbb{C}[x]/((x-h)^k)$ and let $W(c,h,k) = \mathfrak{L}_0(U(c,h,k)) = M(c,h,k)/J(c,h,k)$. \\   

Case (0): If $T(c,h) = 0$, then  $J(c,h,k) =0$, and $W(c,h,k) = M(c,h,k)$.\\

Case (1):  If $T(c,h) = \langle S_{r,s}(t) {\bf 1}_{c,h} \rangle$, we have two subcases: \\

(i) If either $t \neq \pm 1$, or $t = \pm 1$ but $r = s$, then   
\[J(c,h,k)(\ell) = \left\{ \begin{array}{ll}
0 & \mbox{if $\ell<rs$}\\
\mathrm{span}\{S_{r,s}(t)u_1 \} & \mbox{if $\ell = rs$} 
\end{array}
\right.  \]
and therefore 
\begin{equation}\label{containment1}
J(c,h,k) \supseteq \langle S_{r,s}(t) u_1 \rangle.
\end{equation}

(ii) 
If $t = \pm 1$, $r\neq s$, and $\kappa_{r,s}^\pm$ is as given in Theorem \ref{degree-d-theorem}, then   
\begin{multline}\label{J-pm-(ii)}
J(c,h,k)(\ell) = \\
\left\{ \begin{array}{ll}
0 & \mbox{if $\ell<rs$}\\
 \mathrm{span} \{ S_{r,s}(\pm1)u_n +  \sum_{i = 1}^{n-1} R_{r,s}^{\pm, i}u_{n-i} \; | \; n = 1, \dots, k \} & \mbox{if $\ell = rs$ and $k\leq \kappa_{r,s}^\pm$} \\
  \mathrm{span} \{ S_{r,s}(\pm1)u_n +  \sum_{i = 1}^{n-1} R_{r,s}^{\pm, i}u_{n-i} \; | \; n = 1, \dots, \kappa_{r,s}^\pm\} & \mbox{if $\ell = rs$ and $k> \kappa_{r,s}^\pm$}
  \end{array} \right. 
  \end{multline}
where in the last line $\kappa_{r,s}^\pm$ is assumed to be finite, and where the $R_{r,s}^{\pm, i} \in \mathcal{R}_{rs}$ are uniquely determined modulo $S_{r,s}(\pm 1)$ for $i = 1, \dots, \kappa_{r,s}^\pm$.
Furthermore 
\begin{equation}\label{containment2}
J(c,h,k) \supseteq 
\left\{ \begin{array}{ll}
 \langle S_{r,s}(\pm1)u_n +  \sum_{i = 1}^{n-1} R_{r,s}^{\pm, i}u_{n-i} \; | \; n = 1, \dots, k \rangle & \mbox{if $k\leq \kappa_{r,s}^\pm$} \\
  \langle S_{r,s}(\pm1)u_n +  \sum_{i = 1}^{n-1} R_{r,s}^{\pm, i}u_{n-i} \; | \; n = 1, \dots, \kappa_{r,s}^\pm\rangle & \mbox{if $k> \kappa_{r,s}^\pm$}
  \end{array} \right. .
\end{equation}
\\

Case (2):  If $T(c,h) = \langle S_{r,s}(t) {\bf 1}_{c,h}, S_{r's'}(t) {\bf 1}_{c,h} \rangle$ with $rs<r's'$, we have that  
\[J(c,h,k)(\ell) = \left\{ \begin{array}{ll}
0 & \mbox{if $\ell<rs$}\\
\mathrm{span}\{S_{r,s}(t)u_1 \} & \mbox{if $\ell = rs$} 
\end{array}
\right.  \]
and 
\begin{equation}\label{containment4}
J(c,h,k) \supseteq \langle S_{r,s}(t) u_1, S_{r',s'}(t) u_1\rangle.
\end{equation} 
\end{cor}

\begin{proof}
By Proposition \ref{determinant-prop} and Theorem \ref{det-thm} in Case (0) for all $\ell \in \mathbb{Z}_{>0}$, or in Cases (1)--(2) for all $\ell<d$, we have that  $\det \mathcal{A}_\ell(c,h)\neq 0$, which implies  $\det \mathfrak{A}_\ell^{(k)}(c,h) \neq 0$, so it follows that $J(c,h,k)(\ell) = Ker\,  \mathfrak{A}_\ell^{(k)}(c,h) = 0$. This finishes the proof for Case (0), i.e., $J(c,h,k) = 0$ in this case, and the proof for Case (1)-(2) for degree $\ell<rs$.

For $\ell = rs$, the result follows from Theorem \ref{degree-d-theorem}.

For $\ell>rs$, the containments in Eqns.\ (\ref{containment1}), (\ref{containment2}), and  (\ref{containment4}) follow from the fact that $J$ is a $V_{Vir}(c,0)$-submodule of $M(U(c,h,k))=\overline{M}(U(c,h,k))$ as explained in Remark \ref{Vir-M=Mbar}.
\end{proof}

In the next Proposition, we extend the result of Corollary \ref{J-theorem} to include information on the degree $d + 1 = rs + 1$ space if $T(c,h) \neq 0$ with $d = rs$ the degree of the lowest singular vector, and in Case (2), with the additional assumption that the second generator of $T(c,h)$ is not at degree $d + 1$.  However, outside of these cases more complicated phenomenon occurs, some of which we will illustrate in examples in Section \ref{examples-subsection} below.

\begin{prop}\label{d+1-prop} 
Let $W(c,h, k)(0)= \mathrm{span}\{u_1, \dots, u_k\} \cong U(c,h,k) = \mathbb{C}[x]/((x-h)^k)$ and let $W(c,h,k) = \mathfrak{L}_0(U(c,h,k)) = M(c,h,k)/J(c,h,k)$.  Then we have the following:\\

Case (1):  If $T(c,h) = \langle S_{r,s}(t){\bf 1}_{c,h}\rangle$, then $J(c,h,k)(rs + 1) = L_{-1}.J(c,h,k)(rs)$.\\

Case (2): If $T(c,h) = \langle S_{r,s}{\bf 1}_{c,h} (t), S_{r',s'}(t){\bf 1}_{c,h} \rangle$ and $r's' > rs +1$, then 
\[J(c,h,k)(rs +1) = L_{-1}.J(c,h,k)(rs).\]
\end{prop}

\begin{proof} 
Let $S_{r,s}(t) {\bf 1}_{c,h}$ for $rs = d$ be the lowest degree singular vector of  $M(c,h)$, and if $M(c,h)$ falls into Case (2), assume that the second generator of $T(c,h)$ is at degree strictly greater than $d+1 = rs +1$ (in particular, $(c,h) \neq (0,0))$.  In this setting,  $\mathcal{A}_{d+1}(c,h)$ is an $n \times n$ matrix for $n = \mathfrak{p}(d+1)$, and we claim that the rank of $\mathcal{A}_{d+1}(c,h)$ is $\mathfrak{p}(d+1) - 1$.  Thus Proposition \ref{Jacobi-lemma} applies, and we can follow the methods used to prove Theorem \ref{degree-d-theorem}.
The claim that the rank of  $\mathcal{A}_{d+1}(c,h)$ is $\mathfrak{p}(d+1) - 1$  follows from these facts: $[\mathbb{C}S_{r,s}(t)]_{\mathcal{B}_d}= Ker \mathcal{A}_{d}(c,h)$ and so the $(d+1)$-degree subspace of $\langle S_{r,s}(t) {\bf 1}_{c,h} \rangle$, which is  $\mathbb{C}L_{-1}S_{r,s}(t) {\bf 1}_{c,h}$, corresponding to the vector $[\mathbb{C} L_{-1} S_{rs}(t)]_{\mathcal{B}_{d+1}}$ is in $Ker \mathcal{A}_{d+1}(c,h)$.  However, in fact,  $Ker \mathcal{A}_{d+1}(c,h) =  [\mathbb{C} L_{-1} S_{r,s}(t)]_{\mathcal{B}_{d+1}}$, because $\Phi_{r',s'}(c,h)$ does not divide $\det \mathcal{A}_{rs +1}(c,h)$ for any $r',s'$ not equal to $r,s$, since we are assuming any such $r',s'$ satisfies $r's'> rs+ 1 = d+ 1$. Thus the claim holds.

Let $v \in M(c,h,k)(d+1)$ for $d = rs$. Then we can write $v = \sum_{j = 1}^k R_{d+1}^ju_j$ for some $R_{d+1}^1, \dots, R_{d+1}^k \in \mathcal{R}_{d+1}$, and by Proposition \ref{derivative-equations-prop}, $v \in Ker \, \mathfrak{A}_{d+1}^{(k)}$ is equivalent to the  Eqns.\ (\ref{der1})-(\ref{der-last}) holding for $\ell = d +1 = rs +1$.

If $R_{d+1}^n \neq 0$ and $R^j_{d+1} = 0$ for $j>n$, then this is equivalent to  Eqns.\ (\ref{der-n})-(\ref{der-last}) holding for $k-m = n$. Then Eqn.\ (\ref{der-n}) at $k-m=n$ implies that $[R_{d+1}^n]_{\mathcal{B}_{d+1}} \in Ker \, \mathcal{A}_{d+1}(c,h)$, i.e., 
$R_{d+1}^nu_n \in \langle S_{r,s}(t) u_n
\rangle$, i.e., up to a scalar multiple $R_{d+1}^n = L_{-1} S_{r,s}(t)$.  

The next equation in Proposition \ref{derivative-equations-prop} implies 
\begin{equation}\label{b-derivative2}
0=   \mathcal{A}_{d+1}(x,y) [R_{d+1}^{n-1}]_{\mathcal{B}_{d+1}} +  \left( \frac{\partial}{\partial y} \mathcal{A}_{d+1}(x,y)\right)  [L_{-1}S_{r,s}(t)]_{\mathcal{B}_{d+1}} \Big|_{(x,y) = (c(t), h_{r,s}(t))} 
 \end{equation}
must hold.

By Proposition \ref{Jacobi-lemma} with $A(y) = \mathcal{A}_{d+1}(x,y)|_{x = c(t)}$ and $h = h_{r,s}(t)$, we have that  Eqn.\ (\ref{b-derivative2}) has a solution $\vec{b} = [R_{d+1}^{n-1}]_{\mathcal{B}_{d+1}}$ if and only if  
\[ \frac{\partial}{\partial y} \det \mathcal{A}_{d+1}(x,y) \Big|_{(x,y) = (c(t), h_{r,s}(t))} = 0. \]
Therefore, by Proposition \ref{Phi-derivative-prop}, and the assumption in Case (2) that $r's'>d+1$, 
there is a solution to Eqn.\ (\ref{b-derivative2}) if and only if $t = \pm 1$ and $r \neq s$. This proves that in Case (1)(i) and Case (2) of Corollary \ref{J-theorem}, we have $J(c,h,k)(rs + 1) = \mathrm{span} \{L_{-1}S_{r,s}(t) u_1\} = L_{-1}.J(c,h,k)(rs)$, proving the result in these cases.

For Case (1)(ii), we let $t = \pm 1$ and $r \neq s$. Then by Propositions \ref{Phi-derivative-prop} and \ref{Jacobi-lemma}, Eqn.\ (\ref{b-derivative2}) does have a solution, and that solution is exactly $[R^{n-1}_{d+1}]_{\mathcal{B}_{d+1} }= [L_{-1}R_{r,s}^{\pm,1}]_{\mathcal{B}_{d+1}}$ which is unique up to $Ker \, \mathcal{A}_{d+1}(c,h) =  [\mathbb{C}L_{-1}S_{r,s}(\pm1)]_{\mathcal{B}_{d+1}}$ since Rank $\mathcal{A}_{d+1}(c,h) = \mathfrak{p}(d+1) - 1$, i.e., $\dim Ker \, \mathcal{A}_{d+1}(c,h) = 1$.  By the same argument, each subsequent equation in (\ref{der1})-(\ref{der-last}) has a solution $v = \sum_{j=1}^n R_{d+1}^ju_j$ for $1\leq n \leq \kappa_{r,s}^\pm$ only if $v = L_{-1}S_{r,s}(\pm 1)u_n + \sum_{j=1}^{n-1} L_{-1}R_{r,s}^{\pm, j}u_j$, in which case this is indeed in $J(c,h,k)$ if $k\leq \kappa_{r,s}^\pm$.  The result follows.   
\end{proof} 

\begin{rem}
Case (0) is of course not mentioned in Proposition \ref{d+1-prop} because this is the trivial case when $J(c,h,k) = \coprod_{\ell \in \mathbb{Z}_{\geq 0}} J(c,h,k)(\ell) = 0$.
\end{rem}

\begin{rem} In light of Proposition \ref{d+1-prop}, one might be tempted to conjecture that the containments in Eqns. (\ref{containment1}), (\ref{containment2}), and (\ref{containment4}) are equalities. However this is not the case in general. For instance, we give a counter example below in Case (1)(i). One conceptual reason to see why this is  the case is to observe that techniques such as the use of the Jacobi Formula to prove Proposition \ref{Jacobi-lemma}, which is the main technique used in both Theorem \ref{degree-d-theorem} and Proposition \ref{d+1-prop}, rely heavily on the fact that at degree $d = rs$ where the first singular vector occurs and at degree $d+ 1 = rs+ 1$ (as long as $r's'>d+1$ in Case (2)) the Gram matrix for the Shapovalov form $\mathcal{A}_\ell (c,h)$ has rank $\ell  - 1$ at  both levels $\ell = d = rs$ and $\ell = d + 1$.  However, for $\ell > d+1$, this matrix has rank strictly less than $\ell-1$ and this gives room for particular discrete values of $(c,h)$ to have additional vectors in $Ker \, \mathfrak{A}^{(k)}_\ell = J(c,h,k)(\ell)$ which are solutions to the corresponding differential equations for $\mathcal{A}_\ell(x,y).$  Furthermore, our techniques up to this point say very little in Case (1)(ii) when $k>\kappa_{r,s}^\pm$ and further techniques are necessary in order handle this case. However, we can characterize $J$ completely when $k\leq \kappa_{r,s}^\pm$ as we show in the next result.   
\end{rem}

\begin{thm}\label{J=Case1ii}
Let $(c,h)$ be in Case (1)(ii), i.e., $t = \pm1$, $r \neq s$, and $\kappa_{r,s}^\pm$ is as given in Theorem \ref{J-theorem}.  Let 
$W(c,h, k)(0)= \mathrm{span}\{u_1, \dots, u_k\} \cong U(c,h,k) = \mathbb{C}[x]/((x-h)^k)$ and let $W(c,h,k) = \mathfrak{L}_0(U(c,h,k)) = M(c,h,k)/J(c,h,k)$.  Then if $k\leq \kappa_{r,s}^\pm$
\begin{equation}\label{equality}
J(c,h,k) = 
 \langle S_{r,s}(\pm1)u_n +  \sum_{i = 1}^{n-1} R_{r,s}^{\pm, i}u_{n-i} \; | \; n = 1, \dots, k \rangle,
\end{equation}
where the $R_{r,s}^{\pm, i} \in \mathcal{R}_{rs}$ are uniquely determined modulo $S_{r,s}(\pm 1)$ for $i = 1, \dots, \kappa_{r,s}^\pm$.
\end{thm}

\begin{proof} By Eqn.\ (\ref{containment2}), we have that the right hand side of Eqn.\ (\ref{equality}) is contained in $J(c,h,k)$. To prove the other containment, we denote by $J_S$ the right hand side of Eqn.\ (\ref{equality}), and thus $J_S$ is a submodule of $J(c,h,k)$. Note that by Lemma \ref{S-lemma}, if $v_j = \sum_{i = 1}^j R_\ell^i u_i \in J(c,h,k)(\ell)$, with $R_\ell^j \neq 0$, then $R_\ell^ju_j = R_{\ell - rs}S_{r,s}(\pm1) u_j$ for some $R_{\ell - rs} \in \mathcal{R}_{\ell-rs}$.  That is, $v_j$ must be of the form $v_j = R_{\ell-rs}S_{r,s}(\pm1)u_j + \sum_{i = 1}^{j-1} R_{\ell}^iu_i$.   

We prove the result by induction on $j$.  If $j = 1$, then $v_1 \in \mathcal{R}_{\ell-rs}.S_{r,s}(\pm1)u_1 = \langle S_{r,s}(\pm1)u_1 \rangle \subset J_S$. Assume that $w = \tilde{R}_{\ell - rs}S_{r,s}(\pm1)u_m + \sum_{i = 1}^{m-1} \tilde{R}_\ell^i u_i  \in J(c,h,k)$ for $m<j$, and some $\tilde{R}_{\ell-rs} \in \mathcal{R}_{\ell-rs}$ and $\tilde{R}_\ell^i \in \mathcal{R}_\ell$ for $i = 1,\dots, m-1$, then $w \in J_S$.  Therefore since $v' = R_{\ell - rs} (S_{r,s}(\pm1) u_j + \sum_{i = 1}^{j-1} R_{r,s}^{\pm, i} u_i) \in J_S \subseteq J(c,h,k)$, if $v_j \in J(c,h,k)$, then $v_j - v' = \sum_{i = 1}^{j-1}R_\ell^i u_i - R_{\ell-rs}\sum_{i = 1}^{j-1} R^{\pm,i}_{r,s} u_i \in J(c,h,k)$.  If $v_j - v' = 0$, then $v_j = v' \in J_S$.  Otherwise, there is some maximal integer $m$ with $1<m<j$ such that $R_\ell^m - R_{\ell -rs}R_{r,s}^{\pm,m} \neq 0$.  Since $v_j - v' \in J(c,h,k)$, by Lemma \ref{S-lemma}, we must have that $R_\ell^m - R_{\ell-rs}R_{r,s}^{\pm, m} = \hat{R}_{\ell - rs}S_{r,s}(\pm1)$ for some $\hat{R}_{\ell-rs} \in \mathcal{R}_{\ell-rs}$. By the induction hypothesis, this implies that $v_j - v' \in J_S$. The result follows by linearity in $J_S$.
\end{proof}

\subsection{Examples of low degree behavior in Case (1).}\label{examples-subsection}

In this Subsection, we give some illustrative examples to show some of the behaviors that occur at low degrees in Case (1) since this is the more interesting case of varied behavior, and as we shall see below in Section \ref{graded-pseudo-trace-section}, the case in which the existence of interlocked modules, and hence graded pseudo-traces is much more subtle.\\

 {\bf Example for Case (1)(i):} Let $t = -2$ (or equivalently, for this conformal weight, $t = -1/2$), and $r = s =  1$. Then $c = -2$ and $h = 0$.  Since $c = c_{2,1} \neq c_{p,q}$ for $p, q \in \{2,3,\dots\}$ relatively prime, this is in Case (1), and more precisely in Case(1)(i) of Corollary \ref{J-theorem}.  Then $J(c,0,k)(1) = \mathbb{C}L_{-1}u_1$.  In addition, $J(c,0,k)(2) = \mathbb{C}L_{-1}^2 u_1$ by Proposition \ref{d+1-prop}.

However, at degree $\ell = 3$, we have that 
\begin{eqnarray}\label{h=0-example}
J(-2,0,k)(3) &=& \mathrm{span} \{L_{-1}^3u_1,\  L_{-2}L_{-1}u_1, \ (L_{-1}^2-2L_{-2})L_{-1}u_2 + L_{-3}u_1   \}\\ 
&=& \langle S_{1,1}(-2)u_1\rangle(3) \cup  \mathbb{C}(S_{2,1}(-2)S_{1,1}(-2)u_2 + L_{-3}u_1 ).\nonumber\\
 &\neq&  L_{-1}.J(-2, 0, k)(2).
\end{eqnarray}

This, and more generally the case of what happens if $h = 0$ in Case (1), follows by considering $v \in J(c,0,k)(3) = Ker \, \mathfrak{A}_{3}^{(k)}(c,0)$. Then $v = \sum_{j = 1}^k R_{3}^ju_j$ for $R_{3}^1, \dots, R_{3}^k \in \mathcal{R}_{3}$. Let  $n$ for $1\leq n \leq k$ be such that $R_{3}^j = 0$ for $j>n$ and $R_3^n \neq 0$. Then by Lemma \ref{S-lemma}, we have that $R_{3}^nu_n  = R_2 S_{1,1}(t) u_n = R_2L_{-1}u_n$ for some 
$R_2 \in \mathcal{R}_2$. Therefore $v$ must be of the form 
$v = \sum_{j=1}^{n-1} R_{3}^iu_i + (aL_{-1}^2 + bL_{-2}) L_{-1}u_n$, for some constants $a,b \in \mathbb{C}$ not both zero.

But then the requirement that $T^\dag_3 v = 0$ for all $T_3 \in \mathcal{R}_3$ implies that either $n = 1$ and $v \in \langle L_{-1}u_1 \rangle$, or $n>1$ and $\pi_{j}^0(T_3^\dag v)= 0$ for all $j = 1,\dots, n$. Thus by Proposition \ref{derivative-equations-prop} we must have 
\begin{equation}\label{example}
\mathcal{A}_3(c,0)[R_3^{n-1}]_{\mathcal{B}_3} = -\frac{\partial}{\partial y} \mathcal{A}_3(c,y)[R^n_3]_{\mathcal{B}_3}\Big|_{y = 0}
\end{equation}
for $\vec{s} = [R^n_3]_{\mathcal{B}_3} = [(aL_{-1}^2 + bL_{-2})L_{-1}]_{\mathcal{B}_3} = (a, b, 0)^T$, having a solution $\vec{b} = [R_3^{n-1}]_{\mathcal{B}_3}$.  But this matrix equation, Eqn.\ (\ref{example}), is equivalent to
\[\begin{bmatrix}
0 & 0 & 0\\
0 & 0 & 0 \\
0 & 0 & 2c
\end{bmatrix} \vec{b}  = -\begin{bmatrix} 
24 & 12 & 24\\
12 & 8+c & 10\\
24 & 10 & 6
\end{bmatrix} \begin{bmatrix}
a\\b\\0
\end{bmatrix},\]
which has a solution $\vec{b} = [R_3^{n-1}]$ if and only if $c = -2$ and, up to a scalar multiple $\vec{s} = (a,b,0)^T = (1,-2,0)^T = [(L_{-1}^2 - 2L_{-2})L_{-1}]_{\mathcal{B}_3}$, in which case, $\vec{b} = (0,0,1) = [L_{-3}]_{\mathcal{B}_3}$. Thus $v$ must be of the form
$v = \sum_{j = 1}^{n-2} R_3^j u_j + L_{-3}u_{n-1} + (L_{-1}^2 - 2L_{-2})L_{-1}u_n$ and we must have $c = -2$.

Thus if $n = 2$, and $c = -2$, we have  $L_{-3}u_1 + (L_{-1}^2 - 2L_{-2})L_{-1})u_2 \in J(c,h,k)(3)$.

If $n>2$, and $c = -2$, the requirement that $\pi_{j-2}(T^\dag_3v) = 0$ for all $T_3 \in \mathcal{R}_3$ is equivalent to
\begin{equation}\label{example2}
\mathcal{A}_3(-2,0)[R_3^{n-2}]_{\mathcal{B}_3} =   -\left. \Big(\frac{\partial}{\partial y} \mathcal{A}_3(-2,y)\Big) \vec{b} - \frac{1}{2!}\left(\Big(\frac{\partial}{\partial y} \Big)^2 \mathcal{A}_3(-2,y)\right) \vec{s} \right|_{y = 0}.
\end{equation}
But  Eqn.\ (\ref{example2}) is equivalent to
\[\begin{bmatrix}
0 & 0 & 0\\
0 & 0 & 0 \\
0 & 0 & -4
\end{bmatrix} [R_3^{n-2}]_{\mathcal{B}_3}  = -\begin{bmatrix} 
24 & 12 & 24\\
12 & 6 & 10\\
24 & 10 & 6
\end{bmatrix} \begin{bmatrix}
0\\0\\1
\end{bmatrix} -\begin{bmatrix} 
72 & 36 & 0\\
36 & 8 & 0\\
0 & 0 & 0
\end{bmatrix}\begin{bmatrix}
1\\-2\\0
\end{bmatrix} = \begin{bmatrix}
-24\\-30\\-6
\end{bmatrix}  ,\]
which obviously has no solution.

From this analysis above for degree 3, we have proven the following Corollary.

\begin{cor}
The submodule $J(c,0,k) \subset W(c,0,k)$ is of the form $J(c,0,k)(3) =  \langle L_{-1}u_1 \rangle(3) = \mathrm{span}\{L_{-1}^3u_1, L_{-2}L_{-1} u_1 \}$ if and only if $c \neq -2$.  If $c = -2$, then $J(c,0,k)$ is given by Eqn.\ (\ref{h=0-example}).  
\end{cor}

{\bf Example for Case (1)(ii):} 
Let $t = \pm 1$, and $(c,h) = (25, -5/4)$ or $(1, 1/4)$, respectively. Then $(c,h) \in \Phi_{2,1}(c,h)$ and we have that 
$[S_{2,1}(\pm1)]_{\mathcal{B}_2} \in Ker\, \mathcal{A}_2(c,h)$ and so $S_{2,1}(\pm1)u_1 \in Ker \, \mathfrak{A}_2^{(k)} = J(c,h,k)(2)$.  Then we also have that for  $1<n\leq k$
\begin{equation}\label{example3}
\mathcal{A}_2(c(\pm1),h_{2,1}(\pm1))[R_2^{n-1}]_{\mathcal{B}_2} = -\frac{\partial}{\partial y} \mathcal{A}_2(c(\pm 1),y)[R^n_2]_{\mathcal{B}_2}\Big|_{y = h_{2,1}(\pm1)}
\end{equation}
has a solution by Theorem \ref{degree-d-theorem} for $[R^n_2]_{\mathcal{B}_2} \in Ker\, \mathcal{A}_2(c,h)$, because $t = \pm 1$.  Indeed, when $t = -1$, for instance,  Eqn.\ (\ref{example3}) is equivalent to
\begin{equation}\label{find-b} \frac{3}{2} \begin{bmatrix}
1 & 1 \\
1 & 1
\end{bmatrix} \vec{b}  = -\begin{bmatrix} 
8 & 6\\
6  & 4
\end{bmatrix} \begin{bmatrix}
a\\-a
\end{bmatrix},
\end{equation}
and has any $\vec{b} = (b_1, b_2) \in  \mathbb{C}^2$ such that $b_1 + b_2 = -\frac{4}{3}a$ as a solution.  Thus letting $n = 2$, and $k\geq 2$, we have that $(-b -\frac{4}{3}) L_{-1}^2u_1 + b L_{-2}u_1 + (L_{-1}^2 - L_{-2})u_2 \in J(1,1/4,k)(2)$ for any $b \in \mathbb{C}$, proving that
\begin{equation}\label{for-c=1}\mathrm{span}\{S_{2,1}(-1)u_1, \ S_{2,1}(- 1)u_2  - \frac{4}{3}L_{-1}^2u_1 \} \subset J(1,1/4,k)(2),
\end{equation}
and $\kappa^-_{2,1} \geq 2$.

Next we consider 
\[\mathcal{A}_2(1,1/4)[R_2^{n-2}]_{\mathcal{B}_2} =   -\left. \Big(\frac{\partial}{\partial y} \mathcal{A}_2(1,y)\Big) \vec{b} - \frac{1}{2!}\left(\Big(\frac{\partial}{\partial y} \Big)^2 \mathcal{A}_2(1,y)\right) \vec{s} \right|_{y = 1/4}\]
which is equivalent to 
\begin{equation}\label{no-solutions} \frac{3}{2} \begin{bmatrix}
1 & 1 \\
1 & 1
\end{bmatrix} [R_2^{n-2}]_{\mathcal{B}_2}  = -\begin{bmatrix} 
8 & 6\\
6  & 4
\end{bmatrix} \begin{bmatrix}
b_1\\b_2
\end{bmatrix} -  \frac{1}{2}\left[\begin{array}{cc}16&0\\0&0\end{array}\right] \begin{bmatrix}
a\\-a
\end{bmatrix},
\end{equation} 
for $\vec{b} = (b_1, b_2) = (-b_2 - \frac{4}{3} a, b_2)$ a solution to Eqn.\ (\ref{find-b}), and $\vec{s} = (a, -a) \in Ker \mathcal{A}_2(c,h)$. Then Eqn.\ (\ref{no-solutions}) has a solution if and only if $a = 0$. Thus the only solution is 
already in the left hand side of Eqn.\ (\ref{for-c=1}), so that no vectors involving $u_j$ for $j>2$ are in $J(1,1/4,k)(2)$, and $\kappa_{2,1}^- = 2$.    Therefore  for $k\geq 2$
\[
J(1,1/4,k) = \langle S_{2,1}(-1)u_1, \ S_{2,1}(-1) u_2 - \frac{4}{3} L_{-1}^2u_1 \rangle.
\]

The proof for $t = 1$, i.e., $c = 25$ and $h = h_{2,1}(1) = -5/4$ is analogous, and in this case, we also have that $\kappa^+_{2,1} = 2$, and for $k\geq 2$
\[J(25,-5/4,k) = \langle S_{2,1}(1)u_1, \ S_{2,1}( 1)u_2 + \frac{4}{3}L_{-1}^2u_1  \rangle.\]

\section{Classification of strongly interlocked $V_{Vir}$-modules induced from the level zero Zhu algebra}\label{Virasoro-classification-subsection}

In this section, we characterize when  the modules $W(c,h,k)$ for $V_{Vir}(c,0)$ induced from the level zero Zhu algebra are interlocked. We also prove that when such a module $W(c,h,k)$ is interlocked, then it is strongly interlocked, and has well-defined graded pseudo-traces. 

In analogy with the Heisenberg vertex operator algebra, the $A_0(V_{Vir}(c,0))$-modules $U(c,h,k)$ are interlocked for all $(c,h)\in \mathbb{C}^2$ and $k \in \mathbb{Z}_{>0}$ and the proof is exactly analogous to that for the Heisenberg algebra since $U(c,h,k) = U(h, k)$. However for $k>1$, we will show that the $V_{Vir}(c,0)$-module $W(c,h,k) = \mathfrak{L}_0(U(c,h,k))$ will be interlocked if and only if $(c,h) \in \mathbb{C}^2$ are such that $T(c,h)= 0$, i.e., $M(c,h)$ falls in Case (0), or $(c,h)$ falls in Case (1)(ii) and $k\geq \kappa_{r,s}^\pm$. To prove this, we first observe that for all $c \in \mathbb{C}$:\\
(i) $V_{Vir}(c,0) = \langle \omega \rangle$.\\
(ii) $A_0(V_{Vir}(c,0) \cong \mathbb{C}[o(\omega)]$. \\
Thus we can consider how Theorem \ref{extra-conditions-theorem} applies to this setting.   

If $T(c,h) = 0$, then $\det \mathcal{A}_\ell(c,h) \neq 0$, i.e. the Shapovalov form at each level $\ell \in \mathbb{Z}_{\geq 0}$ is nondegenerate.  This implies that for $\{u_1, \dots, u_k\}$ a Jordan basis for $o(\omega) = L_0$ acting on $U(c,h,k)$, then without loss of generality, we can let $u_1 = \mathbf{1}_{c,h}$ and $< \cdot, \cdot >_\ell$ is the Shapovalov form at level $\ell$.  Thus we have the following Corollary to Theorem \ref{extra-conditions-theorem}:

\begin{cor}\label{Case-0-cor}
Let $U = U(c, h, k) = \mathbb{C}[x]/((x- h)^k)$  for $k\geq 1$. If $T(c,h) = 0$, then the $V_{Vir}(c,0)$-modules $W = W(c,h,k) = \mathfrak{L}_0(U(c,h, k)) = M_0(U(c,h,k))$ are  strongly interlocked, and $\mathrm{pstr}_W(v ,\tau)$ is well defined for every $v \in V_{Vir}(c,0)$. 

In particular, letting $\{u_1, \dots, u_k\}$ be a Jordan basis for $o(\omega) = L_0$ acting on $U(c,h,k)$, and setting $W^{(j)} = \mathfrak{L}_0(\mathrm{span}\{u_1, \dots, u_j\})$ for $j = 0,\dots,k$, we have that $W^{(j)}$ is interlocked with $W^{(k-j)}$, and  
$\langle u_1 \rangle = Soc(W(c, h,k))$ is interlocked with  $\langle u_1, \dots, u_{k-1} \rangle = Rad(W(c,h,k))$.

 Furthermore, each $o(v)$ acting on $W(c,h, k)(\ell)$ for $v \in V_{Vir}(c,0)$ has a $k \mathfrak{p}(\ell) \times k\mathfrak{p}(\ell)$ matrix realization in this basis so that $o(v)|_{W(c,h,k)(\ell)}$ 
can be decomposed as in Eqn.\ (\ref{Our-decomp}), with $A = o(v)|_{Soc(W(c,h, k))(\ell)}$ a $ \mathfrak{p}(\ell) \times \mathfrak{p}(\ell)$ matrix, since $\dim Soc(W)(\ell) = \mathfrak{p}(\ell)$. 
\end{cor}

To determine which $V_{Vir}(c,0)$-modules are strongly interlocked if $T(c,h) \neq 0$, i.e. Case (1) or Case (2), we first calculate the socle and \sout{Jacobson} radical of the induced modules.  For comparison, we include Case (0) in our analysis even though Corollary \ref{Case-0-cor} covers this case.   

\begin{prop}\label{soc-rad-prop}
Let $U = U(c, h, k) = \mathbb{C}[x]/((x- h)^k)$  for $k\geq 2$,  $J = J(c,h,k)$, and $W = W(c, h, k)  = \mathfrak{L}_0(U) =  M(c,h,k)/J$. Then, for all $(c,h)$, the socle of the $V_{Vir}(c,0)$-module $W$ is given by
\[Soc(W)  = ( \langle u_1 \rangle  + J)  / J\cong L(c,h) = M(c,h)/T(c,h).\]

For the radical of $W$, we have:

In Case (0), i.e., if $T(c,h) = 0$, then recalling that in this case $J = 0$,
\[Rad(W) = \langle u_1, \dots, u_{k-1} \rangle.\]

In Case (1), i.e., if $T(c,h) = \langle S_{r,s}(t){\bf 1}_{c,h} \rangle$ for some $r,s$, then 
\[Rad(W) = (\langle u_1, \dots, u_{k-1}, S_{r,s}(t)u_k \rangle  + J) /J .\]

In Case (2), i.e., if $T(c,h) = \langle S_{r,s}(t){\bf 1}_{c,h}, S_{r',s'}(t) {\bf 1}_{c,h} \rangle$ for some $r,s, r', s'$, then 
\[Rad(W) = (\langle u_1, \dots, u_{k-1}, S_{r,s}(t)u_k, S_{r',s'}(t)u_k\rangle)  + J)  /J.\]  
\end{prop} 

\begin{proof}  
For all cases, since $u_1$ generates a simple module, it is clear that $\mathbb{C}u_1\subset Soc(U)$. Assume there is another simple submodule $\mathbb{C}[x].v\subset Soc(U)$. Then, by Eqn.\ (\ref{L_0-on-u}) it follows that $(L_0-hI)^i v\in \mathbb{C}u_1$ for some $0\leq i\leq k$. Therefore, $\mathbb{C}u_1$ is the only simple submodule and we have that $Soc(U)=\mathbb{C}u_1$. Since the functor $\mathfrak{L}_0$ gives a bijection between simple $A_0(V_{Vir})$-modules and simple $V_{Vir}$-modules, it follows that $Soc(W) = \mathfrak{L}_0(Soc(U)) = \mathfrak{L}_0(\mathbb{C}u_1).$  Since $\mathbb{C}u_1$ is the only irreducible module for the level zero Zhu algebra with conformal weight $h$,  $Soc(W) =\mathfrak{L}_0(\mathbb{C}u_1)$ is the unique simple $V_{Vir}$-module of conformal weight $h$, i.e., $Soc(W) \cong L(c,h)$. 
 
For $Rad(W)$, we first note that viewing $U = \mathrm{span} \{ u_1, \dots, u_k\}$ as an $A_0(V_{Vir})$-module, it follows from linear algebra that $Rad(U) = \mathrm{span}_ \mathbb{C} \{ u_1, \dots, u_{k-1} \}$.  Thus the induced module $\mathfrak{L}_0( Rad(U)) =  (\langle  u_1, \dots, u_{k-1} \rangle  + J)  / J = W'$ is a proper $V_{Vir}$-submodule of $W = \mathfrak{L}_0(U)$, and $W'$ is contained in some maximal submodule $M$ of $W$. In fact, we will show that $W'$ must be contained in every maximal submodule $M$ of $W$. 

Case (0) follows from Corollary \ref{Case-0-cor}. 

For Case (1), by Lemma \ref{derivative-lemma}, for any $T_d \in \mathcal{R}_d$ for $d=rs$, we have that $T_d^\dag S_{r,s}(t) u_k  + J  = a_{T_d, S_{r,s}(t)} (c,h)  u_k + u' + J$ for some $u'  + J\in (\langle u_1, \dots, u_{k-1} \rangle + J) / J= W'$. However, since $< T_d {\bf 1}_{c,h} , S_{r,s}(t) {\bf 1}_{c,h} > =   a_{T_d, S_{r,s}(t)}(c,h)  = 0$  for all $T_d\in \mathcal{R}_d$ if and only if $S_{r,s}(t) {\bf 1}_{c,h}$ is a singular vector in $Ker \, \mathcal{A}_d(c,h) = T(c,h)$ if and only if $R_d = S_{r,s}(t)$, we have that  any descendant of $S_{r,s}(t) u_k + J$ is in $W' = (\langle u_1, \dots, u_{k-1} \rangle + J)  / J$, and no other $R_\ell u_k + J \in Rad(W)$ for $\ell \leq d$. Thus $(\langle u_1, \dots, u_{k-1}, S_{r,s}u_k  \rangle + J) /J = Rad(W)$.  

Finally, for Case (2), by the argument above, $S_{r,s} u_k + J, \ S_{r',s'} u_k + J \in Rad(W)$ and any other vectors in $Rad(W)$ must be descendants of these,  and therefore  $(\langle u_1, \dots, u_{k-1}, S_{r,s}u_k, S_{r',s'} u_k \rangle + J)  / J = Rad(W)$. 
\end{proof}

\begin{thm} \label{interlocked-thm} 
For all $c,h \in \mathbb{C}$ and $k\in \mathbb{Z}_{>0}$, the $V_{Vir}(c,0)$-module $W = W(c,h,k)$  satisfies
\begin{equation}\label{W/Rad=Soc}
W/Rad(W) \cong Soc(W).
\end{equation}
 
\noindent Moreover, the equation
\begin{equation}\label{W/Soc=Rad}
W/Soc(W) \cong Rad(W)
\end{equation}
trivially holds when $k = 1$, i.e., for $W=W(c,h,1)$ and for all $(c,h) \in\mathbb{C}^2$, whereas for $k>2$, Eqn.\ (\ref{W/Soc=Rad}) is satisfied  
if and only if one of the following holds:

Case (0) holds, i.e., $T(c,h) = 0$;

Case (1)(ii) holds and $k \leq \kappa_{r,s}^\pm$, where $\kappa_{r,s}^\pm$ is as defined in Theorem \ref{degree-d-theorem}.  That is, $t = \pm 1$, $r \neq s$, and $\kappa_{r,s}^\pm$ is either the maximum positive integer such that for each $1\leq n\leq \kappa^\pm_{r,s}$, there exist $R_d^1, \dots, R_d^n \in \mathcal{R}_d$ with $R_d^n \neq 0$ that give a solution to Eqn.\ (\ref{kappa})
and if $n>\kappa^\pm_{r,s}$, then there is no solution to Eqn.\ (\ref{kappa}) with $R_d^n \neq 0$, or no such positive integer $\kappa_{r,s}^\pm$ exists and solutions to Eqn.\ (\ref{kappa}) exist for all $n \in \mathbb{Z}_{>0}$, in which case we write $\kappa_{r,s}^\pm = \infty$.
\end{thm}

\begin{proof} We first prove Eqn.\ (\ref{W/Rad=Soc}) holds for $W = W(c,h,k)$ all $(c,h) \in \mathbb{C}^2$ and $k \in \mathbb{Z}_{\geq 0}$.

Case (0) follows from Corollary \ref{Case-0-cor}. 

For Case (1): If $T(c,h) =\langle S_{r,s}(t) {\bf 1}_{c,h} \rangle$, by Proposition \ref{soc-rad-prop}, we have 
 \begin{eqnarray*}
W/Rad(W) &=& W/((\langle u_1, \dots, u_{k-1}, S_{r,s}(t) u_k \rangle + J)  / J) \ = \ (\langle u_k \rangle  + J)/(\langle u_1, \dots, u_{k-1}, S_{r,s}(t) u_k \rangle  + J)\\
& \cong & (\langle u_1 \rangle  + J)  / (\langle  S_{r,s}(t) u_1 \rangle  + J) \ \cong \ M(c,h)/T(c,h) \ = \ L(c,h) \ \cong \ Soc(W),
\end{eqnarray*}
where we have used that $(\langle u_{k}\rangle  + J)/ (\langle u_1, \dots, u_{k-1}, S_{r,s}(t)u_k\rangle  + J$ and $(\langle u_1 \rangle  + J)/ (\langle S_{r,s}(t) u_1 \rangle + J)$ are both highest weight modules for $\mathcal{L}$, with highest weight $h$ and central charge $c$ which are generated by cyclic vectors $u_k + \langle u_1, \dots, u_{k-1}, S_{r,s}(t)u_k\rangle + J$ and $u_1 
 + \langle S_{r,s}(t)u_1 \rangle  + J$,  respectively,  that satisfy $S_{r,s}(t)u_i \equiv 0$ for $i=k, 1$, respectively, modulo the corresponding submodule. By uniqueness of $M(c,h)$, and  $T(c,h)$ for a given $(c,h)$ in Case (1), the result follows.

For Case (2): If $T(c,h) = \langle S_{r,s}(t) {\bf 1}_{c,h},   S_{r',s'}(t) {\bf 1}_{c,h}\rangle$ by Proposition \ref{soc-rad-prop}, we have  
 \begin{eqnarray*}
 W/Rad(W) &= & W/((\langle u_1, \dots, u_{k-1}, S_{r,s}(t)u_k, S_{r',s'}(t)u_k \rangle  + J)  / J)\\
&\cong& \ (\langle u_k \rangle + J) /(\langle  u_1, \dots, u_{k-1}, S_{r,s}(t) u_k , S_{r',s'}(t) u_k \rangle + J) \\
&\cong&  (\langle u_1 \rangle + J)/(\langle S_{r,s}(t) u_1 , S_{r',s'}(t) u_1 \rangle  + J) \ \cong \  
 M(c,h)/T(c,h) \ = \ L(c,h) \ \cong \ Soc(W) ,
\end{eqnarray*}
where we have used that 
\[(\langle u_{k}\rangle + J)/ (\langle u_1, \dots, u_{k-1}, S_{r,s}(t)u_k, S_{r',s'}(t) u_k \rangle  + J) \quad  \mbox{and} \quad (\langle u_1 \rangle  + J) / (\langle S_{r,s}(t) u_1, S_{r',s'}(t) u_1  \rangle + J)\] 
are both highest weight modules for $\mathcal{L}$, with highest weight $h$ and central charge $c$ which are generated by cyclic vectors $u_k + \langle u_1, \dots, u_{k-1}, S_{r,s}(t)u_k, S_{r',s'}(t) u_k \rangle  + J$ and $u_1 
 + \langle S_{r,s}(t)u_1, S_{r',s'}(t) u_1 \rangle  + J$,  respectively,  that satisfy $S_{r,s}(t)u_i \equiv 0$ and $S_{r',s'}(t)u_i \equiv 0$ for $i=k, 1$, respectively, modulo the corresponding submodule. By uniqueness of $M(c,h)$, and  $T(c,h)$ for a given $(c,h)$ in Case (2), the result follows.   

For Eqn.\ (\ref{W/Soc=Rad}), we note that if $k = 1$, then $W(c,h,k) = W(c,h,1) = Soc(W) \cong L(c,h)$ and $Rad(W) = 0$, so that the equation holds.  

For $k>1$, we  next prove the statements for Eqn.\ (\ref{W/Soc=Rad}) hold. 

Case (0) follows from Corollary \ref{Case-0-cor}.

For Case (1)(i) and Case (2): If we are in Case (1)(i), (i.e.,  $t \neq \pm 1$ or $t = \pm1$ and $r = s$), or we are in Case (2), then by Corollary \ref{J-theorem}  at degree $d = rs$ we have $J(c,h,k)(d) = \mathbb{C} S_{r,s}(t)u_1$. Therefore, since $Soc(W) = ((\langle u_1 \rangle + J)  / J)\cong L(c,h)$ and thus $Soc(W)(d) = \mathcal{R}_d.u_1 \smallsetminus \mathbb{C}S_{r,s}(t)u_1 \mod J$, we have that $(W/Soc(W))(d) = (\mathcal{R}_d.\{u_2,\dots, u_k\} + J) / J$. 
However $Rad(W)(d) = ((\mathcal{R}_d.
\{u_1, \dots, u_{k-1} \} + \mathbb{C}S_{r,s}(t)u_k \smallsetminus \mathbb{C}S_{r,s}(t)u_1) + J)  / J$. Therefore the module $W' = W/Soc(W)$ contains an eigenvector $w$ for $L_0$ in $W'(0)$ such that $S_{r,s}(t)w + J$ is a nonzero singular vector in $W'$. Namely, such a vector is given by $w = (u_2 + J) \mod Soc(W)$.

Whereas for $Rad(W)$ no such vector in $Rad (W)(0)$ exists since the only eigenvector for $L_0$ at degree zero is a scalar multiple of $u_1 + J$.  But $S_{r,s}(t)u_1 + J = 0 + J$. Thus $W' = W/Soc(W) \ncong Rad(W)$. This shows that in Case (1)(i) and Case (2), Eqn.\ (\ref{W/Soc=Rad}) holds if and only if $k = 1$.

It is left to show that Eqn.\ (\ref{W/Soc=Rad}) holds for Case (1)(ii) when $k \leq \kappa_{r,s}^\pm$ and does not hold if $k>\kappa_{r,s}^\pm$.  Thus assume that 
$t = \pm 1$ and $r \neq s$. Then by Corollary \ref{J-theorem}, in this case $J(c,h,k)(d)$ is given by Eqn.\ (\ref{J-pm-(ii)}).  
 
Then for any $k \in \mathbb{Z}_{>0}$, we have $Soc(W) = (\langle u_1 \rangle  + J) / J\cong L(c,h)$ and,  if $k \geq 2$, up to a scalar multiple, $u_2 +  J +  Soc(W)$ is the only eigenvector for $L_0$ at degree zero in $W/Soc(W)$, whereas up to a scalar multiple, $u_1 + J$ is the only eigenvector for $L_0$ at degree zero in $Rad(W)$.  

Thus, if there were an isomorphism $\varphi: W/Soc(W) \longrightarrow Rad(W)$ then $\varphi(u_2  + J + Soc(W)) = au_1 + J$ for some nonzero constant $a$. 
Looking at subsequent generalized eigenvectors for $L_0$ and their order of nilpotency with respect to $L_0 - h Id_W$ at degree zero in $W/Soc(W)$ and $Rad(W)$, if such a $\varphi$ exists, we must have that for $j = 2, \dots, k$, there exist nonzero constants $a_j$ such that $\varphi(u_j + J + Soc(W)) = a_j u_{j-1} + J$. 

If $k > \kappa_{r,s}^{\pm}$, then letting $n = \kappa_{r,s}^\pm + 1$ (we are necessarily assuming that $\kappa_{r,s}^\pm < \infty$ here), then since $S_{r,s}(\pm 1)u_{n-1} + \sum_{i = 1}^{n-1} R_{r,s}^{\pm, i}u_{n-i} \in J(c,h,k)$, we have
\begin{eqnarray*}
 \varphi(S_{r,s}(\pm 1) u_n  + J + Soc(W)) &=& a_{n} S_{r,s}(\pm1)u_{n-1} + J \ = \ -a_n \sum_{i = 2}^{n-1} R_{r,s}^{\pm, i} u_{n-i} + J  \nonumber \\
&=& -a_n \sum_{i = 2}^{n-1} R_{r,s}^{\pm, i} 
 \frac{1}{a_{n-i+1}} \varphi(u_{n-i+1} + Soc(W)) \nonumber\\
 &=& \varphi ( -a_n \sum_{i = 2}^{n-1} R_{r,s}^{\pm, i} 
 \frac{1}{a_{n-i+1}} u_{n-i+1} + Soc(W)). 
 \end{eqnarray*}
 However, $S_{r,s}(\pm1)u_n + J  \notin  (\langle u_1, \dots, u_{n-1}\rangle  + J)/J \subset W$ for $n>\kappa_{r,s}^\pm$ and since $Soc(W) = (\langle u_1 \rangle + J)  / J$, this implies $S_{r,s}(\pm1)u_n  + J + Soc(W) \notin (\langle u_1, \dots, u_{n-1} \rangle + J) /Soc(W)$. Therefore no such $\varphi$ exists, and Eqn.\ (\ref{W/Soc=Rad}) does not hold.  

It remains to show that Eqn.\ (\ref{W/Soc=Rad}) does hold when $k \leq \kappa_{r,s}^\pm$. In this case, we claim that $\varphi(u_j  + J + Soc(W)) = u_{j-1} + J$ for $2\leq j \leq k$ uniquely determines a $V_{Vir}(c,0)$-module map that is an isomorphism between $W/Soc(W)$ and $Rad(W)$.  In particular, we define 
\begin{eqnarray*}
\varphi: W/Soc(W) &\longrightarrow & Rad(W)\\
R_\ell u_j  + J + Soc(W) & \mapsto & R_\ell u_{j-1} + J,
\end{eqnarray*}
for $R_\ell \in \mathcal{R}_\ell$, and $j = 2, \dots, k$, with $\varphi$ extended linearly.  Then  
$\varphi$ is well defined since if $R_\ell u_j  + J + Soc(W) = T_\ell u_j  + J + Soc(W)$, for some $R_\ell, T_\ell \in \mathcal{R}_\ell$, then $R_\ell u_j - T_\ell u_j + J \in Soc(W)  = (\langle u_1 \rangle + J)/J$, and thus $\varphi((R_\ell - T_\ell)u_j  + J + Soc(W)) =  \varphi(0 + J + Soc(W)) =  0 + J$,   implying $R_\ell u_{j-1} + J = T_\ell u_{j-1} + J$. 

By definition $\varphi$ is surjective.

To show that $\varphi$ is injective, we note that if $\varphi(\sum_{j = 2}^k R_{\ell,j} u_j + J + Soc(W)) = 0 + J$ for some $R_{\ell,j} \in \mathcal{R}_\ell$, then $\sum_{j=2}^k R_{\ell,j} u_{j-1} \in J$.  By Theorem \ref{J=Case1ii} this implies that $\sum_{j=2}^k R_{\ell,j} u_{j-1} \in \langle S_{r,s}(\pm1)u_n + \sum_{i = 1}^{n-1} R_{r,s}^{\pm, i}u_{n-i} \; | \; n = 1, \dots, k \rangle $. If $\ell < rs$, this implies $\sum_{j=2}^k R_{\ell,j} u_{j-1}  = 0$.  By linear independence this implies $R_{\ell,j} = 0$ for each $j = 2, \dots, k$, so that 
$\sum_{j = 2}^k R_{\ell,j} u_j + J + Soc(W) = 0 + J + Soc(W)$.  If $\ell \geq rs$, then letting $m$ be the largest $m \in \{2, \dots, k\}$ such that $R_{\ell, m} \neq 0$, we have that 
$\sum_{j=2}^k R_{\ell,j} u_{j-1} = \sum_{j = 2}^m R_{\ell ,j} u_{j-1} = T_{\ell-rs} (S_{r,s}(\pm1)u_{m-1} + \sum_{i=1}^{m-2}R_{r,s}^{\pm, i}u_{m-1-i})$ for some $T_{\ell - rs} \in \mathcal{R}_{\ell-rs}$.  But then this implies that 
\begin{eqnarray*}
\sum_{j = 2}^k R_{\ell,j} u_j + J + Soc(W) &=& T_{\ell-rs}(S_{r,s}(\pm1) u_m + \sum_{i=1}^{m-2}R_{r,s}^{\pm, i}u_{m-i}) + J + Soc(W)\\
&=& -T_{\ell - rs}R_{r,s}^{\pm,m-1}u_1 + J +  (\langle u_1\rangle +J )/J \\
&=& 0 + J +(\langle u_1\rangle +J )/J \\
&=& 0+ J + Soc(W),
\end{eqnarray*}
proving injectivity.
\end{proof}

Note that, in particular, Theorem \ref{interlocked-thm} implies that for some $C_1$-cofinite vertex operator algebras, there exist indecomposable reducible modules which do not have well-defined graded pseudo-traces as defined in this paper or in the setting of \cite{Miyamoto2004}.

Next we have the following Lemma which first assumes $W$ is an interlocked $V_{Vir}(c,0)$-module with certain properties.

\begin{lem}\label{interlocked-lemma}
Let $W$ be an interlocked $V_{Vir}(c,0)$-module induced from level 0 which contains a unique maximal submodule  and has an irreducible socle. Then, as $V_{Vir}(c,0)$-modules
$$W/Soc(W)\cong Rad(W).$$
\end{lem}
\begin{proof}
    Let $W_1=Soc(W)\subset W$. Then, there exists $W_2\subset W$ such that $W/Soc(W)\cong W_2$ and $W/W_2\cong Soc(W)$.  Since $Soc(W)$ is a simple module $W_2$ must be a maximal submodule of $W$. Hence,  $W_2 \cong Rad(W)$, and in particular, $W/Soc(W)\cong Rad(W)$.
\end{proof}

Using Theorem \ref{interlocked-thm} and Lemma \ref{interlocked-lemma}, we have
\begin{thm}\label{interlocked-thm2}
Let $W = W (c, h, k) =  \mathfrak{L}_0(U(c, h, k))$ be the   $V_{Vir}(c,0)$-module induced from the $A_0(V_{Vir}(c,0))$-module $U(c,h,k)\cong \mathbb{C}[x]/((x-h)^k)$ for $h\in \mathbb{C}$ and $k\in \mathbb{Z}_{> 0}$ with $k \geq 2$. Then, $W(c,h,k)$ is interlocked if and only if one of the following holds:

Case (0) holds, i.e., $T(c,h) = 0$;

Case (1)(ii) holds, (i.e., $t \ \pm1$ and $r \neq s$), and $k \leq  \kappa^\pm_{r,s}$, where $\kappa^\pm_{r,s}$ is defined as in Theorem \ref{degree-d-theorem}.  

Moreover, in these cases when $W$ is interlocked, then it is strongly interlocked.
\end{thm}

\begin{proof}
If $T(c,h) = 0$, then the proof follows from Corollary \ref{Case-0-cor}.

If Case (1)(i), Case (1)(ii) with $k > \kappa_{r,s}^\pm$, or Case (2) holds, then by  Propositions \ref{soc-rad-prop},  Theorem \ref{interlocked-thm}, and Lemma \ref{interlocked-lemma}, if $k>1$ then $W(c,h,k)$ is not interlocked. 

It remains to show that $W(c,h,k)$ is interlocked if Case (1)(ii) holds, i.e., $t = \pm 1$, $r\neq s$, and $k \leq \kappa_{r,s}^\pm$, and that $W(c,h,k)$ in this case is in fact strongly interlocked.  In this case, all the submodules of $W(c,h,k)$ are given by  $W^{(j)} \cong W(c,h,j)$ for $1 \leq j \leq k$, and for $j<k$, we have $W(c,h,j)  \cong  (\langle u_1, u_2, \dots, u_{j}\rangle  + J)  / J= (\langle u_1, u_2, \dots, u_{j}, S_{r,s}(\pm1)u_{j + 1}\rangle  + J)  / J$, since $S_{r,s}(\pm1)u_{j+1}$ is equivalent mod $J(c,h,k)$ to a linear combination of elements in $\mathcal{R}_{rs}$ acting on $u_1,\dots, u_{j}$. In fact, in this case  by Theorem \ref{J=Case1ii}
\[J(c,h,k) = 
 \langle S_{r,s}(\pm1)u_n +  \sum_{i = 1}^{n-1} R_{r,s}^{\pm, i}u_{n-i} \; | \; n = 1, \dots, k \rangle.\] 

For fixed $j \in \{ 1,\dots, k\}$, define 
\begin{eqnarray*}
\varphi_j: W/W^{(j)} &\longrightarrow & W^{(k-j)}\\
R_\ell u_{i+j} + J + W^{(j)} & \mapsto & R_\ell u_{i} + J,
\end{eqnarray*}
for $R_\ell \in \mathcal{R}_\ell$, and $i = 1, \dots, k-j$, with $\varphi_j$ extended linearly.  Then  
$\varphi_j$ is well defined since if $R_\ell u_{i+j} + J + W^{(j)} = T_\ell u_{i+j} \in J + W^{(j)}$, for some $R_\ell, T_\ell \in \mathcal{R}_\ell$, then $R_\ell u_{i+j} - T_\ell u_{i+j} + J  \in W^{(j)}$, and thus $\varphi_j((R_\ell - T_\ell)u_{i+j} + J + W^{(j)}) = \varphi_j(0 + J + W^{(j)}) = 0 + J$.

By definition $\varphi_j$ is surjective.  

To show that $\varphi_j$ is injective, we note that if 
$\varphi_j(\sum_{i = 1}^{k-j} R_{\ell,i} u_{i+j} + J + W^{(j)}) = 0 + J$ for some $R_{\ell,i} \in \mathcal{R}_\ell$, then $\sum_{i = 1}^{k-j} R_{\ell,i}u_i \in J$.  By Theorem \ref{J=Case1ii} this implies that $\sum_{i = 1}^{k-j} R_{\ell,i} u_{i}   \in \langle S_{r,s}(\pm1) u_n + \sum_{i = 1}^{n-1} R^{\pm, i}_{r,s} u_{n-i} \; | \; n = 1,\dots, k \rangle$.  If $\ell<rs$ this implies $\sum_{i = 1}^{k-j} R_{\ell,i} u_{i}$.  By linear independence, this implies $R_{\ell,i} = 0$ for each $i = 1,\dots, k-j$, so that $\sum_{i = 1}^{k-j} R_{\ell,i} u_{i+j} + J + W^{(j)} = 0 + J + W^{(j)}$.  If $\ell \geq rs$, then letting $m$ be the largest $m \in \{1,\dots, k-j\}$ such that $R_{m,i} \neq 0$, we have 
$\sum_{i = 1}^{k-j} R_{\ell,i} u_{i} = \sum_{i = 1}^{m} R_{\ell,i} u_{i}  = T_{\ell - rs}(S_{r,s}(\pm1)u_{m} + \sum_{i = 1}^{m-1} R_{r,s}^{\pm,i} u_{m-i})$ for some $T_{\ell-rs} \in \mathcal{R}_{\ell-rs}$.  But this implies that 
\begin{eqnarray*}
 \sum_{i = 1}^{k-j} R_{\ell,i} u_{i+j} + J + W^{(j)} &=& T_{\ell - rs}(S_{r,s}(\pm1)u_{m+j} + \sum_{i=1}^{m-1}R_{r,s}^{\pm, i} u_{m+j-i}) + J + W^{(j)}\\
 &=& -T_{\ell -rs}( \sum_{i = m}^{m+j-1} R_{r,s}^{\pm,i} u_{m + j - i} ) + J 
 + (\langle u_1, \dots, u_j\rangle + J)/J \\
 &=& 0 + J +  (\langle u_1, \dots, u_j\rangle + J)/J\\
 &=& 0 + J + W^{(j)},
\end{eqnarray*}
since $-T_{\ell -rs}( \sum_{i = m}^{m+j-1} R_{r,s}^{\pm,i} u_{m + j - i} ) \in (\langle u_1, \dots, u_j\rangle + J)/J$,
proving injectivity.  

Replacing $j$ with $k-j$ above, we also have that $W/W^{(k-j)} \cong W^{(j)}$ proving that $W = W(c,h,k)$ is interlocked for all $1\leq k \leq \kappa_{r,s}^\pm$.

Since  $W^{(j)}/W^{(j-1)}$ is generated by $(\langle u_j\rangle + J)/J$ for $1 \leq j \leq k \leq \kappa_{r,s}^\pm$, it is simple and isomrophic to the unique simple module with conformal weight $h$, i.e. $W^{(j)}/W^{(j-1)} \cong L(c,h) \cong W^{(1)}$ for all $j = 1,\dots, k$.  Therefore $W = W(c,h,k)$ is strongly interlocked.  
\end{proof}

By Proposition \ref{matrix-form-prop}, in the cases when $W(c,h,k)$ is interlocked, and thus strongly interlocked as proven above,  there exists a strongly interlocked family of bases in which any weight-preserving module map restricted to each graded component of  $W(c,h,k)$, is of the form of Eqn.\ (\ref{Our-decomp}), and thus it is possible for there to be well-defined graded pseudo-traces for $W(c,h,k)$.  
 
\begin{rem}\label{M-bar-remark}
We note here that following Remark \ref{Vir-M=Mbar}, the universal $V_{Vir}(c,0)$-module induced from the $A_0(V_{Vir}(c,0))$-module $U(c,h,k)$, as defined in Eqn.\ (\ref{define-M-bar}), satisfies $\overline{M}_0(U(c,h,k)) = M_0(c,h,k) = M(c,h,k)$.  Thus these are the interlocked modules induced from level zero for $(c,h)$ in  Case (0), but they are not interlocked in Cases (1) and (2).  For instance,  $Soc(\overline{M}_0(U(c,h,k))) = Soc(M(c,h)) = 0$ for $(c,h)$ in Case (1) or (2),  whereas, by definition $Rad(\overline{M}_0(U(c,h,k)) \neq \overline{M}_0(U(c,h,k)) = \overline{M}_0(U(c,h,k))/Soc(\overline{M}_0(U(c,h,k)))$.  
\end{rem}

\section{Graded pseudo-traces for interlocked $V_{Vir}$-modules induced from the level zero Zhu algebra}\label{graded-pseudo-trace-section}

In this section we first recall some of the graded-traces for Verma modules for $V_{Vir}(c,0)$.  Then we show that graded pseudo traces are well defined when $W(c,h,k)$ is strongly interlocked as classified in Theorem \ref{interlocked-thm2}, i.e., in Case (0) and Case (1)(ii), and
give some key graded pseudo-traces for these $V_{Vir}(c,0)$-modules. 

\subsection{Graded traces for Verma modules}

Recall from Eqn. (\ref{grade-M}) that 
\[
M(c,h)=\coprod_{\ell \in \mathbb{Z}_{\geq 0}} M(c,h)(\ell)
\]
with  $M(c,h)(\ell) = M(c,h)_{\ell + h}$ the $L_0$  eigenspace with eigenvalue $\ell + h$.
From Eqn. (\ref{p(d)}) we know that the vectors of the form 
\begin{align*}
L_{-j_1}\cdots L_{-j_s}\vac_{c,h}, \ \ \quad \mbox{for  $j_1\geq\cdots \geq j_s>0$, $s\geq 0$, and $j_1+\cdots +j_s=n$}
\end{align*}
constitute a $\mathbb{C}$-basis for $M(c,h)(\ell)$ and therefore, $\dim M(c,h)(\ell)=\mathfrak{p}(\ell)$, the number of partitions of $\ell$. 
It follows that the graded dimension of the Verma module $M(c,h)$ is given by
\begin{equation}\label{M(c,h)-trace}
Z_{M(c,h)}({\bf 1}_{c,0}, \tau) \ = \ q^{-c/24}\sum_{\ell \in \mathbb{Z}_{\geq 0}} \mathfrak{p}(\ell)q^\ell q^h \ = \ q^{(1-c)/24 +h}\eta(q)^{-1},
\end{equation}
where $\eta(q)$ is the Dedekind $\eta$-function given by Eqn.\ (\ref{eta}).

In fact, as shown in \cite{FF}, we have the following graded dimensions, also called graded traces, for $L(c,h) = M(c,h)/T(c,h)$:

Case (0):  If $M(c,h) = L(c,h)$, i.e., $M(c,h)$ has no singular vectors and $T(c,h) = 0$,  then the graded trace is given by Eqn. \eqref{M(c,h)-trace}.

Case (1):
If $T(c,h) = \langle S_{r,s}(t) {\bf 1}_{c,h} \rangle$, so that 
\[L(c,h) = M(c,h)/T(c,h) = M(c,h)/\langle S_{r,s}(t){\bf 1}_{c,h} \rangle,\]  
with $\langle S_{r,s}(t){\bf 1}_{c,h} \rangle\cong{M(c, h+d)}$, we have that  
\begin{eqnarray*}\label{L(c,h)-one-singular-trace}
Z_{L(c,h)}({\bf 1}_{c,0}, \tau)  &= &Z_{M(c,h)}({\bf 1}_{c,0}, \tau) - q^d Z_{M(c, h + d)}({\bf 1}_{c,0}, \tau) \nonumber \\
&= & q^{(1-c)/24 + h}(1-q^d) \eta(q)^{-1}. 
\end{eqnarray*}

Case (2):  If $T(c,h) = \langle S_{r,s}(t) {\bf 1}_{c,h}, S_{r',s'}(t) {\bf 1}_{c,h}\rangle$ for some $rs=d$ and $r's' = d'$ with $d<d'$, then $L(c,h) = M(c,h)/\langle S_{r,s}(t){\bf 1}_{c,h}, S_{r',s'}(t){\bf 1}_{c,h}\rangle$  and $\langle S_{r,s}(t){\bf 1}_{c,h}\rangle \cong M(c,h+d)$, $\langle S_{r',s'}(t){\bf 1}_{c,h}\rangle\cong M(c, h+d')$ with $\langle S_{r,s}(t){\bf 1}_{c,h}, S_{r',s'}(t){\bf 1}_{c,h}\rangle\cong M(c,h+d)\oplus M(c,h+d')$. In this case, the graded traces are more complicated to express, and we refer the reader to pp.\ 485 of \cite{FF}. In fact, in the next section, we will show that the indecomposable reducible $V_{Vir}(c,0)$-modules $W(c,h,k)$ induced from the level zero Zhu algebra are not interlocked and therefore there are no graded pseudo-traces for this case, i.e., Case (2).

\subsection{Graded pseudo-traces for Case (0): $M(c,h) = L(c,h)$}

We are now ready to compute the graded pseudo-traces in Case (0), and in this setting either by Theorem \ref{extra-conditions-theorem} or by Corollary \ref{Case-0-cor} all the  $V_{Vir}$-modules $W(c,h,k)$ induced from the level zero Zhu algebra are strongly interlocked and  by Theorem \ref{extra-conditions-theorem} have well-defined graded pseudo-traces.

Thus fix $W = W(c,h, k) = \mathfrak{L}_0(U(c,h,k))$, such that  $M(c,h)$ is irreducible, i.e., the maximal submodule  $T(c,h) = 0$.  As with the Heisenberg vertex operator algebra, for convenience we organize our calculation of the graded pseudo-traces by degree using the fact that the weight spaces of $W(c,h,k)$ are related to the degree spaces by a shift, namely $W(c,h,k)_{\ell + h} = W(c,h,k)(\ell)$ for $\ell \in \mathbb{Z}_{\geq 0}$.

In this case, since $W(c,h,k) = M(c,h,k)$ we have that $\dim W(\ell) = k \mathfrak{p}$ and expressing $L_0|_{W(\ell)}$ in a  strongly interlocked  basis gives the decomposition into a $k\mathfrak{p}(\ell) \times k\mathfrak{p}(\ell)$ matrix of the form of Eqn.\ (\ref{Miyamoto-decomp}) with $A$ a $\mathfrak{p}(\ell)\times \mathfrak{p}(\ell)$ matrix, and thus, since $W$ is interlocked, the matrix $B$ must be $\mathfrak{p}(\ell) \times \mathfrak{p}(\ell)$. 

For convenience, we let $S_\ell = L^{S}_0|_{W(\ell)}$ and
$N_\ell = L^{N}_0|_{W(\ell)}$  denote the semisimple and nilpotent parts of $L_0|_{W(\ell)}$, respectively, so that from Eqn.\ (\ref{Jordan-block}) we have that, for instance, $S_0 = hI_k$ and $N_0 = D_{k,1}$.  More generally, $L_0|_{W(\ell)}$ is a $k\mathfrak{p}(\ell) \times k\mathfrak{p}(\ell)$ matrix with $S_\ell = (h + \ell)I_{k\mathfrak{p}(\ell)}$ and $N_\ell = D_{k\mathfrak{p}(\ell), \mathfrak{p}(\ell)}$.  

Then analyzing 
$q^{N_\ell}$, we have 
\[q^{N_\ell} =  \sum_{j \in \mathbb{Z}_{\geq 0}} \frac{1}{j!}(N_\ell)^j (\log q)^j = \sum_{j \in \mathbb{Z}_{\geq 0}} \frac{1}{j!} (D_{k\mathfrak{p}(\ell),\mathfrak{p}(\ell)})^j (\log q)^j = \sum_{j \in \mathbb{Z}_{\geq 0}} \frac{1}{j!}D_{k\mathfrak{p}(\ell),j\mathfrak{p}(\ell)} (\log q)^j. \]
Then the $B$ matrix for $q^{N_\ell}$ is the $\mathfrak{p}(\ell) \times \mathfrak{p}(\ell)$ matrix in the upper right corner, and thus consists of any terms involving $D_{k\mathfrak{p}(\ell), i}$ for $(k-1)\mathfrak{p}(\ell) \leq i \leq k\mathfrak{p}(\ell) - 1$.  There is just one such term, namely $\frac{1}{(k-1)!} D_{k \mathfrak{p}(\ell), (k-1)\mathfrak{p}(\ell)} (\log q)^{k-1}$, and thus $B$ is the diagonal matrix $B = \frac{1}{(k-1)!} I_{ \mathfrak{p}(\ell)} (\log q)^{k-1}$ and 
\[\mathrm{pstr}\, q^{N_\ell} = \mathrm{tr}(B) = \frac{1}{(k-1)!} (\log q)^{k-1} \mathfrak{p}(\ell).  \]
Thus 
\begin{eqnarray*}
\lefteqn{\mathrm{pstr}_{W(c,h,k)}( {\bf 1}_{c,0}, \tau) \ = \   \sum_{\ell \in \mathbb{Z}_{\geq 0}} \mathrm{pstr} (q^{N_\ell})  q^{S_\ell - c/24} \ 
= \  q^{h  - c/24}  \sum_{\ell \in \mathbb{Z}_{\geq 0}}\mathrm{pstr} ( q^{N_\ell} )  q^\ell } \\
&=& q^{h  - c/24}  \sum_{\ell \in \mathbb{Z}_{\geq 0}} \frac{1}{(k-1)!} (\log q)^{k-1} \mathfrak{p}(\ell) q^\ell  \ = \  q^{h  - c/24} \frac{1}{(k-1)!} (\log q)^{k-1} \sum_{\ell \in \mathbb{Z}_{\geq 0}} \mathfrak{p}(\ell) q^\ell\\
&=& q^{(1-c)/24 + h } \frac{1}{(k-1)!} (\log q)^{k-1} \eta(q)^{-1} \ = \   \frac{1}{(k-1)!} (\log q)^{k-1} Z_{M(c,h)}( {\bf 1}_{c,0}, \tau) \\
&=& \frac{1}{(k-1)!} (\log q)^{k-1} Z_{L(c,h)}( {\bf 1}_{c,0}, \tau).
\end{eqnarray*} 

This, along with the logarithmic derivative property, gives
\begin{thm}
If $M(c,h) = L(c,h)$ and $W(c,h,k) = \mathfrak{L}_0(U(c,h,k))$ is the $V_{Vir}(c,0)$-module induced from the $A_0(V_{Vir}(c,0))$-module $U(c,h,k) = \mathbb{C}[x]/((x - h)^k)$, then 
\begin{eqnarray*}
\mathrm{pstr}_{W(c,h,k)}({\bf 1}_{c,0}, \tau) &=& q^{(1-c)/24 + h } \frac{1}{(k-1)!} (\log q)^{k-1} \eta(q)^{-1} \\
&=&   \frac{1}{(k-1)!} (\log q)^{k-1} Z_{M(c,h)}({\bf 1}_{c,0},\tau),\nonumber 
\end{eqnarray*}
and this graded pseudo-trace satisfies the logarithmic derivative property
\begin{eqnarray*}
\mathrm{pstr}_{W(c,h,k)} (\omega, \tau) &=& q^{-c/24} \frac{1}{2\pi i} \frac{d}{d\tau} q^{c/24} \mathrm{pstr}_{W(c,h,k)} ({\bf 1}_{c,0}, \tau)  \\
&=& q^{-c/24} q\frac{d}{dq} q^{c/24} \mathrm{pstr}_{W(c,h,k)} ({\bf 1}_{c,0}, \tau),
\end{eqnarray*}
where 
\begin{eqnarray*}
\lefteqn{\mathrm{pstr}_{W(c,h,k)} (\omega, \tau) }\\
&=& q^{h - c/24}  \sum_{\ell \in \mathbb{Z}_{\geq 0}} \mathfrak{p}(\ell) q^\ell \Big( \frac{(\ell + h)}{(k-1)!} (\log q)^{k-1}  + \frac{1}{(k-2)! }  (\log q)^{k-2}  \Big)\\
&=& q^{h - c/24}  \frac{1}{(k-2)!} (\log q)^{k-2} \Big( \frac{\log q}{k-1} \Big(\sum_{\ell \in \mathbb{Z}_{\geq 0}} (\ell + h)\mathfrak{p}(\ell) q^\ell \Big)+ q^{1/24} \eta(q)^{-1} \Big).
\end{eqnarray*}
\end{thm}

\subsection{Graded pseudo-traces for Case (1)(ii) for  $k \leq \kappa^{\pm}_{r,s}$. }\label{Case(1)(ii)-subsection}

 For the case when $t = \pm 1$, i.e., $c = 1$ or $25$, and $h = h_{r,s}$ is such that $r \neq s$,  by Theorem \ref{interlocked-thm2} if $k \leq \kappa_{r,s}^\pm$, the module $W(c,h,k)$ is strongly interlocked with $Soc(W(c,h,k)) = W^{(1)}(c,h,k) = L(c,h,k)$ the unique irreducible $V_{Vir}$-module with conformal weight $h = h_{r,s}$ and degree zero space $\mathbf{1}_{c,h}$. Thus by Theorem \ref{Virasoro-exceptional-theorem}, $W(c,h,k)$ has well-defined graded pseudo-traces.   Thus fix such a $k\leq \kappa_{r,s}^\pm$.

In this case, by Theorem \ref{J=Case1ii}, $\dim J(c,h,k)(\ell) = k\mathfrak{p}_{rs}(\ell)$, where $\mathfrak{p}_{rs}(\ell)$ is the number of partitions of $\ell$ that contain at least one $d=rs$, or equivalently, contain at least $d$ ones.  Therefore $\dim W(c,h,k)(\ell) = k(\mathfrak{p}(\ell) - \mathfrak{p}_{rs}(\ell))$, and $\dim Soc(W) (\ell) = (\mathfrak{p}(\ell) - \mathfrak{p}_{rs}(\ell)) = \dim (W/Rad(W))(\ell)$.

Again, we let $S_\ell = L^{S}_0|_{W(\ell)}$ and
$N_\ell = L^{N}_0|_{W(\ell)}$  denote the semisimple and nilpotent parts of $L_0|_{W(\ell)}$, respectively, so that from Eqn.\ (\ref{Jordan-block}) we have that, for instance, $S_0 = hI_k$ and $N_0 = D_{k,1}$. More generally, since $W$ is strongly interlocked, there exists a strongly interlocked family of basis such that for each $\ell \in \mathbb{N}$, the matrix representation of $L_0|_{W(\ell)}$ with respect to the basis at degree $\ell$ is a $k(\mathfrak{p}(\ell) - \mathfrak{p}_{rs}(\ell)) \times k(\mathfrak{p}(\ell) - \mathfrak{p}_{rs}(\ell))$ matrix with $S_\ell = (h + \ell)I_{k(\mathfrak{p} (\ell) - \mathfrak{p}_{rs}(\ell))}$ and $N_\ell = D_{k(\mathfrak{p}(\ell) - \mathfrak{p}_{rs}(\ell)), \mathfrak{p}(\ell) - \mathfrak{p}_{rs}(\ell)}$.  

Then analyzing 
$q^{N_\ell}$, we have 
\begin{eqnarray*}
q^{N_\ell} &=&  \sum_{j \in \mathbb{Z}_{\geq 0}} \frac{1}{j!}(N_\ell)^j (\log q)^j \ = \ \sum_{j \in \mathbb{Z}_{\geq 0}} \frac{1}{j!} (D_{k(\mathfrak{p}(\ell) - \mathfrak{p}_{rs}(\ell)),\mathfrak{p}(\ell) - \mathfrak{p}_{rs}(\ell)})^j (\log q)^j \\
&=& \sum_{j \in \mathbb{Z}_{\geq 0}} \frac{1}{j!}D_{k(\mathfrak{p}(\ell) - \mathfrak{p}_{rs}(\ell)),j(\mathfrak{p}(\ell) - \mathfrak{p}_{rs}(\ell))} (\log q)^j. 
\end{eqnarray*}
Then the $B$ matrix for $q^{N_\ell}$ is the $(\mathfrak{p}(\ell)- \mathfrak{p}_{rs}(\ell)) \times (\mathfrak{p}(\ell) - \mathfrak{p}_{rs}(\ell))$ matrix in the upper right corner, and thus consists of any terms involving 
$D_{k(\mathfrak{p}(\ell) - \mathfrak{p}_{rs}(\ell)), i}$ for $(k-1)(\mathfrak{p}(\ell)- \mathfrak{p}_{rs}(\ell))\leq i \leq k(\mathfrak{p}(\ell) - \mathfrak{p}_{rs}(\ell)) - 1$. There is just one such term, namely $\frac{1}{(k-1)!} D_{k (\mathfrak{p}(\ell) - \mathfrak{p}_{rs}(\ell)), (k-1)(\mathfrak{p}(\ell) - \mathfrak{p}_{rs}(\ell))} (\log q)^{k-1}$, and thus $B$ is the diagonal matrix $B = \frac{1}{(k-1)!} I_{ \mathfrak{p}(\ell) - \mathfrak{p}_{rs}(\ell)} (\log q)^{k-1}$ and 
\[\mathrm{pstr}\, q^{N_\ell} = \mathrm{tr}(B) = \frac{1}{(k-1)!} (\log q)^{k-1} (\mathfrak{p}(\ell) - \mathfrak{p}_{rs}(\ell)).  \]

Thus 
\begin{eqnarray*}
\mathrm{pstr}_{W(c,h,k)}({\bf 1}_{c,0}, \tau)  &=&   \sum_{\ell \in \mathbb{Z}_{\geq 0}} \mathrm{pstr} (q^{N_\ell})  q^{S_\ell - c/24} \ 
= \  q^{h  - c/24}  \sum_{\ell \in \mathbb{Z}_{\geq 0}}\mathrm{pstr} ( q^{N_\ell} )  q^\ell  \\
&=& q^{h  - c/24}  \sum_{\ell \in \mathbb{Z}_{\geq 0}} \frac{1}{(k-1)!} (\log q)^{k-1} (\mathfrak{p}(\ell) - \mathfrak{p}_{rs}(\ell)) q^\ell  \\
&=&  q^{h  - c/24} \frac{1}{(k-1)!} (\log q)^{k-1} \sum_{\ell \in \mathbb{Z}_{\geq 0}} (\mathfrak{p}(\ell) - \mathfrak{p}_{rs}(\ell)) q^\ell\\
&=& q^{(1-c)/24 + h  } \frac{1}{(k-1)!} (\log q)^{k-1} (1-q^{rs}) \eta(q)^{-1} \\
&=& \frac{1}{(k-1)!} (\log q)^{k-1} Z_{L(c,h)}( {\bf 1}_{c,0}, \tau).
\end{eqnarray*} 

Therefore we have the following  Corollary to Theorems \ref{interlocked-thm}, \ref{Virasoro-exceptional-theorem}, and \ref{log-thm}.
\begin{cor}  
Let $t = \pm 1$, so that $c = 25$ or $1$, respectively, let $h = h_{r,s}$ such that $r \neq s$,  and $k \leq \kappa_{r,s}^\pm$, where $\kappa_{r,s}^\pm$ is as defined in Theorem \ref{degree-d-theorem}.   Let $W(c,h,k) = \mathfrak{L}_0(U(c,h,k))$ be the $V_{Vir}(c,0)$-module induced from the $A_0(V_{Vir}(c,0))$-module $U(c,h,k) = \mathbb{C}[x]/((x-h)^k)$.  Then $W(c,h,k)$ is strongly interlocked and  has well-defined graded pseudo-traces.  For instance
\begin{eqnarray*}
\mathrm{pstr}_{W(c,h,k)}({\bf 1}_{c,0}, \tau) &=&  \sum_{\ell \in \mathbb{Z}_{\geq 0}} \mathrm{pstr} (L_0|_{W(\ell)}  \, q^{N_\ell})  q^{S_\ell - c/24} \\ 
&=& q^{(1-c)/24 + h}(1-q^{rs})\frac{1}{(k-1)!} (\log q)^{k-1} \eta(q)^{-1}\\
&=& \frac{1}{(k-1)!} (\log q)^{k-1} Z_{L(c,h)}({\bf 1}_{c,0}, \tau).
\end{eqnarray*}

Furthermore, this graded pseudo-trace satisfies the logarithmic derivative property.
\end{cor}

\section{Conclusions and Future Work}

We have defined the notion of strongly interlocked indecomposable generalized module for a vertex operator algebra $V$, and shown that the notion of graded pseudo-trace for such modules when well defined under changes of strongly interlocked bases  gives a symmetric linear operator that satisfies the logarithmic derivative property.

We have given examples of two settings in which strongly interlocked modules have well-defined graded pseudo-traces and applied these results to the Hiesenberg and Virasor vertex operator algebras. We have given a complete classification of  strongly interlocked indecomposable modules for all modules of the Heisenberg vertex operator algebras and for all modules induced from level zero for the Virasoro algebra.  In particular, we have shown that all of the indecomposable modules for the Heisenberg vertex operator algebras are strongly interlocked, have well-defined graded pseudo-traces, and we have calculated some of the graded pseudo-traces. 

For the universal Virasoro vertex operator algebras, we have shown the following:

In Case (0), all the modules $W(c,h,k)$ induced from the level zero Zhu algebra are strongly interlocked and have well-defined graded pseudo-traces. We have calculated some key graded pseudo-traces.

In Case (1)(i) and Case (2) with $k>1$, and in Case (1)(ii) with $k > \kappa_{r,s}^\pm$,  we have shown that the modules $W(c,h,k)$ induced from the level zero Zhu algebra are not interlocked and thus graded pseudo-traces as defined in this paper are not well defined.

In Case (1)(ii) when $k \leq \kappa_{r,s}^{\pm}$, we have shown the modules $W(c,h,k)$ induced from the level zero Zhu algebra are strongly interlocked and have well-defined graded pseudo-traces, and we calculated some of their graded pseudo-traces.

In future work, we further study the properties of  graded pseudo-traces  for interlocked modules for vertex operator algebras as well as other trace-like $q$-series \cite{BOHY}, and calculate more general graded pseudo-traces for the Heisenberg and Virasoro vertex operator algebras. We also plan to 
systematically study and classify $V_{Vir}(c,0)$-modules induced by higher level Zhu algebras, such as the $V_{Vir}(c,0)$-modules induced from the level one Zhu algebra that are not induced at level zero discussed in \cite{BVY, BVY-Virasoro}.  This includes further use and development of techniques such as those we have introduced and applied in this paper for studying indecomposable $V_{Vir}(c,0)$-modules. We plan to use this work to systematically study graded pseudo-traces for general $C_1$-cofinite vertex operator algebras and the categorical structures that arise from the corresponding classes of modules.


\end{document}